\documentclass[11pt,letterpaper]{amsart}
\usepackage{amssymb}
\usepackage[dvips]{epsfig}
\usepackage{graphicx}
\usepackage{color}
\usepackage{amsmath}
\usepackage{amssymb}
\usepackage{amsfonts}
\usepackage{geometry}
\usepackage{amsthm}
\usepackage{cite}
\usepackage{tikz}
\newcommand{\pa}{\partial}
\newcommand{\Z}{\mathbb{Z}}
\newcommand{\C}{\mathbb{C}}
\newcommand{\N}{\mathbb{N}}
\newcommand{\R}{\mathbb{R}}
\newcommand{\T}{\mathbb{T}}

\newcommand{\si}{\sigma}

\newcommand{\ct}{\theta}

\newcommand{\bfa}{\mathbf{\alpha}}

\newcommand{\La}{\Lambda}
\newcommand{\meas}{{\rm meas}}
\newcommand{\bfk}{\mathbf{k}}
\newcommand{\yo}{\iota}
\newcommand{\vp}{\varPsi}

\newcommand{\LakL}{\Lambda_{k\nearrow L}}
\newcommand{\tLakL}{\tilde{\Lambda}_{k\nearrow L}}
\newcommand{\hLakL}{\hat{\Lambda}_{k\nearrow L}}

\newcommand{\nm}{|\!|}
\newcommand{\nmt}{|\!|_{\mathbb{T}}}

\newcommand{\mcA}{\mathcal{A}}
\newcommand{\mcB}{\mathcal{B}}

\newcommand{\mcE}{\mathcal{E}}
\newcommand{\mcF}{\mathcal{F}}

\newcommand{\mcP}{\mathcal{P}}
\newcommand{\mcS}{\mathcal{S}}
\newcommand{\mcI}{\mathcal{I}}
\newcommand{\mcR}{\mathcal{R}}
\newcommand{\mcX}{\mathcal{X}}

\newcommand{\mcO}{\mathcal{O}}

\newcommand{\mcW}{\mathcal{W}}
\newcommand{\LC}{\mathcal{LC}}
\newcommand{\DC}{{\rm DC}_{\gamma,\tau}}
\newcommand{\mcN}{\mathcal{N}}
\newcommand{\tH}{\tilde{H}}
\newcommand{\tN}{\tilde{N}}
\newcommand{\Op}{{\rm Op}}
\newcommand{\amplitude}{\mathbf{a}}

\newcommand{\nomega}{n\cdot\omega}
\newcommand{\jalpha}{j\cdot\alpha}

\renewcommand{\Im}{{\rm Im}}
\renewcommand{\Re}{{\rm Re}}
\newcommand{\dist}{{\rm dist}}
\newcommand{\diam}{{\rm diam}}
\newcommand{\poly}{{\rm Poly}}
\newcommand{\supp}{{\rm supp}}

\newcommand{\disk}{{\rm Disk}}

\newtheorem{thm}{Theorem}[section]
\newtheorem{lem}[thm]{Lemma}
\newtheorem{prop}[thm]{Proposition}

\newtheorem{Def}{Definition}[section]
\newtheorem{assumption}[thm]{Assumption}
\newtheorem{rmk}{\bf Remark}[section]
\newtheorem{state}{Statement}[section]
\theoremstyle{definition}

\newtheorem{ex}[thm]{Example}

\newtheorem{claim}[thm]{\bf Claim}

\numberwithin{equation}{section}
\usepackage[colorlinks=true,pdfstartview=FitV,linkcolor=magenta,citecolor=cyan]{hyperref}
\usepackage{bm}
\keywords{Nonlinear Maryland model, Anderson localization, Multi-scale analysis, Craig-Wayne-Bourgain method,  Green's function estimates}
\begin{document}
 \title[Nonlinear Maryland Model]{Anderson Localized States for the nonlinear Maryland model on $\mathbb{Z}^d$}

\author[Liu]{Shihe Liu}
\address[S. Liu] {School of Mathematical Sciences,
Peking University,
Beijing 100871,
China}
\email{2301110021@stu.pku.edu.cn}

\author[Shi]{Yunfeng Shi}
\address[Y. Shi] {School of Mathematics,
Sichuan University,
Chengdu 610064,
China}
\email{yunfengshi@scu.edu.cn}

\author[Zhang]{Zhifei Zhang}
\address[Z. Zhang] {School of Mathematical Sciences,
Peking University,
Beijing 100871,
China}
\email{zfzhang@math.pku.edu.cn}

\begin{abstract}
In this paper, we  investigate  Anderson  localization    for a  nonlinear perturbation of  the  Maryland model $H=\varepsilon\Delta+\cot\pi(\theta+j\cdot\alpha)\delta_{j,j'}$ on $\Z ^d$. Specifically, if $\varepsilon,\delta$ are sufficiently small, we construct a large number of time quasi-periodic and space exponentially decaying solutions (i.e., Anderson localized states) for the equation  $i\frac{\partial u}{\partial t}=Hu+\delta|u|^{2p}u$  with a Diophantine $\alpha$. Our proof combines eigenvalue estimates of the Maryland model with the Craig-Wayne-Bourgain method, which originates from KAM theory for Hamiltonian PDEs.
\end{abstract}

\maketitle

\maketitle
\tableofcontents
\section{Introduction}

\subsection{The model and main result}
Let 
\[\alpha=(\alpha_1,\alpha_2,\cdots,\alpha_d)\in [0,1]^d,\ \theta\in\T:= [0,1]\]
and consider the Maryland model  (or the quasi-periodic Schr\"odinger operator)
\[H=\varepsilon\Delta+\cot\pi(\theta+\alpha\cdot j)\delta_{j,j'},\ j\in\Z^d, \]
 where  $|\cdot|_1$ is the $\ell^1$-norm,  $\Delta(j,j')=\delta_{|j-j'|_1,1}$ is the Laplacian on $\Z^d$ and $\alpha\cdot j=\sum\limits_{k=1}^d\alpha_kj_k$.  Assume further that $\alpha$ is Diophantine, i.e., for some $\gamma>0,\tau>d$, 
\begin{equation}\label{Diophantine condition}
\bfa\in {\rm DC}_{\gamma,\tau}:=\{\alpha\in [0,1]^d:\ \nm j\cdot \alpha\nmt=\inf_{m\in \Z}|m-j\cdot\alpha|\geq\frac{\gamma}{|j|_1^{\tau}}\ {\rm for}\ {\forall}j\in \Z^d\backslash\{0\}\biggl\}. 
\end{equation}
It is well-known that $H$ exhibits Anderson localization (i.e., pure point spectrum with exponentially decaying eigenfunctions) for almost every $\theta\in\T$ \cite{BLS83, FP84} under the assumption that $\alpha\in {\rm DC}_{\gamma, \tau}$. In addition, if $0< \varepsilon\leq  \varepsilon_0(d,\tau,\gamma)$,  $H$ can be diagonalized via a unitary operator, which is close to the identity operator.  

Assume now  that $H$ exhibits Anderson localization.  Let $\{\phi_j\}_{j\in\Z^d}$ be the complete set of eigenfunctions of $H$. (These eigenfunctions depend on the parameters $\varepsilon,\alpha,\theta$,  although we suppress this dependence.) 
Define  $l_j$  to be  the point where
\[|\phi(l_j)|=\max_{x\in\Z^d}|\phi_j(x)|.\]
We refer to $l_j$ as the localization center of $\phi_j$. In the regime of $0\leq \varepsilon\leq  \varepsilon_0(d,\tau,\gamma)$, we have $l_j=j$ \cite{BLS83}.  Denote by $\mu_j$ the eigenvalue corresponding to $\phi_j$. We know that $\{\mu_j\}_{j\in\Z^d}$ is simple \cite{BLS83}. We consider the linear Schr\"odinger equation 
\begin{equation}\label{LS}
	i\frac{\partial u}{\partial t}=Hu
\end{equation}
and all the solutions (with initial data $u_0(0,x)=\sum\limits_{j\in\Z^d}c_j\phi_j(x)\in \ell^2(\Z^d)$) to $(\ref{LS})$ are of the form 
\[u_0(t,x)=\sum_{j\in\Z^d}c_j e^{-i\mu_j t}\phi_j(x).\]
 If $c_j\neq 0$ for  finitely many $j$, then $u_0(t, x)$ is  quasi-periodic in $t$ and has space exponential  decay  in $x$ (i.e., it is an Anderson localized state). This paper is concerned with the persistency of the Anderson localized states for the nonlinear Maryland model  or the nonlinear Schr\"odinger equation (NLS)
   \begin{equation}\label{NLS}
i\frac{\partial u}{\partial t}=Hu+\delta|u|^{2p}u,\ p\in \mathbb{N}.
\end{equation}

Denote by  $|\cdot|$ the $\ell^{\infty}$-norm on $\Z^d$ and ${\rm meas(\cdot)}$ the Lebesgue measure. Our main result is stated as follows. 

\begin{thm}\label{Main Thm}
Let  $\alpha\in{\rm DC}_{\gamma,\tau}$ and  fix any $b$ distinct lattice sites $\beta_k\in \Z^d,\ k=1,2,\cdots,b$ with $\sup_{k}|\beta_k|\leq B$. Define
\begin{equation}\label{solu to LS}
	u_0(t,x)=\sum_{k=1}^{b}a_k e^{-i\mu_{\beta_k} t}\phi_{\beta_k}(x)
\end{equation}
with $\mathbf{a}=(a_1,\cdots, a_k)\in [1,2]^b$. Then for $0<\delta\leq \tilde{\delta}(\gamma,\tau,b,d,B,p)\ll 1$,  there exist  a constant $\tilde{\varepsilon}=\tilde{\varepsilon}(\delta)>0$ and a set $\Theta=\Theta_{\alpha,\delta}\subset [0,1]$ in $\theta$ satisfying 
$\lim_{\delta\rightarrow 0+}\tilde{\varepsilon}(\delta)=0$ and ${\rm meas}([0,1]\backslash \Theta) \leq |\log \delta|^{-1}$
such that, for any $0<\varepsilon<\tilde{\varepsilon}$ and $\theta\in \Theta$, there exists   a set $\mathbf{\mathcal{R}}=\mathbf{\mathcal{R}}_{\bfa,\delta,\varepsilon,\theta}\subset[1,2]^b$ (in $\mathbf{a}$)  with 
\[{\rm meas}([1,2]^b\backslash \mathbf{\mathcal{R}})\leq e^{-|\log\delta|^{\frac{1}{2}}},\]
such that,  if $\mathbf{a}\in \mathbf{\mathcal{R}}$, then  \eqref{NLS}  has a solution $u(t,x)$  satisfying 
\begin{align*}
u(t,x)&=\sum_{(n,j)\in \Z^b\times\Z^d}\hat{u}(n,j)e^{in\cdot\omega t}\phi_{j}(x), \\
\hat{u}(-e_k,\beta_k)&=a_k\ {\rm for}\ k=1,\cdots,b, \\
\sum_{(n,j)\notin \mathcal{S}_{0}}|\hat{u}(n,j)|e^{\rho(|n|+|j|)}&\leq \sqrt{\varepsilon+\delta}\ {\rm for\ some}\ \rho>0, \\
\omega=(\omega_1,\omega_2,\cdots,\omega_b)&=(\mu_{\beta_1},\mu_{\beta_2},\cdots,\mu_{\beta_k}) +O(\delta),
\end{align*}
where $\mathcal{S}_{0}=\{(-e_k,\beta_k)\}^b_{k=1}$ is the resonant set and $\{e_k\}_{k=1}^b$ is the standard basis for $\Z^b$.
\end{thm}

\begin{rmk}
\begin{itemize}

\item[(1)]  To the best of our knowledge, this is the first result on  Anderson localization for nonlinear Maryland model on $\Z^d$ for  $d>1.$  

\item[(2)]  Shi-Wang \cite{SW24} first proved the existence of Anderson localization for the quasi-periodic NLS on $\Z^d$, where they considered multi-variable trigonometric polynomial potentials. In \cite{SW24}, the desired solutions were constructed perturbatively, starting from decoupled linear equations.
However, in the present work, we first diagonalize the (coupled) Maryland model and treat $\varepsilon\in[0,1]$ as a parameter rather than a perturbative parameter. This approach is more aligned with the spirit of Liu-Wang \cite{LW24}, who dealt with the nonlinear Anderson model. In the quasi-periodic case, due to the lack of parameters, the competition between $\varepsilon$ and $\delta$ happens  and  $\tilde{\varepsilon}(\delta)$  can be  $\delta^{\frac{1}{8}+}$. 

\item[(3)]  It would be very interesting if one could prove the existence of {\bf almost-periodic} solutions for the nonlinear Maryland model. 

\end{itemize}
\end{rmk}

\subsection{Related results}
The Maryland model was first introduced in \cite{FGP82, GFP82} and plays an important role in the study of quantum chaos. Indeed, one of the main models in quantum chaos is the so-called quantum kicked rotor, which can be regarded as the quantum analogue of the usual Chirikov standard map. Typically, the Chirikov map exhibits chaotic behaviors within certain ranges of parameters, while the motion in the quantum kicked rotor is generally almost-periodic for all couplings. This remarkable quantum suppression of chaos was well understood after the work of Fishman et al. in \cite{FGP82}. They first mapped the quantum kicked rotor into the Maryland model, and the almost-periodic motion of the rotor can be explained via the Anderson localization of the Maryland model (cf. \cite{Bou02} for an alternative approach). Since then, the study of Anderson localization in Maryland-type models has attracted great attention in the field of almost-periodic Schrödinger operators. The monotonicity property of Maryland-type potentials results in the absence of resonances, and the perturbative KAM diagonalization method can be employed to prove Anderson localization \cite{Cra83, BLS83, Pos83}. More importantly, it turns out that the exact Maryland model is solvable (cf. \cite{GFP82}), and Anderson localization holds true for all couplings \cite{FP84, Sim85}. Recently, there have been several new developments in the Anderson localization for Maryland-type models (including some quasi-periodic operators with more general Lipschitz monotone potentials) on $\Z^d$ (cf., e.g., \cite{JL17, JK19, Kac19, JY21, HJY22, KZ24, JK24, KPS22, Shi23, CSZ24, KPS24} for a few).   

Contrary to the linear case, very little is known about the Anderson localization for the nonlinear Maryland model. Typically, the Schrödinger operator with random potentials or quasi-periodic ones describes the motion of a single electron in disordered media; it completely ignores the interactions between different electrons. Considering electron-electron interactions and using the Hartree approximation \cite{AF88, AFS88} leads to the study of the nonlinear Schrödinger equations (NLS) with ergodic potentials. In this context, Albanese et al. \cite{AF88} first proved the existence of periodic solutions for NLS on $\Z^d$ with random potentials via some novel bifurcation techniques. The main difficulty encountered in such attempts is that the spectrum of the unperturbed system is dense in some interval. If one tries to extend \cite{AF88, AFS88} to construct quasi-periodic solutions, it becomes significantly more challenging: This problem was resolved by Bourgain and Wang \cite{BW08}. They proved the existence of time quasi-periodic and space exponentially decaying solutions for random NLS on $\Z^d$ in the regime of both high disorder and weak nonlinearity. The proof in \cite{BW08} combines the celebrated Craig-Wayne-Bourgain (CWB) method, initially developed in the KAM theory for Hamiltonian PDEs \cite{CW93, Bou98, Bou05}, with the spectral analysis of the Anderson model. Of particular importance is the use of semi-algebraic geometry arguments and matrix-valued Cartan's estimates in the study of Anderson localization for quasi-periodic Schrödinger operators on $\Z^2$ \cite{BGS02}. An alternative method, based on the Hamiltonian normal form and KAM iterations, for dealing with nonlinear Anderson localization for nonlinear difference equations without the Laplacian, was developed by Fröhlich, Spencer, and Wayne \cite{FSW86} and Yuan \cite{Yua02}. Relying on KAM-type methods, Geng et al. \cite{GYZ14} and Geng-Zhao \cite{GZ13} proved the existence of quasi-periodic solutions for quasi-periodic NLS on $\Z$ with certain Gevrey regular and Maryland potentials, respectively. Very recently, Liu and Wang \cite{LW24} removed the high-disorder condition of \cite{BW08} in the one-dimensional case via the CWB technique and were also able to handle multi-dimensional random NLS in the regime of linear Anderson localization. Compared with \cite{BW08}, \cite{LW24} first diagonalized the linear Anderson model and sought the nonlinear Anderson localized states starting from those of the Anderson model. For multi-dimensional NLS with quasi-periodic multi-variable trigonometric potentials, Shi and Wang \cite{SW24} obtained the existence of Anderson-localized states via the CWB method again. The proof in \cite{SW24} follows the spirit of \cite{BW08} but employs new ingredients, including the deep geometric lemma of Bourgain \cite{Bou07} and certain Diophantine estimates on manifolds \cite{KM98}. For more results on nonlinear Anderson localization, we refer to \cite{GSW23, SW23, KLW24}.

The existence of Anderson-localized solutions for the nonlinear Maryland model on $\Z^d$ for $d>1$ seems, however, to have remained completely open until the present paper.

\subsection{Main new ingredients of the proof}
Our proof is based on the CWB method together with spectral analysis of the Maryland model, and is more in the spirit of Liu-Wang \cite{LW24}: We first diagonalize the Maryland model and construct the desired Anderson-localized states near its eigenfunctions. The general scheme is based on Lyapunov-Schmidt decomposition and the Nash-Moser iteration. Of particular importance is the so-called large deviation theorem (LDT) for Green's functions. In contrast to \cite{LW24, SW24}:

\begin{itemize}
\item  One of the main differences compared with \cite{LW24} is the lack of parameters: Due to the long-range corrections of the quasi-periodic potential, we cannot derive the eigenvalue estimates of linearized operators via potential or eigenvalue variations as in \cite{LW24}. Fortunately, the fine spectral properties of the Maryland model established in \cite{BLS83, CSZ24, KPS24} allow us to handle resonances without potential variations. Moreover, the dependence relation $\varepsilon\leq \tilde\varepsilon(\delta)\sim \delta^{\frac18+}$  emerges during the approximation of eigenvalues with the potential (cf. (\ref{error between varepsilon and delta 1}) and (\ref{error between varepsilon and delta 2}) below). This requirement can be avoided in \cite{LW24}, the i.i.d. random potential case, since the independence of the random potential at different lattice sites enables direct eigenvalue variations. In 
	\cite{LW24},  it can be $\varepsilon=1$ if $d=1$ and  the sites $\beta_1,\cdots, \beta_b$ associated with the  tangential frequency  $\omega^{(0)}$ are carefully selected. In the present case, however, we can choose any distinct $\beta_1,\cdots,\beta_b\in\Z^d$, which could be useful for future work on full-dimensional KAM tori.
			
\item  Since the tangential frequency $\omega$ in the present setting is of order $\mathcal{O}(\delta)$ rather than $\mathcal{O}(1)$ as in \cite{BW08}, we must treat the intermediate scales of the LDT for Green's functions. For this purpose, we use Diophantine estimates on irrational shifts of functions from \cite{SW24}, which rely on the specific properties of the potential. In the proof of the LDT for Green's functions in large scales, the weak second Melnikov condition (\ref{weak second Melnikov}) on $\omega$ is used to show the \textit{smallness} (or sub-linear bound) of the number of bad boxes for normal frequency sites with large norms. In contrast, in \cite{SW24}, the weak second Melnikov condition together with Bourgain's geometric lemma (see \cite{Bou07} and also Lemmas 3.5, 3.12 in \cite{SW24}) provides the sub-linear bound. However, in the Lipschitz monotone case, the variation arguments on eigenvalues of the operator restricted to boxes with general shapes, as used in \cite{SW24}, are difficult to apply.  Additionally, the Rellich functions analysis MSA method can only provide information on eigenvalues for so-called ``n-regular'' blocks (see Appendix A or \cite{CSZ24} for details), whose iterations highly depend on the geometric structures of resonant blocks. Thus, applying the non-diagonalized CWB method of \cite{SW24} to the Lipschitz monotone potential case would complicate the iterations significantly. Another important issue is how to obtain off-diagonal decay after establishing the $\ell^2$-operator norm estimates. In \cite{SW24}, Bourgain's geometric lemma plays a crucial role by providing a good annulus structure of resonant regions, allowing the application of the resolvent identity from \cite{JLS20} (cf. Theorem 3.3) to achieve off-diagonal decay. In this paper and \cite{LW24}, however, we first diagonalize the Schrödinger operator and obtain the exponential decay of eigenfunctions, ensuring off-diagonal decay via the standard Neumann series argument (i.e., Lemma \ref{Neumann expansion}). This exponential decay can be transferred to the nonlinear term $\mathcal W_{u}$ in Section  \ref{NLsect} and then to the linearized operator $\tilde{H}$ as in (\ref{perturbation Op}). From this perspective, we can view the decay of eigenfunctions as a replacement for Bourgain's geometric lemma. This is reasonable because Bourgain's geometric lemma can also be used to prove the exponential decay of eigenfunctions \cite{Bou07}.

\item For the Maryland model, its eigenvalues are unbounded, which is quite different from the situations in \cite{LW24} and \cite{SW24}. Therefore, it requires establishing both upper and lower bounds (cf. (\ref{(2.4)})) on some eigenvalues by removing more $\theta$, which again relies largely on the fine spectral properties of the Maryland model. Note that the lower bound of eigenvalues plays an important role in the initial steps of the Newton iteration in nonlinear analysis: In \cite{SW24}, this can be handled via certain potential translations. We also need an upper bound on the eigenvalues for the application of the matrix-valued Cartan's lemma.
\end{itemize}

\subsection{Structure of the paper}
This paper is organized as follows. Some spectral properties on Maryland type operators are introduced in Section \ref{Eigensect}. The LDT for Green's functions and the proof  are  presented in  Section \ref{LDTsect}. In Section \ref{NLsect}, we prove our main theorem via the CWB method. Some useful arguments are included in appendixes. 

\section{Eigenvalues estimates of the Schr\"odinger operator}\label{Eigensect}

Although we focus on  $V(\theta)=\cot(\pi\theta)$ in our main theorem,  we  deal with more general Lipschitz monotone potentials in this section.  Let $V(\theta):\ [0,1)\rightarrow\R$ be a  continuous function   satisfying,   for some $L>0$ and any $0\leq \theta_1\leq \theta_2<1$, 
\[V(\theta_1)-V(\theta_2)\geq L(\theta_2-\theta_1).\]
Consider the operator
\[H=H(\theta)=\varepsilon\Delta+V(\theta+\alpha\cdot j)\delta_{j,j'}.\]

In the following, we will introduce some eigenvalues estimates for $H$, which will play  an essential role in the nonlinear analysis.  Of particular importance is certain lower bounds on linear combinations of eigenvalues (or the separation property of eigenvalues) of $H$  with $V=\cot(\pi\theta).$

We first recall the Anderson localization result from \cite{KPS24}. 

\begin{lem}[Theorem 1.1, \cite{KPS24}]\label{KPS}
	Let $0<\eta<1$  and $\bfa\in \DC$.  Then there exists some  $\varepsilon_0=\varepsilon_0(\eta,\gamma,\tau,d,L)>0$ such that  for  $0<\varepsilon<\varepsilon_0$,  there exist  a $1$-periodic  Lipschitz monotone function $E(\theta):\ \R\rightarrow \R\cup\{\infty\}$ with Lipschitz constant still $L$, strictly increasing in $[0,1)$, and a $1$-periodic function $\phi(\theta):\ \R\rightarrow \ell^2(\Z^d)$ such that 
	\[H(\theta)\phi(\theta)=E(\theta)\phi(\theta)\ {\rm for}\ {\forall}\theta\in \R.\]
	Moreover, we have 
	\begin{equation}\label{Localization}
		\nm\phi(\theta)\nm_{\ell^2(\Z^d)}=1,\ |\phi(\theta,0)-1|<\varepsilon^{1-\eta},\ |\phi(\theta,j)|<\varepsilon^{(1-\eta)|j|_1} \ {\rm for}\  j\in \Z^d\setminus\{0\}.
	\end{equation}
\end{lem}

Based on this theorem, we can prove 

\begin{lem}\label{prop of eigenvectors}
\begin{itemize}
\item[(1)] Define  $(T^{m}\phi)(j)=\phi(j-m),\ m\in \Z^d$. Then $\{\phi_m\}_{m\in\Z^d}=\{(T^m\phi)(\theta+m\cdot\alpha)\}_{m\in \Z^d}$ forms an orthonormal basis of eigenvector of $H(\theta)$ with corresponding  eigenvalues $\{\mu_m(\theta) = E(\theta+m\cdot\bfa)\}_{m\in\Z^d}$. 

\item[(2)]  For a.e.  $\theta\in\T$ and $0<\varepsilon\leq c(\eta)\ll1,$ we have
\begin{equation}\label{Approximate eigenvalue by potential}
	|E(\theta)-V(\theta)|\leq 2d\varepsilon. 
\end{equation}
\end{itemize}
\end{lem}

\begin{proof}
	(1) This is just Corollary 1.2 in \cite{KPS24}.
	
	(2) Remove a zero-measure set of $\theta\in\T$ to avoid the presence of infinite values of $H(\theta)\phi(\theta).$  Then, expanding $H(\theta)\phi(\theta)=E(\theta)\phi(\theta)$ at $0$ yields  
	\[(E(\theta)-V(\theta))\phi(\theta,0)=\varepsilon\sum_{|j|_1=1}\phi(\theta,j).\]
	Thus, applying  $(\ref{Localization})$ shows that for $0<\varepsilon\leq c(\eta)\ll1,$
	\[|E(\theta)-V(\theta)|=\varepsilon\sum_{|j|_1=1}\frac{|\phi(\theta,j)|}{|\phi(\theta,0)|}\leq 2d\varepsilon\frac{\varepsilon^{1-\eta}}{1-\varepsilon^{1-\eta}}\le 2d\varepsilon.\]
	\end{proof}

Let us make some comments on Lemmas \ref{KPS}, \ref{prop of eigenvectors}. 

\begin{rmk}
	\begin{itemize}
	\item[(1)] Lemma \ref{KPS} holds also for $\alpha$ satisfying weak  Diophantine condition and for more general  H\"older monotone potential (cf.  \cite{KPS24} for details).  Anderson localization for $H$  with  Lipschitz monotone potentials  was first proved in \cite{CSZ24} via multi-scale analysis based on Rellich functions analysis. However, in that paper, the property concerning eigenvalue $E(\theta)$ is not proved. 
	Inspired by   the insight of \cite{KPS24}, we can  prove similar finer results on  eigenvalues and eigenfunctions  as in \cite{KPS24} via Rellich functions analysis (cf.  Appendix \ref{CSZapp} for details).
	
	\item[(2)] As a matter of fact, Lemma \ref{prop of eigenvectors} automatically provides the relabelling of eigenfunctions such that  $\phi_j$ has a unique localization center $j$,  and every  lattice site $j\in\Z^d$  admits  a unique eigenfunction $\phi_{j}$ peaking on it. The uniqueness of the localization center is largely due to the monotonicity of the potentials and will not necessarily hold in general cases, such as the i.i.d. potentials \cite{LW24} or the cosine-like ones \cite{CSZ24b}. In fact, in the general cases,  not only one eigenfunction may have more than one localization centers, but also one lattice site may have no eigenfunctions or more than one eigenfunction peaking on it. However, once one can obtain some semi-uniform localization control of the form
	\[|\phi_j(l)|\leq C (1+|l_j|)^q e^{-c|l-l_j|}\]
	 with $C>0$ being independent of $j$, the relabelling can be constructed in arbitrary dimensions through Hall's marriage theorem. We believe that such a proof is of independent interest and have included it in Appendix \ref{Hallapp}. We hope this result will be helpful in future studies of Anderson localization for NLS with more general quasi-periodic potentials.	
\end{itemize}
\end{rmk}

We are ready to establish the desired lower bounds on the separation property of eigenvalues of  $H$. At present, however, we can only handle the special case $V=\cot(\pi\theta)$. 

We restrict ourselves to $V(\theta)=\cot(\pi\theta)$ and take $\eta=\frac{1}{2}$ in Lemma  $\ref{KPS}$. Then, for all $(\varepsilon,\ct)\in [0,\varepsilon_0(\gamma,\tau,d)]\times\T,$  we can find $\phi(\theta), E(\theta)$ such that 
\begin{align}\label{phiidef}
\{\phi_j(\theta)\}_{j\in\Z^d}=\{(T^j\phi)(\theta+j\cdot\alpha)\}_{j\in \Z^d}
\end{align}
is a complete set of eigenfunctions of $H(\ct)$ with corresponding  eigenvalues  
\begin{align}\label{muidef}\{\mu_j(\theta)=\mu_{0}(\theta+j\cdot\bfa):=E(\theta+j\cdot\bfa)\}_{j\in\Z^d}.\end{align}
 Here, $E(\theta)$ is still nondecreasing and  Lipschitz monotone with a Lipschitz constant $\pi$. Let $0<\delta\ll 1$. This subsection 
aims to establish  separation  property  of eigenvalues  by removing  some ``bad''  $(\varepsilon,\theta), $ in order to get the $\tilde{\varepsilon}=\tilde{\varepsilon}(\gamma,\tau,b,d,B,\delta)$ and $\Theta=\Theta_{\alpha}\subset [0,1]$ as stated in Theorem \ref{Main Thm}.

We begin with a useful lemma on the symmetry  property of  $E(\theta)$, which may inherit from that of $V$.  While we believe  such a result may have been proven in \cite{BLS83}, we provide a proof relying on Lemma \ref{KPS} for completeness.

\begin{lem}\label{center symmertic}
Fix  $V(\theta)=\cot(\pi\theta)$ in Lemma \ref{KPS}. Let  $E(\theta),\psi(\theta)$ be  as in  Lemma \ref{KPS} with $\psi(\theta)$ peaking at $0\in \Z^d$. We have $E(\theta)=-E(1-\theta)$ and $\psi(\theta, j)=e^{i\pi\sum\limits_{k=1}^{d}j_k}\psi(1-\theta, -j)$. Especially, the roots of $E(\theta)$ are given by   $\frac{1}{2}+\Z^d$, which are  independent of $\varepsilon$. 
\end{lem}
\begin{proof}
	Since 
	\[(\varepsilon\Delta+\cot \pi(\theta+\jalpha)\delta_{j,j'})\psi(\theta)=E(\theta)\psi(\theta),\]
	and replacing  $\theta$ with  $1-\theta$,  we have, due to $\cot(\pi-x)=\cot x$, 
	\[(\varepsilon\Delta-\cot\pi(\theta-\jalpha)\delta_{j,j'})\psi(1-\theta)=E(1-\theta)\psi(1-\theta).\]
Denote $\varPsi=\psi(1-\theta)$ and $\tilde{\varPsi}(j)=\varPsi(-j)$. Direct computation shows that 
	\[\varepsilon\sum_{|m-j|_1=1}\tilde{\vp}(m)-\cot\pi(\theta+\jalpha)\tilde{\vp}(j)=E(1-\theta)\tilde{\vp}(j).\]
	This  implies
	\begin{align}\label{psit}
	(-\varepsilon\Delta+\cot \pi(\theta+\jalpha)\delta_{j,j'})\tilde{\vp}=-E(1-\theta)\tilde{\vp}.
	\end{align}
	
	Now, we define the Fourier transformation operator $U:\ L^2(\T^d)\rightarrow \ell^2(\Z^d)$ as 
	\[U(F(x))=(a_j)_{j\in\Z^d}\]
	with 
	\[F(x)=\sum_{j\in\Z^d}a_j e^{i2\pi j\cdot x},\ x\in\T^d.\]
	If we take $\tilde{\Delta}=U^{-1}\Delta U:\ L^2(\T^d)\rightarrow L^2(\T^d)$, we then get by direct computations  that 
	\[\tilde{\Delta}F(x)=2(\sum_{k=1}^{d}\cos 2\pi x_k)F(x)\]
	Let $T_{\frac{1}{2}}:\ L^2(\Z^d)\rightarrow L^2(\T^d)$ to be the shift operator given by 
	\[T_{\frac{1}{2}}F(x)=F(x+(\frac{1}{2},\frac{1}{2},\cdots,\frac{1}{2})):=F(x+\mathbf{\frac{1}{2}}).\]
	Then  $(T_{\frac{1}{2}})^2=id$ with $id$ denoting the identity operator, and 
	\begin{align*}
		T_{\frac{1}{2}}\tilde{\Delta}T_{\frac{1}{2}} F(x)=-\tilde{\Delta}F(x).
	\end{align*}
	That shows $T_{\frac{1}{2}}\tilde{\Delta}T_{\frac{1}{2}}=-\tilde{\Delta}$. Moreover, for any $\phi\in\ell^2(\Z^d)$, 
	\begin{align*}
		UT_{\frac{1}{2}}U^{-1}\phi&= UT_{\frac{1}{2}}\big( \sum_{j\in \Z^d}\phi(j)e^{i2\pi j\cdot x}     \big) \\
		      &=(\phi(j)e^{i\pi\sum_{k=1}^{d}j_k})_{j\in \Z^d},
	\end{align*}
	which means that $UT_{\frac{1}{2}}U^{-1}:\ \ell^2(\Z^d)\rightarrow\ell^2(\Z^d)$ is a multiplier. By taking $\tilde{\vp}=UT_{\frac{1}{2}}U^{-1}\Phi$ and recalling  \eqref{psit}, we obtain
	\[-\varepsilon UT_{\frac{1}{2}}U^{-1}\Delta UT_{\frac{1}{2}}U^{-1}\Phi+UT_{\frac{1}{2}}U^{-1}\cot \pi(\theta+\jalpha)\delta_{j,j'}UT_{\frac{1}{2}}U^{-1}\Phi=-E(1-\theta)\Phi.\]
	From 
	\[-UT_{\frac{1}{2}}U^{-1}\Delta UT_{\frac{1}{2}}U^{-1}=U(-T_{\frac{1}{2}}\tilde{\Delta}T_{\frac{1}{2}})U^{-1}=U\tilde{\Delta}U^{-1}=\Delta\]
	and the fact $UT_{\frac{1}{2}}U^{-1}$  is a multiplier as shown above,  we get  
	\[UT_{\frac{1}{2}}U^{-1}\cot \pi(\theta+\jalpha)\delta_{j,j'}UT_{\frac{1}{2}}U^{-1}\Phi= \cot \pi(\theta+\jalpha)\Phi,\]
	which means 
	\[(\varepsilon\Delta+\cot \pi(\theta+\jalpha)\delta_{j,j'})\Phi=-E(1-\theta)\Phi, \]
	and $\Phi$ is also an  eigenfunction of $H(\theta)$. 
	
	Now by the construction  of $\Phi$, we have $\Phi(j)=e^{i2\pi j\cdot x}\vp(-j)$ and $|\Phi(j)|=|\vp(-j)|$. Since $\vp=\psi(1-\theta)$ peaks at $0$, $\Phi$ peaks at $0$ as well. Then by (1) of Lemma \ref{prop of eigenvectors}, we must have 
	\[\Phi=\psi(\theta),\ E(\theta)=-E(1-\theta),\]
	which implies $\psi(\theta, j)=e^{i\pi\sum\limits_{k=1}^{d}j_k}\psi(1-\theta, -j)$.  Finally, by Lemma \ref{KPS}, $E(\theta)$ is monotone.  The fact that $E(\theta)$ is centrally symmetrical  about  the point $(\frac{1}{2},0)$ ensures  that $\frac{1}{2}$ is the unique root of $E(\theta)$ in the interval $[0,1)$.  
	\end{proof}

Recall that $\beta_1,\cdots,\beta_b\in\Z^d$  are fixed and distinct,  and $\max\limits_{1\le k\leq b}|\beta_k|\leq B.$ Below is our  main theorem in this section. 

\begin{thm}[Separation  property of eigenvalues]\label{seperative spectrum}
	Fix $\bfa\in \DC$ and denote $\omega^{(0)}=(\mu_{\beta_1},\mu_{\beta_2},\cdots,\mu_{\beta_b})$.  Then  there is a constant $\tilde{\delta}=\tilde{\delta}(\gamma,\tau,b,d,B)$ such that for all $0<\delta\leq \tilde{\delta}$, there exist a constant $\tilde{\varepsilon}(\delta)>0$ depending only  on $\delta$, and a set $\Theta=\Theta_{\bfa,\delta}\subset [0,1]$ depending  on $\bfa,\delta$ with   
	\[{\rm meas}([0,1]\backslash \Theta) \leq |\log\delta|^{-1}\]
	so that the following statements hold true for all $0<\varepsilon\leq \tilde{\varepsilon}$ and $\theta \in\Theta$. 
	\begin{itemize}
	\item[(1)]  $\nm\phi_j\nm_{\ell^2(\Z^d)}=1,\ |\phi_j(j)-1|<\varepsilon^{\frac{1}{2}},\  |\phi_j(x)|<\varepsilon^{\frac{1}{2}|x-j|_1} \  {\rm for}\  x\neq j\in \Z^d$.
	\item[(2)]  For any $N\geq |\log\delta|^2 ,\ |j|,|j'|\leq N$ and $j\neq j',$ 
     \begin{equation}\label{(2.3)}
			|\mu_j-\mu_{j'}|\geq \frac{\pi \gamma}{(2d)^{\tau}N^{\tau}}
	 \end{equation}
	\indent and 
	\begin{equation}\label{(2.4)}
			\frac{1}{2N^{d+2}}\leq |\mu_j|\leq  2N^{d+2}.
	\end{equation}
	\item[(3)]  For any $(n,j)\in [-e^{|\log\delta|^{\frac{3}{4}}},e^{|\log\delta|^{\frac{3}{4}}}]^{b+d}\backslash\{(-e_k,\beta_k)\}_{k=1}^b,$
	\begin{equation}\label{(2.5)}
			|n\cdot\omega^{(0)}+\mu_j|\geq 2\delta^{\frac{1}{8}};
	\end{equation}
	\indent and for any $(n,j)\in [-e^{|\log\delta|^{\frac{3}{4}}},e^{|\log\delta|^{\frac{3}{4}}}]^{b+d}\backslash\{(e_k,\beta_k)\}_{k=1}^b,$
	\begin{equation}\label{(2.6)}
			|-n\cdot\omega^{(0)}+\mu_j|\geq 2\delta^{\frac{1}{8}}.
	\end{equation}
	\item[(4)]  For any $(n,j,j')\in [-3|\log\delta|^K,3|\log\delta|^K]^{b+d+d}\backslash \{(-e_k+e_{k'},\beta_k,\beta_{k'})\cup(0,0,0)\}_{k,k'=1}^b$,	\begin{equation}\label{(2.7)}
			|n\cdot\omega^{(0)}+\mu_j-\mu_{j'}|>2\delta^{\frac{1}{8}}. 
	\end{equation}
Here $K=K(b)>2$ will be chosen to be a large constant depending only on $b$ in Section 3.
\end{itemize}
\end{thm}
\begin{rmk}\label{tepsdel}
As we will say below, we can set 
$$\tilde\varepsilon(\delta)=\delta^{\frac18+}.$$
\end{rmk}

\begin{proof}[Proof  of Theorem $\ref{seperative spectrum}$]

	(1) This  follows directly from Lemma \ref{KPS} if we  assume 
	\begin{equation}\label{varepsilon bound 1}
		 0< \varepsilon\leq \varepsilon_0(\gamma,\tau,d). 
	\end{equation}

	(2) By $\bfa\in\DC$ and $E(\theta)$ is  Lipschitz monotone, we obtain 
	\begin{align*}
		|\mu_j-\mu_{j'}| &=|E(\theta+j\cdot\bfa)-E(\theta+j'\cdot\bfa)|\\
		             &\geq \pi\nm(j-j')\cdot\bfa\nmt\geq\frac{\pi\gamma}{|j-j'|_1^{\tau}}\\
					 &\geq \frac{\pi \gamma}{(2d)^{\tau}N^{\tau}}. 
	\end{align*}
	 Moreover,   since $\mu_j=E(\theta+j\cdot\bfa)$ and $E(\theta)$ is  Lipschitz monotone with   a Lipschitz constant $\pi$, we can  use  Lemma \ref{center symmertic}  (i.e., $E(\frac{1}{2})=0$) to show 
	 \[\left\{\theta\in \T:\  |\mu_j(\theta)|<\frac{1}{2N^{d+2}}\right\}\subset \Theta^1_{j,N} =: \left\{\theta\in \T: \ |\theta+\jalpha-\frac{1}{2}|<\frac{1}{2\pi N^{d+2}}\right\}\]
	 with 
	 \[\meas (\Theta^1_{j,N})\leq \frac{1}{\pi N^{d+2}}.\]
	Then it suffices to  take account of all $ N\geq |\log\delta|^2,|j|\leq N$, which  will give a set 
	 \begin{equation}\label{2.8}
			 \Theta^1=\bigcup_{N\geq |\log\delta|^2}\bigcup_{|j|\leq N} \Theta^1_{j,N}
	 \end{equation}
	 with  estimate 
	 \begin{align}\label{2.9}
		\meas(\Theta^1) &\leq \sum_{N\geq |\log\delta|^2}\frac{(2N+1)^d}{\pi N^{d+2}}\\
		   \notag    &\le C(d) \sum_{N\geq |\log\delta|^2}N^{-2}\\
			 \notag &\le C(d) |\log\delta|^{-2} \ll |\log\delta|^{-\frac{3}{2}},
	 \end{align}
	 where the last inequality needs  $0<\delta\le c(d)\ll1$. Hence, for any $\theta\notin\Theta^1$, $|\mu_j|\geq\frac{1}{2N^{d+2}}$ will hold  for all $N\geq |\log\delta|^2,|j|\leq N$. Finally, we will mention  that although $E(\theta)$ depends on $\varepsilon$, its root (i.e.,  $\frac{1}{2}$) does not.   Consequently, $\Theta^1$ does not  depend on $\epsilon$ and  is solely dependent on $\alpha,\delta$.	
	 
The upper bound $2N^{d+2}$ of $|\mu_j|$ can be proved by the following arguments. By (2) of Lemma \ref{prop of eigenvectors}, we know that $|\mu_j(\theta)-\cot\pi(\theta+\jalpha)|\leq 2d\varepsilon$. Hence,
	 \begin{align*}
	 \{\theta\in \T: |\mu_j(\theta)|> 2N^{d+2} \}&\subset \{\theta\in \T:\  |\cot\pi(\theta+\jalpha)|>N^{2+d} \} \\
	 &=\{\theta\in \T: \ \nm \theta+\jalpha\nmt <\frac{2}{\pi} (\frac{\pi}{2}-\arctan N^{d+2}) \}.
	 \end{align*}
	We denote 
	 \[\Theta^0_{j,N}=  \{\theta\in \T:\  \nm \theta+\jalpha\nmt <\frac{2}{\pi} (\frac{\pi}{2}-\arctan N^{d+2}) \}. \]
	 Since 
	 \[\frac{2}{\pi}(\frac{\pi}{2}-\arctan N^{d+2})\sim \frac{2}{\pi N^{d+2}} \  {\rm as} \ N\rightarrow \infty,\]
	 we have 
	 \[\meas (\Theta^0_{j,N})\le C(d) \frac{1}{ N^{d+2}}.\]
	 Taking account of all $ N\geq |\log\delta|^2,|j|\leq N$ will give  a set 
	 \begin{equation}\label{Theta^2}
			 \Theta^0=\bigcup_{N\geq |\log\delta|^2}\bigcup_{|j|\leq N} \Theta^0_{j,N}
	 \end{equation}
	 with 
	 \begin{align}\label{Theta^2 mes}
		\meas(\Theta^0) &\leq \sum_{N\geq |\log\delta|^2}\frac{(2N+1)^d}{\pi N^{d+2}}\\
		   \notag    &\le C(d) \sum_{N\geq |\log\delta|^2}N^{-2}\\
			 \notag &\le C(d) |\log\delta|^{-2} \ll |\log\delta|^{-\frac{3}{2}},
	 \end{align}
	  where the last inequality needs  $0<\delta\le c(d)\ll1$.  Hence, for any $\theta\notin\Theta^0$, $|\mu_j|\leq 2N^{d+2}$ will hold  for all $N\geq |\log\delta|^2,|j|\leq N$. As a result,  \eqref{(2.4)}  hold true for any $\theta\notin\Theta^1\cup\Theta^0$.

\begin{rmk}
	The independence of root of $E(\theta)$ in $\varepsilon$ is largely due to the fact that $V(\theta)=\cot\pi(\theta)$  is centrally symmetric, and makes the set of $\theta$ independent of $\varepsilon$. However, for general potentials $V$,  we do not necessarily have Lemma \ref{center symmertic}, and the root of the eigenvalue  $E(\theta)$ may change as $\varepsilon$ varies.  Consequently, the set $\Theta^1$ may depend on $\varepsilon$. This is quite different from \cite{SW24}. In \cite{SW24}, the potential is bounded, hence one can ensure that the eigenvalues are bounded from below via potential translation (see Remark 1.1 in \cite{SW24})  (cf. $\lambda_l>\frac{1}{2}$ in Lemma 3.12, \cite{SW24}). However, the Maryland potential is unbounded, which leads to the failure of the translation argument! For this reason, we must additionally remove some $\theta$. This is inevitable because we need a positive lower bound on the eigenvalues during the Newton scheme in Section 4.  Indeed, if we assume that the root of $E(\theta)$  undergoes a continuous drift as $\varepsilon$ varies, the feasible initial parameter space would be
	\[\mcP_{\bfa,\delta}=\bigcup_{0<\varepsilon\leq \tilde{\varepsilon}}\{\varepsilon\}\times\Theta_{\bfa,\delta,\varepsilon},\]
   where each sector $\Theta_{\bfa,\delta,\varepsilon}$ would be a dense Cantor set with a drift (this can be seen from the construction of $\Theta^1$). This may result in $([0,1]\times\{\theta\}) \cap \mcP_{\bfa,\delta}$  not containing some interval $[0,\varepsilon_1]$ for any $\theta$.  This implies that, after fixing $\alpha,\theta,\delta$, one cannot expect $\varepsilon$ to be made arbitrarily small. From this perspective, it is better to view $\varepsilon$ as a parameter rather than a perturbative  scale.
	
	\end{rmk}

To prove parts (3) and (4) of Theorem $\ref{seperative spectrum}$, we need some estimates on the transversality of functions in the spirit of Kleinbock-Margulis \cite{KM98}, which can be used to remove resonant parameters. 

\begin{lem}[cf. Lemma A.4 of  \cite{SW24}]\label{Margulis}
Let $I\subset \R$ be a finite interval and $k\geq 1$. If $f\in C^k(I,\R)$ satisfies 
\[\inf_{x\in I}|\frac{d^k}{dx^k}f(x)|\geq A>0\]
then for all $\zeta>0,$ 
\[\meas(\{x\in I:\  |f(x)|\leq \zeta\})\leq C_k (\frac{\zeta}{A})^{\frac{1}{k}},\]
where $C_k=k(k+1)[(k+1)!]^{\frac{1}{k}}$.
\end{lem}
The proof of Lemma $\ref{Margulis}$ is elementary and can be found in either  \cite{KM98}  or Appendix A of  \cite{SW24}.  

We also need an estimate involving Cramer's rule and Hadamard's inequality. 

\begin{lem}[cf. Lemma A.5 of  \cite{SW24}]\label{Hadamard}
	Let $W$ be a $s\times s$ real invertible matrix with  its entry satisfying $\sup\limits_{1\leq p,q\leq s}|W(p,q)|\leq M$. Then for any $\mathbf{v}\in \R^s,$  we have
	\[\max_{1\leq p\leq s}|(M\cdot \mathbf{v})(p)|\geq s^{-3/2} \frac{|\det(W)|\cdot \nm\mathbf{v}\nm_2}{\sup_{p,q}|W_{p.q}|},\]
	where $W_{p,q}$ is the cofactor of $W$. Moreover,
	\[\sup_{p,q}|W_{p,q}|\leq (s-1)^{s-1}M^{s-1}.\]
\end{lem}

Recall that we take $V(\theta)=\cot(\pi\theta)$. To prove parts (3) and (4) of Theorem $\ref{seperative spectrum}$, we first establish the separation property of shifts of $\cot(\pi\theta)$ and then the eigenvalues by the  perturbation $(\ref{Approximate eigenvalue by potential})$.

In the following, we will prove parts (3) and (4) of Theorem $\ref{seperative spectrum}$,  completing the proof of the entire theorem. 

Consider $\theta\in [0,1)$. Fix  $2\leq s\leq b+2$,  $s$-many lattice sites $j_1,j_2,\cdots,j_s\subset\Z^b$ with $\max\limits_{1\le l\leq s}|j_l|\leq N$ and integer vector $\mathbf{k}=(k_1,k_2,\cdots,k_s)\in \Z^s$ with $0<\min\limits_{1\le l\leq s}|k_l|\le\max\limits_{1\le l\leq s}|k_l|\leq N$.  Denote 
	\[F_{\bfk,j_1,j_2,\cdots,j_s}(\theta)=\sum_{l=1}^{s}k_l\cot\pi(\theta+j_l\cdot \bfa).\]
	Hence, for $1\leq t\leq s$,  we have 
	\[\frac{d^t}{d\theta^t}F_{\bfk,j_1,j_2,\cdots,j_s}(\theta)=\sum_{l=1}^{s}k_l\cot^{(l)}\pi(\theta+j_l\cdot \bfa)\]
    where $V^{(l)}(\cdot)=\frac{d^l}{d\theta^l}V(\cdot)$ is the $l$-th derivative. Moreover, denote $g(\ct)=\tan(\pi\theta)$ and then $\cot(\pi\theta)=-g(\theta+\frac{1}{2})$. Thus, we get  
	\begin{align*}
	\mathbf{D} &=  
	            \begin{pmatrix}
					\frac{d}{d\theta}F_{\bfk,j_1,j_2,\cdots,j_s}(\theta)\\
					\frac{d^2}{d\theta^2}F_{\bfk,j_1,j_2,\cdots,j_s}(\theta)\\
					\vdots\\
					\frac{d^s}{d\theta^s}F_{\bfk,j_1,j_2,\cdots,j_s}(\theta)
				\end{pmatrix}
				:= -W \mathbf{k}, 	
	\end{align*}
	where $W$ is a $s\times s$ matrix with entry given by 
	\[W(p,q)=g^{(p)}(\ct+\frac{1}{2}+j_q\cdot\bfa).\]
	Now we want to apply Lemma $\ref{Hadamard}$. Observe  that  
	\[\frac{d^p}{dx^p}(\tan x)=\frac{1}{\cos^2 x}\poly_{(p)}(\tan x)\] 
	for some deterministic polynomial $\poly_{(p)}(\cdot)$  of degree $p-1$ (whose coefficients  are absolute constants independent of $x$). Hence,
	\[W(p,q)=\frac{\pi^p}{\cos^2 \pi(\theta+\frac{1}{2}+j_q\cdot\bfa)} \poly_{(p)}(g(\ct+\frac{1}{2}+j_q\cdot\bfa)).\]
	Direct calculations  give 
	\begin{align}\label{Vandermonde}
		\det(W)&= C''(s) \prod_{1\leq q\leq s}\frac{1}{\cos^2 \pi(\theta+\frac{1}{2}+j_q\cdot \bfa)} \\
		    \notag        &\ \  \cdot \left|\begin{matrix}
							1 &1 &\cdots &1\\
						\tan\pi(\theta+\frac{1}{2}+j_1\cdot \bfa) &\tan\pi(\theta+\frac{1}{2}+j_2\cdot\bfa) &\cdots &\tan\pi(\theta+\frac{1}{2}+j_s\cdot\bfa)\\
						\vdots &\vdots &\ &\vdots \\
						\tan^{s-1}\pi(\theta+\frac{1}{2}+j_1\cdot \bfa) &\tan^{s-1}\pi(\theta+\frac{1}{2}+j_2\cdot\bfa) &\cdots &\tan^{s-1}\pi(\theta+\frac{1}{2}+j_s\cdot\bfa)
					  \end{matrix}\right| \\
					  \notag &=C''(s) \prod_{1\leq q\leq s}\frac{1}{\cos^2 \pi(\theta+\frac{1}{2}+j_q\cdot \bfa)} \cdot
					 \prod_{1\leq p<q\leq s}(\tan\pi(\theta+\frac{1}{2}+j_p\cdot \bfa)-\tan\pi(\theta+\frac{1}{2}+j_q\cdot \bfa))\\
					  \notag &=C''(s) \prod_{1\leq q\leq s}\frac{1}{\cos^2 \pi(\theta+\frac{1}{2}+j_q\cdot \bfa)} \cdot\prod_{1\leq p<q\leq s}(\cot\pi(\theta+j_q\cdot\bfa)-\cot\pi(\theta+j_p\cdot\bfa))
	\end{align}
	for some constant $C''(s)>0$ depending only on $s$. The second equality is due to the calculation of s Vandermonde determinant. Similarly, one can check  
    \[W_{i,j}=  \prod_{1\leq q\neq j \leq s}\frac{1}{\cos^2 \pi(\theta+\frac{1}{2}+j_q\cdot \bfa)} \cdot \widetilde{W}_{i,j},\]
    where $\widetilde{W}=(\poly_{(p)}(g(\ct+\frac{1}{2}+j_q\cdot\bfa)))_{1\leq p,q\le s}$ and $\widetilde{W}_{i,j}$ is the cofactor of it. Applying  Lemma $\ref{Hadamard}$ implies 
	\begin{align*}
		\max_{1\leq p\leq s}|\mathbf D(p)| &\geq s^{-3/2} \frac{|\det(W)|\cdot \nm\mathbf{k}\nm_2}{\sup_{p,q}|W_{p.q}|}\\
		 &=C''(s) \inf_{p,q}(\frac{|\det(W)|}{|W_{p,q}|})\cdot\nm \mathbf{k}\nm_2\\
		 &=C''(s)\nm\mathbf{k}\nm_2 \inf_{p,q}\left(\frac{1}{|\cos^2\pi(\ct+\frac{1}{2}+j_q\cdot \bfa )|\cdot|\widetilde{W}_{p,q}|}\right)\\
		 &\ \  \cdot\prod_{1\leq p<q\leq s}|\cot\pi(\theta+j_q\cdot\bfa)-\cot\pi(\theta+j_p\cdot\bfa)|\\
		 &\geq C(s)\frac{\prod_{1\leq p<q\leq s}\nm (j_p-j_q)\cdot\bfa\nmt}{\sup_{p,q}|\widetilde{W}_{p,q}|},
	\end{align*}
	where in the last inequality, we have used  $\nm \mathbf k\nm_2\geq 1, \cos^2 x\leq 1$ and $\cot(\pi\theta)$ is  Lipschitz monotone.

	To apply Lemma \ref{Margulis}, we first perform a cut-off on $\tan (\pi \theta)$. Fix any  $X\gg1$. One can check that the measure of  each of two intervals  $I_1\cup I_2\subset [-\frac{1}{2},\frac{1}{2}]$ such that $|\tan(\pi\ct)|>X$ satisfies  
	\[\frac{2}{\pi}(\frac{\pi}{2}-\arctan X)\sim \frac{2}{\pi X} \  {\rm as} \ X\rightarrow \infty.\]
	So,  if we take 
	\[\Theta_{j}=\{\theta:\ \theta+\frac{1}{2}+j\cdot\bfa \in I_1\cup I_2\},\quad \Theta^2_N=\bigcup_{|j|\leq N}\Theta_j, \]
	we have 
	\begin{equation}\label{2.10}
			\meas(\Theta_N^2)\lesssim N^d X^{-1}. 
	\end{equation}
	Then, for $\theta$ outside of $\Theta_N^2$,  all $|g(\ct+\frac{1}{2}+j_q\cdot\bfa)|\leq X$. Consequently,
	\[|\widetilde{W}(p,q)|=|\poly_{(p)}(g(\ct+\frac{1}{2}+j_q\cdot\bfa))|\leq C(s)X^{s-1}=M\ {\rm for}\ {\forall} 1\leq p,q\leq s,\]
	where we used the fact that a finite upper bound of coefficients of $\poly_{(p)},1\leq p\leq s$, depends only on $s$.  Applying  Lemma $\ref{Hadamard}$  yields  $\sup_{p,q}|\widetilde{W}_{p,q}|\lesssim_s M^{s-1}= X^{(s-1)^2}$, and hence,
	 \begin{align}\label{transversal property of cotan}
	 \max_{1\leq p\leq s}|\mathbf D(p)| &\geq C(s) X^{-(s-1)^2}\prod_{1\leq p<q\leq s}\nm (j_p-j_q)\cdot\bfa\nmt\\
                 \notag      &\geq C(s,\gamma,\tau)X^{-(s-1)^2}N^{-\frac{s(s-1)}{2}\tau} := A>0,
	 \end{align}
	which is the lower bound needed in Lemma $\ref{Margulis}$. So, applying Lemma $\ref{Margulis}$  shows that for  
	 \[\Theta_{\mathbf{k},j_1,\cdots,j_s}^3(\zeta):=\{\theta\in \T\backslash \Theta_N^2:\ |F_{\bfk,j_1,j_2,\cdots,j_s}(\theta)|<\zeta\}, \]
	we have the following bound 
	 \[\meas(\Theta_{\mathbf{k},j_1,\cdots,j_s}^3(\zeta))\le C(s,\gamma,\tau)\zeta^{\frac{1}{s}}X^s N^{\frac{s-1}{2}\tau}.\]  
	 
     Finally, taking account of    $2\leq s\leq b+2,\  \bfk,j_1,\cdots,j_s$ leads to  
	 \[\Theta^3_N(\zeta)=\bigcup_{2\leq s\leq b+2,\ }\bigcup_{|\bfk|,|j_1|, \cdots,|j_s|\leq N}\Theta_{\mathbf{k},j_1,\cdots,j_s}^3(\zeta)\]	
     with 
      \begin{equation}\label{2.11}
			 \meas(\Theta^3_N(\zeta))\leq C(b,\gamma,\tau)\zeta^{\frac{1}{b+2}}X^{b+2} N^{C(b,\tau,d)}.
	  \end{equation}
	In conclusion,  we prove that for any $\theta\notin \Theta^3_N(\zeta)\cup \Theta_N^2$,   $|F_{\bfk,j_1,j_2,\cdots,j_s}(\theta)|\geq \zeta$ for any $2\leq s\leq b+2$ and any $\bfk,j_1\cdots,j_s$ with $\ell^{\infty}$-norm less than $N$.

    Now we can prove (3) and (4) of Theorem $\ref{seperative spectrum}$. Choose $0<\delta\leq C(B,b)\ll1 $  such that $|\beta_1|,\cdots,|\beta_b|\leq B\leq 2|\log\delta|^K\ll e^{|\log\delta|^{\frac{3}{4}}}$.

	(3) Take $N=e^{|\log \delta|^{\frac{3}{4}}},\zeta=3\delta^{\frac{1}{8}}$ in the above arguments. If $n=0,$  then by \eqref{(2.4)},
	\[|\mu_j|\geq\frac{1}{2}e^{-(d+2)|\log \delta|^{\frac{3}{4}}}\gg 2\delta^{\frac{1}{8}} \]
	provided  $0<\delta<c(d)\ll1$. Otherwise,  if $n\neq  0$, at least two terms are present, and we can use the above estimates. Let 
	\begin{equation}\label{2.12}
			\Theta'=\Theta^3_N(\zeta)\cup \Theta_N^2.
	\end{equation}
	By \eqref{2.10}  and  \eqref{2.11},  we have 
	\[\meas(\Theta')\le C({b,\gamma,\tau})e^{d|\log\delta|^{\frac{3}{4}}}X^{-1}+\delta^{\frac{1}{8(b+2)}}X^{b+2}e^{C(b,\tau,d)|\log\delta|^{\frac{3}{4}}}.\]
    Taking $X=\delta^{-\frac{1}{16(b+2)^2}}$ and letting $0<\delta<c(b,\tau,\gamma,d)\ll1$  lead to 
	\begin{equation}\label{(2.13)}
		\meas(\Theta')\leq \delta^{\frac{1}{18(b+2)}}, 
	\end{equation}
	and for any $\theta\notin\Theta'$, 
	\[|n\cdot V^{(0)}+V_j|=|F_{(n,1),\beta_1,\cdots,\beta_b,j}(\theta)|\geq \zeta=3\delta^{\frac{1}{8}}\]
	for all $(n,j)\in [-e^{|\log\delta|^{\frac{3}{4}}},e^{|\log\delta|^{\frac{3}{4}}}]^{b+d}\backslash\{(-e_k,\beta_k)\}_{k=1}^b,n\neq 0$, where $V^{(0)},V_j$ are given by 
	\[V_j(\theta)=\cot\pi(\theta+j\cdot\bfa),\ j\in \Z^d,\]
	\[V^{(0)}=(V_{\beta_1},V_{\beta_2},\cdots,V_{\beta_b}).\]
	Now  applying  $(\ref{Approximate eigenvalue by potential})$ implies  
	\begin{equation}\label{approached muj}
		|V_j-\mu_j|=|V(\theta+j\cdot\bfa)-E(\theta+j\cdot\bfa)|\leq 2d\varepsilon\quad {\forall}\theta \in \T, j\in\Z^d.
	\end{equation}
	Hence,
	\begin{align}\label{error between varepsilon and delta 1}
		|n\cdot\omega^{(0)}+\mu_j| &\geq |n\cdot V^{(0)}+V_j|-\sum_{k=1}^{b}|n_k|\cdot |V_{\beta_k}-\mu_{\beta_i}|-|V_j-\mu_j|\\
		\notag  &\geq 3\delta^{\frac{1}{8}}-2d(b+1)\varepsilon e^{|\log \delta|^{\frac{3}{4}}}\\
		\notag  &\geq 2\delta^{\frac{1}{8}},
	\end{align}
	where in the  last inequality, we need
	\begin{equation}\label{varepsilon bound 2}
		0<\varepsilon\leq\frac{1}{2d(b+1)}\delta^{\frac{1}{8}}e^{-|\log\delta|^{\frac{3}{4}}}.
	\end{equation}
	Thus, we have ensured (\ref{(2.5)}) under the assumptions that  $\theta\notin \Theta'$ and  $(\ref{varepsilon bound 2})$.  The bound  $(\ref{(2.6)})$ holds true  via  the same arguments. 
	
	(4) Take now $N=2|\log \delta|^K,\zeta=2\delta^{\frac{1}{8}}$ with  $K=K(b)>2$ only depending  on $b$.  
	If $n=0$ then by $(\ref{(2.3)})$, 
     \[			|\mu_j-\mu_{j'}|\geq \frac{\pi \gamma}{(6d)^{\tau}|\log\delta|^{K\tau}}\gg 2\delta^{\frac{1}{8}}\]
	provided  $0<\delta<c(\gamma,\tau,b)\ll1$. 
	Otherwise,  if $j=j'$ and $\supp(n)=1$, then from $(\ref{(2.4)})$ and similar arguments used in the proof of (3), we can ensure that (\ref{(2.7)}) holds. Here,
	\[\supp(n)=\#\{1\le k\le b:\ n_k\neq 0\}.\]
	Thus,  it suffices to consider  either   $n\neq 0,j\neq j'$ or  $\supp(n)>1,j=j'$. In both cases, at least two terms are present, and we can use the above estimates. Define 
	\begin{equation}\label{2.14}
			\Theta''=\Theta^3_N(\zeta)\cup \Theta_N^2.
	\end{equation}
	By $(\ref{2.10})$ and $(\ref{2.11})$,  we have 
	\[\meas(\Theta'')\le C(b,d,\gamma,\tau) |\log \delta|^{dK}X^{-1}+\delta^{\frac{1}{8(b+2)}}X^{b+2}|\log\delta|^{C(b,\tau,d)K}.\]
    Taking $X=\delta^{-\frac{1}{16(b+2)^2}}$ and $0<\delta<c(b,\tau,\gamma,d)\ll1 $ will ensure 
	\begin{equation}\label{(2.15)}
		\meas(\Theta')\leq \delta^{\frac{1}{18(b+2)}}, 
	\end{equation}
	and for any $\theta\notin\Theta''$,  the estimate 
	\[|n\cdot V^{(0)}+V_j-V_{j'}|=|F_{(n,1,-1),\beta_1,\cdots,\beta_b,j,j'}(\theta)|\geq \zeta=2\delta^{\frac{1}{8}}\]
	holds true for any $(n,j,j')\in [-2|\log\delta|^K,2|\log\delta|^K]^{b+d+d}\backslash \{(-e_k+e_{k'},\beta_k,\beta_{k'})\}_{k,k'=1}^b$ except for  $n=0,j=j'$. Applying   $(\ref{approached muj})$   deduces  
	\begin{align}\label{error between varepsilon and delta 2}
		|n\cdot\omega^{(0)}+\mu_j-\mu_{j'}| &\geq |n\cdot V^{(0)}+V_j-V_{j'}|-\sum_{k=1}^{b}|n_k|\cdot|V_{\beta_k}-\mu_{\beta_k}|-|V_j-\mu_j|-|V_{j'}-\mu_{j'}|\\
		\notag  &\geq 3\delta^{\frac{1}{8}}-4d(b+2)\varepsilon |\log \delta|^K\\
		\notag  &\geq 2\delta^{\frac{1}{8}},
	\end{align}
	where in the last inequality, we need
	\begin{equation}\label{varepsilon bound 3}
		0<\varepsilon\leq\frac{1}{4d(b+2)}\delta^{\frac{1}{8}}|\log\delta|^{-K}
	\end{equation}
	Thus,  we have ensured $(\ref{(2.7)})$ under the assumptions that $\theta\notin \Theta''$ and  $(\ref{varepsilon bound 3})$.

	Combining the above cases motivates the definition of the entire set of parameters
	\begin{equation}\label{feasible theta}
			\Theta=\Theta_{\bfa,\delta}=\T\backslash(\Theta^1\bigcup\Theta^0\bigcup\Theta'\bigcup\Theta'')
	\end{equation}
	with 
	\begin{equation}\label{feasible theta measure}
		\meas([0,1]\backslash \Theta)\leq 2|\log\delta|^{-\frac{3}{2}}+2\delta^{\frac{1}{18(b+2)}} \ll |\log\delta|^{-1}
	\end{equation}
	provided  $0<\delta\leq c(b)\ll1.$ 
	Moreover, the restrictions on $\varepsilon$ coming  from  $(\ref{varepsilon bound 1}),(\ref{varepsilon bound 2}) ,  (\ref{varepsilon bound 3})$  allow us to take 
	\[0\leq \varepsilon \leq \tilde{\varepsilon}(\delta) \leq \min\left\{\varepsilon_0,\frac{1}{2d(b+1)}\delta^{\frac{1}{8}}e^{-|\log\delta|^{\frac{3}{4}}},\frac{1}{4d(b+2)}\delta^{\frac{1}{8}}|\log\delta|^{-K}\right\}.\]
	Thus,  $\tilde{\varepsilon}(\delta)\sim \delta^{\frac{1}{8}+}$ suffices, provided  $0<\delta<c(b,d,\tau,\gamma)\ll1$, as mentioned in Remark \ref{tepsdel}. 
	 
	Finally, by summarizing all smallness restrictions on $\delta$, it suffices to require $0<\delta<\tilde{\delta}(\gamma,\tau,b,d,B)\ll1$.

	This completes the whole proof of Theorem $\ref{seperative spectrum}$.
\end{proof}

\begin{rmk}
\begin{itemize}
	\item[(1)] For fixed  $\bfa\in\DC$ and  $0<\delta<\tilde{\delta}(\gamma,\tau,b,d,B)\ll1$,  we denote by  $\mathcal{P}_{\alpha,\delta}=[0,\tilde{\varepsilon}]\times\Theta_{\bfa,\delta}$  the  feasible initial parameters  space.

	\item[(2)]  Our proof of (\ref{(2.7)}) relies largely on the fact that the specific form of the potential, namely $\cot(\pi\ct)$),  allows for good control over the Wronskian (\ref{Vandermonde}). This is essentially the calculation of the Vandermonde determinant. The desired lower bound on the determinant can be derived using the Diophantine condition. However, in \cite{SW24}, while the authors can handle more general quasi-periodic potentials, they cannot fix the Diophantine $\alpha$ and must remove some values of $\alpha$. 
\end{itemize}
\end{rmk}

\section{Linear analysis: large deviation theorem (LDT)}\label{LDTsect}

In this section, we establish the Large Deviation Theorem (LDT) for the Green's functions of the linearized operators (\ref{mcW}). This result is crucial for the nonlinear analysis in Section 4. The procedure we follow is now quite classical and standard (cf. e.g., \cite{LW24, SW24, KLW24}).

Assume that $\tH$ is an operator on $\ell^2(\{\pm\}\times \Z^b\times\Z^d)$, and that the elements in its matrix representation satisfy 
\[\tH(r,n,j;r',n',j')=h_{r,r'}(n-n',j,j'),r,r'\in\{\pm\},(n,j),(n',j')\in\Z^{b+d},\]
and 
\[(\tH u)_r(n,j)=\sum_{(n',j')\in\Z^{b+d} \atop r'\in\{\pm\}}h_{r,r'}(n-n',j,j')u_{r'}(n',j').\]
Assume further  that there exist  $C_1>1$ and $c_1>0$ such that 
\begin{equation}\label{perturbation Op}
	|h_{r,r'}(n-n',j,j')|\leq C_1 (1+|n|)^{C_1} e^{-c_1|n|-c_1|j-j'|-c_1\max\{|j|,|j'|\}}.
\end{equation}

\begin{rmk}
	The polynomial component in (\ref{perturbation Op}) naturally emerges during the iteration process when we pass from the approximating solution $u^{(r)}$ to its linearized operator $\mcW_{u^{(r)}}$, as in Proposition \ref{generate the nonlinear operator}. One can include these polynomial components into the exponential decay term with a certain loss of the exponential decay rate$c_1$. However, one may not obtain an estimate like 
	\[C_1 (1+|n|)^{C_1} e^{-c_1|n|}\leq C_2 e^{-qc_2|n|}\]
	with some $0<q<1$, because this would cause deterioration of the constant during the Newton scheme. Such a problem was first pointed out in \cite{HSSY24} and was also explained in detail in \cite{LW24}. Therefore, rather than removing the polynomial components, it is more practical to retain these polynomial factors during the iteration. Indeed, in this case,  the constant $C_1$ can be made independent of the iteration scales, and the polynomial factors will not essentially affect the MSA (this will be explained in Appendix \ref{JLSapp}).  
\end{rmk}

Now let $D(\sigma):\ \R\rightarrow \Op [\ell^2(\{\pm\}\times \Z^b\times\Z^d)]$ be a $\sigma$-parameterized family of diagonal operators on $\ell^2(\{\pm\}\times \Z^b\times\Z^d)$: 
\begin{equation}\label{shift diagonal D}
	{D}(\sigma)=\begin{pmatrix}D_+(\sigma)
		&0\\
		0& D_- (\sigma)
	\end{pmatrix}
\end{equation}
with 
\[D_{\pm}={\rm diag}(\pm (n\cdot \omega +\sigma)+\mu_j),\ (n,j)\in \Z^{b+d}.\]
Define $\sigma$-parameterized operator $T=T(\sigma):\ \R\rightarrow \Op [\ell^2(\{\pm\}\times \Z^b\times\Z^d)]$ as 
\begin{equation} \label{shift op T}
	T(\sigma)=D(\sigma)+\delta \tH
\end{equation}

In this section, for any operator $T$, the norm $\nm T \nm$ always stands for the $\ell^2$-operator norm. Next, we  introduce some definitions  regarding subsets  in $\Z^{b+d}$. The following definitions originate from \cite{BGS02, Bou07} and \cite{JLS20} and are constructed for the convenience of performing multi-scale analysis (MSA) using resolvent identities. For more details on MSA, we refer to Appendix \ref{JLSapp}. Denote by $\La_N$ an elementary region of size $N$ centered at 0, which is one of the following regions:
\begin{equation*}
	\La_N=[-N,N]^{b+d}
\end{equation*}
or
\[\La_N=[-N,N]^{b+d}\setminus\{n\in\mathbb{Z}^{b+d}: \ n_k \ \varsigma_k \ 0, 1\leq k\leq b+d\},\]
where  for $ k=1,2,\cdots,b+d$, $ \varsigma_k\in \{<,>,\emptyset\}$ and at least two $ \varsigma_k$  are not $\emptyset$.
Denote by $\mathcal{E}_N^{0}$ the set of all elementary regions of size $N$ centered at 0. Let $\mathcal{E}_N$ be the set of all translations of  elementary regions  with center at 0, namely,
$$\mathcal{E}_N:=\{(n,j)+\La_N:\ (n,j)\in\mathbb{Z}^{b+d},\La_N\in \mathcal{E}_N^{0}\}.$$
For simplicity, we also call  elements in $\mathcal{E}_N$ elementary regions.  

Let $\La_N(j_0)=\{(n,j)\in\Z^b\times \Z^d:\  (n,j-j_0)\in \La_N\}$ denote the shift along the $j$-direction.
The  width  of a subset  $\Lambda\subset \Z^{b+d}$ is defined  as the maximum 
$M\in \N$ such that  for any  $n\in \Lambda$, there exists  $\hat{M}\in \mathcal{E}_M$ such that $	n\in \hat{M} \subset \Lambda$, and 
\begin{equation*}
	\text{ dist }(n,\Lambda\backslash \hat{M})\geq M/2.
\end{equation*}

Furthermore, a generalized elementary region is defined to be a subset $\Lambda\subset \Z^{b+d}$ of the form
\begin{equation*}
	\Lambda:= R\backslash(R+z),
\end{equation*}
where $z\in\Z^{b+d}$ is arbitrary and $R$ is a rectangle with center $y'\in \Z^{b+d}$,
\begin{equation*}
	R=\{(y_1,y_2,\cdots,y_{b+d})\in \Z^d:\  |y_1-y_1^\prime|\leq M_1, \cdots,|y_{b+1}-y_{b+1}^\prime|\leq M_{b+1}\}.
\end{equation*}

For $ \Lambda\subset\mathbb{Z}^{b+d}$,  we introduce its diameter,
\[\mathrm{diam}(\Lambda)=\sup_{n,n'\in \Lambda}|n-n'|.\]
Denote by $\mathcal{R}_N$
all  generalized elementary regions with diameters less than or equal  to $N$. Denote by $\mathcal{R}_N^M$
all   generalized elementary regions in $\mathcal{R}_N$ with width larger than or equal to $M$.

With a slight abuse of the notation, we also  use $\mathcal{E}_N$, $\mathcal{E}_N^{0}$, $\La_N$, $\La_N(j_0)$,  $\mathcal{R}_N$ and  $\mathcal{R}_N^M$  to denote $\{\pm\}\times \mathcal{E}_N$, $\{\pm\}\times\mathcal{E}_N^{0}$, $\{\pm\}\times \La_N$, $\{\pm\}\times \La_N(j_0)$,  $\{\pm\}\times \mathcal{R}_N $ and  $\{\pm\}\times \mathcal{R}_N^M,$ 
respectively. 
Similar, for any $\Lambda\subset \Z^{b+1}$, denote by  $R_{\Lambda}$  the restriction  to $\{\pm\}\times \Lambda$.

\begin{Def}
Let  $T(\sigma)$  be given by \eqref{shift op T}. We say that $T(\sigma)$ satisfies the $(\tilde{c},N)$-LDT  if there exists a subset $\Sigma_{N}\subset \R$ such that 
\[\meas(\Sigma_N)\leq e^{-N^{\frac{1}{30}}},\]
and for $j_0\in [-2N,2N]^{d},\La_N\in \mcE_N^0$, and $\sigma\notin \Sigma_N$,  we have 
\begin{itemize}
	\item 
	  \begin{equation}\label{LDT L2 norm}
		\nm (R_{\La_N(j_0)}T(\sigma)R_{\La_N(j_0)})^{-1} \nm \leq e^{N^{\frac{9}{10}}};
	  \end{equation}
	\item  for any $(n,j)$ and $(n',j')$ satisfying $|n-n'|+|j-j'|\geq N^{\frac{1}{2}}$, 
		 \[|(R_{\La_N(j_0)}T(\sigma)R_{\La_N(j_0)})^{-1}(r,n,j;r',n^\prime,j')| \leq e^{-\tilde{c}(|n-n'|+|j-j'|)} \ {\rm for}\ {\forall}r,r'\in \{\pm\}.\]
		 For simplicity, we just hidden the index $r,r'$ and write 
		 \begin{equation} \label{LDT off diagonal}
			|(R_{\La_N(j_0)}T(\sigma)R_{\La_N(j_0)})^{-1}(n,j;n^\prime,j')| \leq e^{-\tilde{c}(|n-n'|+|j-j'|)}.
		\end{equation} 
\end{itemize}
\end{Def}

Let $K_1=K_1(b)$ be a sufficiently large constant depending only on $b$. Let $$K=K_1^{100},\ K_2=K_1^5,$$ and take $\tilde{\delta}$ small enough to ensure \eqref{(2.7)} in Theorem $\ref{seperative spectrum}$ matching  with $K$. 

We now state the main theorem of LDT in this section. 
\begin{thm}[LDT]\label{LDT}
	Assume that $\omega=(\omega_1,\omega_2,\cdots,\omega_b)\in \R^b$ satisfies 
	\begin{enumerate}
		\item   $|\omega_k-\mu_{\beta_k}|\leq C_2\delta$ for all $k=1,2,\cdots,b$ and for some $C_2>0$;
		\item for any fixed  $N\geq (\log\frac{1}{\delta})^{K}$, 
		and any  $\tilde{N}$  with $(\log\frac{1}{\delta})^{K}\leq \tilde{N}\leq N$  and $0\neq |n|\leq 2\tilde{N}$,  we have    the Diophantine condition
		\begin{equation}\label{Diophantine on omega frequency}
			|n\cdot\omega| \geq e^{-\tilde{N}^{\frac{1}{K_2}}},
		\end{equation}
		and for any $ |j|\leq 3\tilde{N}, |j'|\leq 3\tilde{N}, |n|\leq 2\tilde{N}$ with $(n,j-j')\neq 0$, we have the so called weak second Melnikov's condition 
		\begin{equation}\label{weak second Melnikov}
			| n\cdot\omega-\mu_j+\mu_{j'}|\geq  e^{-\tilde{N}^{\frac{1}{K_2}}}.
		\end{equation}
	\end{enumerate}	
Then for  $0<\delta\le c(\gamma,\tau,b,d,B,p)\ll1$ (so it needs to  take $\tilde{\delta}$ in Theorem $\ref{seperative spectrum}$ smaller), the $(c_{\tilde N},\tN)$-LDT holds for $T(\sigma)$ at any scale $\tN\leq N$ with  $c_{\tilde N}\geq \frac{c_1}{2}$. 
\end{thm}
\begin{rmk}
	Indeed,  the weak second Melnikov's condition \eqref{weak second Melnikov} for  $n=0$ and $j\neq j'$ can be ensured by \eqref{(2.3)} in Theorem $\ref{seperative spectrum}$:  from  $3\tN\geq 3|\log\delta|^K>|\log\delta|^2$, one has  
	\[|\mu_j-\mu_{j'}|\geq\frac{\pi\gamma}{(2d)^{\tau}(3\tN)^{\tau}}\gg e^{-\tilde{N}^{\frac{1}{K_2}}} \]
	provided   $0<\delta<\tilde{\delta}\ll1$.
\end{rmk}

The following subsections aim to prove Theorem \ref{LDT}. 

\subsection{Some useful lemmas}
To  prove Theorem $\ref{LDT}$,  we need  some useful lemmas which will be used repeatedly during the proof. 

The first lemma is a perturbation argument via Neumann's expansion (cf. \cite{Shi22,Liu22,LW24}).
For any finite set $S\subset\Z^{b+d}$, denote by $|S|=\# S$ the volume of $S$. We have 
\begin{lem}[cf. Lemma 4.2., \cite{LW24}]\label{Neumann expansion}
	Let $S\subset \Z^{d+b}$.  Assume that $A$ and $B$ are two matrices with entries $A_{r,r'}(m, m')$ and $B_{r,r'}(m,m')$, where $(r,m),(r',m')\in \{\pm\}\times S$. Assume further  that $|B_{r,r'}(m,m')|\leq \epsilon_2 (|m-m'|+1)^{C} e^{-c||m-m'||}$, $|(A^{-1})_{r,r'}(m,m')|\leq \epsilon_1^{-1}e^{-c||m-m'||}$ and $\nm A^{-1}\nm\leq \epsilon_1^{-1}$, 
	with $\epsilon_1,\epsilon_2,c>0,C>1$ and $||(n,j)||=|n|+|j|$. Then for 
	\begin{equation}\label{Verification Neumann}
		4|S|^2 (\diam(S)+1)^C \epsilon_2\epsilon_1^{-1}\leq\frac{1}{2}, 
	\end{equation}
	we have 
	\[\nm (A+B)^{-1}\nm\leq 2\epsilon_1^{-1}\]
	and 
	\[|(A+B)^{-1}_{r,r'}(m,m')-(A^{-1})_{r,r'}(m,m')|\leq \epsilon_1^{-1}e^{-c||m-m'||}.\]
\end{lem}
\begin{proof}
	Let $N=|S|,M=\diam(S)$.  Then applying  Schur's  test yields  $\nm B\nm\leq 2N(M+1)^C \epsilon_2$ and 
	\[\nm BA^{-1}\nm \leq 2N(M+1)^C\epsilon_2\epsilon_1^{-1}\leq 4N^2 (M+1)^C\epsilon_2\epsilon_1^{-1}\leq \frac{1}{2}.\]
	Using Neumann series argument  shows 
	\begin{align*}
		\nm (A+B)^{-1}\nm & =\nm A^{-1}\sum_{s\geq 0}(-BA^{-1})^s\nm \\
		        &\leq \nm A^{-1}\nm \frac{1}{1-\nm BA^{-1}\nm}\leq 2\epsilon_1^{-1}. 
	\end{align*}
	Moreover, direct computations imply  
	\begin{align*}
		&\ \ \ |(A+B)^{-1}_{r,r'}(m,m') - (A^{-1})_{r,r'}(m,m')|\\
		 &\leq \epsilon_1^{-1}\nm  \sum_{s\geq 1 \atop m_k,m'_k\in S,r_k,r'_k\in \{\pm\}} \prod_{1\leq k\leq s}(1+\|m_k-m'_k\|)^C (\epsilon_2\epsilon_{1}^{-1})^s e^{-c||m-m_1||-c||m_1-m'_1||-\cdots -c||m'_s-m'||}\\
		 &\leq \epsilon_1^{-1} e^{-c||m-m'||}\sum_{s\geq 1}(4N^2 (M+1)^C\epsilon_2\epsilon_1^{-1})^s\\
		 &\leq 8N^2 (M+1)^C \epsilon_2\epsilon_1^{-2}e^{-c||m-m'||}\\
		 &\leq \epsilon_1^{-1}e^{-c||m-m'||}.
	\end{align*}
\end{proof}	
\begin{rmk}
	From the proof above and Lemma 4.2 of \cite{LW24}, it is evident that there is essentially no difference (regarding off-diagonal estimates) between matrices on $\Z^{b+d}$ 
	and those on $\{\pm\}\times \Z^{b+d}$. This is why we usually omit the indices  $\{\pm\}$. 
\end{rmk}

The second useful lemma is the Matrix-value Cartan's Lemma originated  from \cite{BGS02}.  

\begin{lem}[cf. Chapter 14, \cite{Bou05}]\label{Cartan lemma}
Let $\mathcal T(x)$ be a   $N\times N$ matrix function of a parameter $x\in[-\xi,\xi]$ satisfying the following conditions:
	\begin{itemize}
		\item[(i)] $T(x)$ is real analytic in $x\in [-\xi,\xi]$ and has a holomorphic extension to
		\begin{equation*}
			\mathcal{D}_{\xi,\xi_1}=\left\{z: \ |\Re z|\leq\xi,\ |\Im{z}|\leq \xi_1\right\}
		\end{equation*}
		satisfying
		\begin{equation}\label{mc1}
			\sup_{z\in \mathcal{D}_{\xi,\xi_1}}\|\mathcal T(z)\|\leq B_1, B_1\geq 1.
		\end{equation}
		\item[(ii)]  For all $x\in[-\xi,\xi]$, there is a subset $\Lambda\subset [1,N]$ with
		\begin{equation*}|
			\Lambda|\leq M,
		\end{equation*}
		and
		\begin{equation}\label{mc2}
			\|(R_{[1,N]\setminus \Lambda}\mathcal T(x)R_{[1,N]\setminus \Lambda})^{-1}\|\leq B_2, B_2\geq 1.
		\end{equation}
		\item[(iii)]
		\begin{equation}\label{mc3}
			\meas\{x\in[-{\xi}, {\xi}]: \ \|\mathcal T^{-1}(x)\|\geq B_3\}\leq 10^{-3}\xi_1(1+B_1)^{-1}(1+B_2)^{-1}.
		\end{equation}
		Let
		\begin{equation}\label{mc4}
			0<\epsilon\leq (1+B_1+B_2)^{-10 M}.
		\end{equation}
	\end{itemize}
	Then
	\begin{equation}\label{mc5}
		\mathrm{meas}\left\{x\in\left[-{\xi}/{2}, {\xi}/{2}\right]:\  \|\mathcal T^{-1}(x)\|\geq \epsilon^{-1}\right\}\leq C \xi e^{-c\left(\frac{\log \epsilon^{-1}}{M\log(B_1+B_2+B_3)}\right)},
	\end{equation}
	where $C$ and $c$ are absolute constants.
\end{lem}

To apply Lemma \ref{Cartan lemma}, we also need to introduce semi-algebraic sets. A set $\mathcal{S}\subset \mathbb{R}^d$ is called {\it semi-algebraic} if it is a finite union of sets defined by a finite number of polynomial equalities and inequalities. More precisely, let $\{P_1,\cdots,P_s\}\subset\mathbb{R}[x_1,\cdots,x_d]$ be a family of real polynomials whose degrees are bounded by $\kappa$. A (closed) semi-algebraic set $\mathcal{S}$ is given by an expression
\begin{equation}\label{smd}
	\mathcal{S}=\bigcup\limits_{l}\bigcap\limits_{\ell\in\mathcal{L}_l}\left\{x\in\mathbb{R}^d: \ P_{\ell}(x)\varsigma_{l\ell}0\right\},
\end{equation}
where $\mathcal{L}_l\subset\{1,\cdots,s\}$ and $\varsigma_{l\ell}\in\{\geq,\leq,=\}$.  We say that $\mathcal{S}$ has degree at most $s\kappa$. In fact, the degree of $\mathcal{S}$, denoted by $\deg(\mathcal{S})$, is defined as the smallest $s\kappa$ over all representations as in (\ref{smd}).

Some basic properties of semi-algebraic sets are presented below. These properties are special cases of those in \cite{Bas99} and are restated in \cite{Bou05}.

\begin{lem}[cf. Theorem 9.3, \cite{Bou05}]\label{basu}
	Let $ \mathcal{S}\subset [0,1]^d$ be a semi-algebraic  set of degree $B$. Then  the number of connected components of
	$ \mathcal{S}$ does not exceed $(1+B)^{C(d)}$.
\end{lem}

\begin{lem}[cf. Proposition 9.2, \cite{Bou05}]  \label{projection semi algebraic}
	Let $\mathcal{S}\subset [0,1]^{d_1+d_2}$ be a semi-algebraic  set of degree $B$. Let $(x,y)\in \R^{d_1}\times \R^{d_2}$.
	Then the projection ${\rm proj} _{x_1}\mathcal(S)$ is a semi-algebraic set of degree at most 
	$(1+B)^{C(d_1,d_2)}$.
\end{lem}

We are now ready to prove the Large Deviation Theorem (LDT) via multi-scale analysis (MSA) induction. The induction scales can be decomposed into three parts: small scales (via Neumann series argument), intermediate scales (via separation property of eigenvalues in Section 2), and large scales (via semi-algebraic set arguments and matrix-valued Cartan's lemma).

\subsection{LDT in small scales}

We first prove LDT for  small scales 
$$1\ll  N_0(b, d)\leq N\leq |\log \delta|^{10}.$$
Define 
\begin{align*}
&\ \ \ \Sigma_N\\
&:=\bigcup_{(n,j)\in [-N,N]^b\times [-3N,3N]^d}\big\{  \sigma\in\R:\  |(\nomega +\sigma )+\mu_j|\leq 2e^{-N^{\frac{1}{20}}} \ {\rm or} \ |-(\nomega +\sigma )+\mu_j|\leq 2e^{-N^{\frac{1}{20}}}    \big\}.
\end{align*}
It is easy to see that for $N\geq N_0(b,d)\gg1,$
\begin{equation*}
	\meas(\Sigma_N)\leq 4 (2N+1)^b(6N+1)^d e^{-N^{\frac{1}{20}}}\leq e^{-N^{\frac{1}{30}}}.
\end{equation*}

Next,  we will apply  Lemma \ref{Neumann expansion} with  $S=\La_N(j_0), A=R_{S}D(\sigma) R_S, B=\delta R_S \tH R_S,\epsilon_1=2e^{-N^{\frac{1}{20}}},\epsilon_2=C_1 \delta$ and $C=C_1,c=c_1$.  We have 
\begin{align*}
	4|S|^2(\diam(S)+1)^C\epsilon_2\epsilon_1^{-1} &\leq 2 C_1 \delta (2N+1)^{2(b+d)+C_1} e^{N^{\frac{1}{20}}}\\
	       &\le C({b,p,d}) \delta |\log\delta|^{20(b+d)+C_1} e^{|\log\delta|^{\frac{1}{2}}}\leq \frac{1}{2}   
\end{align*}
provided $0<\delta\le c(b,p,d,C_1)\ll 1.$
 Hence, using  Lemma \ref{Neumann expansion} shows 
\[		\nm (R_{\La_N(j_0)}T(\sigma)R_{\La_N(j_0)})^{-1} \nm\leq e^{N^{\frac{1}{20}}} \leq e^{N^{\frac{9}{10}}},\]
and 
\[			|(R_{\La_N(j_0)}T(\sigma)R_{\La_N(j_0)})^{-1}(n,j;n^\prime,j')| \leq \frac{1}{2}e^{N^{\frac{1}{20}}} e^{-c_1(|n-n'|+|j-j'|)} \]
for any $(n,j),(n',j')\in S$. 

Now assume $|n-n'|+|j-j'|\geq N^{\frac{1}{2}}$. Then 
\begin{align*}
	&\ \ \ |(R_{\La_N(j_0)}T(\sigma)R_{\La_N(j_0)})^{-1}(n,j;n^\prime,j')| \\
	&\leq \frac{1}{2}e^{N^{\frac{1}{20}}} e^{-c_1(|n-n'|+|j-j'|)} \\
	                 &\leq  e^{-(c_1-\frac{N^{\frac{1}{20}}}{|n-n'|+|j-j'| })(|n-n'|+|j-j'|)}\\
					 &\leq e^{-(c_1-N^{\frac{1}{20}-\frac{1}{2}})(|n-n'|+|j-j'|)}.
\end{align*}
This shows that we can  take $c_N=c_1-N^{-\frac{9}{20}}>\frac{c_1}{2}$.

\subsection{LDT in intermediate scales} 
Next, we prove LDT  for the intermediate scales $$|\log\delta|^{10}<N\leq |\log\delta|^K.$$

Remember that   $|\omega-\omega^{(0)}|\leq C_2\delta $. Fix  $|\log\delta|^{10}<N\leq |\log\delta|^K$ and let $\Sigma_N$ be the set  of $\sigma\in\R$ such that at least one of \eqref{LDT L2 norm} and \eqref{LDT off diagonal} does not hold. Then for $\sigma\in \Sigma_N$, we must have for some 
$(n,j)\in [-N,N]^b\times[-3N,3N ]^d$, either $|\sigma+\nomega+\mu_j|\leq \delta^{\frac{3}{8}}$ or $|\sigma+\nomega-\mu_j|\leq \delta^{\frac{3}{8}}$. Otherwise, applying Lemma \ref{Neumann expansion} will yield that for any $j_0$ with 
$|j_0|\leq 2N$,  we have 
\[\nm (R_{\La_N(j_0)}T(\sigma)R_{\La_N(j_0)})^{-1} \nm\leq 2\delta^{-\frac{3}{8}}\leq e^{N^{\frac{9}{10}}},\]
and for any $(n,j),(n',j')$ such that $|n-n'|+|j-j'|\geq \sqrt{N},$
\[  |(R_{\La_N(j_0)}T(\sigma)R_{\La_N(j_0)})^{-1}(n,j;n^\prime,j')| \leq e^{-(c_1-N^{\frac{3}{80}-\frac{1}{2}})(|n-n'|+|j-j'|)}.\]
In fact, the condition  \eqref{Verification Neumann} can be  verified  via  $$4C_1(2N+1)^{2(d+b)+C_1}\delta^{\frac{5}{8}}\le C({C_1,b,d})|\log\delta|^{K(2(d+b)+C_1)}\delta^{\frac{5}{8}}\ll\frac{1}{2}$$
provided $0<\delta \le c(C_1, b,d)\ll1$ since  $K$ is a constant depending only on $b$. 
Therefore, one can restrict $\si$ to an  intervals of size $2\delta^{\frac{3}{8}}$ like
 \[I_{(n,j)}=\{\si:\ |\sigma+\nomega+\mu_j|\leq \delta^{\frac{3}{8}}\}\ {\rm for\ some}\ (n,j)\in \Lambda_N(j_0). \]
 The total number of those intervals  are bounded by   
\[2\cdot \#\{(n,j)\in [-N,N]^b\times[-3N,3N ]^d\}\leq 10^{b+d}N^{b+d}.\]
Hence, $\Sigma_N$ can be covered with at most  $10^{b+d}N^{b+d}$ intervals of length $2\delta^{\frac{3}{8}}$. Moreover, by \eqref{(2.4)}, we have $|\mu_j|\lesssim N^{d+2}$, which implies   $I_{(n,j)}\subset [-100N^{d+3},100N^{d+3}]$.  Denote by $\{I_k\}$ all of those intervals  and  we take one of them, say $I_0$,  for consideration. Without loss of generality,    
we may assume that $I_0$ has the form 
\[I_0=\{\si\in\R:\ |\sigma+n_0\cdot \omega+\mu_{j_0}|\leq \delta^{\frac{3}{8}} \ {\rm and} \ (n_0,j_0)\in [-N,N]^b\times[-3N,3N ]^d\}.\]
The  proof for the other case is totally the same. 

Now for any $\si\in\R$, let 
\begin{equation*}
	\mathcal{A}_1^\si=\{(n,j)\in[-N,N] ^{b}\times[-3N,3N]^d :\  |\si+n\cdot \omega^{(0)}+\mu_j|\leq \delta^{\frac{1}{8}}\},
\end{equation*}
and 
\begin{equation*}
	\mathcal{A}_2^\si=\{(n,j)\in[-N,N] ^{b}\times[-3N,3N]^d :\  |\si +n\cdot \omega^{(0)}-\mu_j|\leq \delta^{\frac{1}{8}}\}.
\end{equation*}
We introduce a   lemma based on the separation property  of eigenvalues (cf. Theorem \ref{seperative spectrum}).
\begin{lem}\label{less resonance intermediate scale}
	Let $0<\delta\ll1$  and $0<\varepsilon<\tilde{\varepsilon}(\delta),\theta\in \Theta$  so that Theorem \ref{seperative spectrum} holds true. Then  
	\[\#\mathcal{A}_1^\si\leq b,\ \# \mathcal{A}_2^\si\leq b.\]
\end{lem}
\begin{proof}[Proof of Lemma \ref{less resonance intermediate scale}]
	The proof is similar to that of Theorem 3.4 in  \cite{SW24}.  We only discuss $\mcA_1^\si$. 
	
	Assume first  that there are $b+1$ distinct $\{(n_k,j_k)\}_{1\leq k\leq b+1}\subset [-N,N] ^{b}\times[-3N,3N]^d$ such that 
	\[|\si+n_k\cdot \omega^{(0)}+\mu_{j_k}|\leq \delta^{\frac{1}{8}}.\]
	Then all those $n_k$ are distinct. Otherwise,  if there are $n_{i_1}=n_{i_2}$ for some $1\leq {i_1}\neq {i_2}\leq b+1$, we have 
	\begin{align*}
		|\mu_{j_{i_1}}-\mu_{j_{i_2}}| &\leq |\si+n_{i_1}\cdot \omega^{(0)}+\mu_{j_{i_1}}|+|\si+n_{i_2}\cdot \omega^{(0)}+\mu_{j_{i_2}}| \\
		    &\leq 2\delta^{\frac{1}{8}}\\
			& < \frac{\pi\gamma}{(2d)^{\tau}|\log\delta|^{K\tau}}\leq \frac{\pi\gamma}{(2d)^{\tau}N^{\tau}}
	\end{align*}
	provided  $0<\delta\ll1 $.  This would  contradict \eqref{(2.3)}.  
	Next,  we claim that all  those $j_k$ are distinct. Otherwise,  if there are $j_{i_1}=j_{i_2}$ for some $1\leq {i_1}\neq {i_2}\leq b+1$, we will have 
	\begin{align*}
		|(n_{i_1}-n_{i_2})\cdot \omega^{(0)}| &\leq |\si+n_{i_1}\cdot \omega^{(0)}+\mu_{j_{i_1}}|+|\si+n_{i_2}\cdot \omega^{(0)}+\mu_{j_{i_2}}| \\
		    &\leq 2\delta^{\frac{1}{8}}. 
	\end{align*}
	This would contradict  \eqref{(2.7)}.   
	Finally, we claim that $\{j_{k}\}_{1\leq k\leq b+1}\subset \{\beta_k\}_{1\leq k\leq b}$. Otherwise,  if $j_{i_1}\notin \{\beta_k\}_{1\leq k\leq b}$ for some $1\leq i_1\leq b+1$,  since we have shown that $j_k$ are all distinct, we take one  $j_{i_2}\neq j_{i_1}$ for arbitrary $i_2\neq i_1$. Hence,
	\begin{align*}
		|(n_{i_1}-n_{i_2})\cdot \omega^{(0)}+\mu_{j_{i_1}}-\mu_{j_{i_2}}| &\leq |\si+n_{i_1}\cdot \omega^{(0)}+\mu_{j_{i_1}}|+|\si+n_{i_2}\cdot \omega^{(0)}+\mu_{j_{i_2}}| \\
		    &\leq 2\delta^{\frac{1}{8}}.
	\end{align*}
	Note that combining   $j_{i_1}-j_{i_2}\neq 0$ and $j_{i_1}\notin\{\beta_k\}_{1\leq k\leq b}$ implies $ (n_{i_1}-n_{i_2},j_{i_1},j_{i_2})\notin \{(-e_k+e_{k'},\beta_k,\beta_{k'})\}_{k,k'=1}^b$. 
	This would contradict  \eqref{(2.7)}. 
	Therfore, we must have $\{j_{i}\}_{1\leq i\leq b+1}\subset \{\beta_k\}_{1\leq k\leq b}$. However,  by the Pigeonhole principle,  there must be  $j_{i_1}=j_{i_2}$ for some $i_1\neq i_2$.  This definitely  contradicts the fact that all $j_k$ are  distinct.
\end{proof}

At this stage,   take a $\tilde{\si}\in I_0$ and define  
\[\mcA_1=\mcA_1^{\tilde{\si}},\ \mcA_2=\mcA_2^{\tilde{\si}}.\]
Since $I_0$ has a length of at most  $2\delta^{\frac{3}{8}}$,  for any $\si\in I_0$, we have $|\si-\tilde{\si}|\leq 2\delta^{\frac{3}{8}}$. If $(n,j)\notin \mcA_1\cup\mcA_2$,  we get 
\[|\tilde{\si}+\nomega^{(0)}\pm\mu_j|\geq \delta^{\frac{1}{8}}.\]
Hence, for all $\si\in I_0$ and $(n,j)\in [-N,N]^b\times[-3N,3N]^d\setminus (\mcA_1\cup\mcA_2)$,  we obtain  
\begin{align}\label{intermediate not small divisor}
	|\si+\nomega\pm\mu_j| &\geq |\tilde{\si}+\nomega^{(0)}\pm\mu_j|-|\si-\tilde{\si}| -|n\cdot(\omega-\omega^{(0)})|\\
	  \notag &\geq \delta^{\frac{1}{8}}-2\delta^{\frac{3}{8}}-b C_2 \delta |\log\delta|^K\geq \frac{1}{2}\delta^{\frac{1}{8}}. 
\end{align}
Moreover,  by Lemma \ref{less resonance intermediate scale},  we have $\#(\mcA_1\cup\mcA_2)\leq 2b$. Now take any $\tilde{\La}\in \mcR_{N_1}^{\sqrt[4]{N}}$ with $N_1\in [\sqrt[4]{N},N]$ and $\tilde{\La}\subset [-N,N]^b\times[-3N,3N]^d$. By Lemma \ref{Neumann expansion},  we have  that for all $\si \in I_0,$ 
\[\nm (R_{\tilde{\La}\setminus (\mcA_1\cup \mcA_2)}T(\si)R_{\tilde{\La}\setminus (\mcA_1\cup \mcA_2)})^{-1}  \nm\leq 4\delta^{-\frac{1}{8}}.\]
\\
Now we are going to apply the matrix-value Cartan's lemma (cf. Lemma \ref{Cartan lemma}). Since $I_0$ is of size $2\delta^{\frac{3}{8}}$, we cover  $I_0$ with disjoint intervals  $\tilde{I}=[-\frac{\xi}{2},\frac{\xi}{2}]+\si'$, and will apply Lemma \ref{Cartan lemma} on each of  enlarged   intervals $\tilde{I}'=[-\xi,\xi]+\si'$. For this, we set in Lemma \ref{Cartan lemma} with  $0<\xi\ll \delta^{\frac{3}{8}},\xi_1=1,\La=\mcA_1\cup\mcA_2,M=2b,B_2=4\delta^{-\frac{1}{8}},B_3=1,\epsilon=e^{-N_1^{\frac{3}{4}}}$. By  $\eqref{(2.4)}$ and $I_0\subset [-100N^{d+3},100N^{d+3}]$, we have 
\[|\si+\nomega^{(0)}\pm\mu_j|\le C({b,d}) N^{d+3}.\]
Hence, using the Schur's  test gives 
\[\sup_{z\in \mathcal{D}_{\xi,\xi_1}}\|\mathcal T(z)\|\le C(b,d) N^{d+3}.\]
So,  we can take $B_1=\mcO(1)|\log\delta|^{K(d+3)}$. We note that since $0<\xi \ll \delta^{\frac{3}{8}}\ll \delta^{\frac{1}{8}}|\log\delta|^{-K(d+3)}$,  the condition \eqref{mc3} holds automatically. Then applying Lemma \ref{Cartan lemma} for all possible $\tilde{\La}\in\mcR_{N_1}^{\sqrt{N}}$ (the number is at most $N^C$) implies  that there is a subset $\tilde{\Sigma}\subset \tilde{I}$ such that 
\[\frac{\meas(\tilde{\Sigma})}{\meas(\tilde{I})}\leq N^C e^{-\frac{N_1^{\frac{3}{4}}}{|\log\delta|^{1.5}}},\]
and for any $\si\notin \tilde{\Sigma}$ and $\tilde{\La}\in \mcR_{N_1}^{\sqrt[4]{N}}$ with $\tilde{\La}\subset [-N,N]^b\times[-3N,3N]^d$, we have 
\[\nm (R_{\tilde{\La}}T(\sigma)R_{\tilde{\La}})^{-1}\nm \leq e^{N_1^{\frac{3}{4}}}.\]
Since  $\tilde{I}$ disjointly pave $I_0$, we get  a subset $\tilde{\Sigma}'\subset I_0$ for each interval $I_0$ with 
\[\meas(\tilde{\Sigma}')\leq \meas(I_0) N^C  e^{-\frac{N_1^{\frac{3}{4}}}{|\log\delta|^{1.5}}}=2\delta^{\frac{3}{8}} N^C e^{-\frac{N_1^{\frac{3}{4}}}{|\log\delta|^{1.5}}}.\]

Finally, since there are approximately  $10^{b+d}N^{b+d}$ many intervals $I_0$, we obtain a subset $\tilde{\Sigma}_{N_1}\subset \R$ with 
\[\meas(\tilde{\Sigma}_{N_1})\lesssim N^{b+d+C}\delta^{\frac{3}{8}}e^{-\frac{N_1^{\frac{3}{4}}}{|\log\delta|^{1.5}}}\leq e^{-\frac{N_1^{\frac{3}{4}}}{|\log\delta|^{2}}}\]
so that, for any $\si\notin \tilde{\Sigma}_{N_1}$ and $\tilde{\La}\in \mcR_{N_1}^{\sqrt[4]{N}}$ with $\tilde{\La}\subset [-N,N]^b\times[-3N,3N]^d$, we have 
\[\nm (R_{\tilde{\La}}TR_{\tilde{\La}})^{-1}\nm \leq e^{N_1^{\frac{3}{4}}}.\]
Let $N_0=\sqrt[4]{N}.$
	We call a box $(n_1,j_1)+\La_{N_0}\in \mathcal{E}_{N_0}$,  $(n_1,j_1)\in [-N,N] ^{b}\times[-3N,3N]^d $ is good if 
	\begin{equation*}
		\nm (R_{(n_1,j_1)+\La_{N_0}}TR_{(n_1,j_1)+\La_{N_0}})^{-1} \nm \leq e^{N_0^{\frac{9}{10}}},
	\end{equation*}
	and for any $(n, j)$ and $(n', j')$ such that $|n-n^\prime|+|j-j'|\geq \sqrt{N_0}$, 
	\begin{equation*}
		|(R_{(n_1,j_1)+\La_{N_0}}TR_{(n_1,j_1)+\La_{N_0}})^{-1} (n,j;n^\prime,j')| \leq e^{-(c_1-N^{-\vartheta})(|n-n^\prime|+|j-j'|)}. 
	\end{equation*}
Otherwise, we call $(n_1,j_1)+\La_{N_0}\in [-N,N] ^{b}\times[-3N,3N]^d $ is bad. By Lemma \ref{less resonance intermediate scale}, \eqref{intermediate not small divisor} and  Lemma \ref{Neumann expansion},  there are at most $2b$ disjoint bad boxes of size $N_0=N^{\frac{1}{4}}$ contained in $[-N,N]^b\times [-3N,3N]^d$, and hence, 
\[2b\leq\frac{N^{\frac{3}{8}}}{N^{\frac{1}{4}}}.\]
This is exactly  the sublinear bound needed  in Theorem \ref{MSA1}, where  we set $\zeta=\frac{3}{8}<\frac{1}{2},\xi=\frac{1}{4}$. Moreover, if we take 
\[\si\notin \bigcup_{N_1\in[\sqrt[4]{N},N]}\tilde{\Sigma}_{N_1},\]
then for any $N_1\in [\sqrt[4]{N},N]$ and any $\tilde{\La}\in\mcR_{N_1}^{\sqrt[4]{N}}$ contained in $[-N,N]^b\times[-3N,3N]^d$,  we have 
\[\nm (R_{\tilde{\La}}TR_{\tilde{\La}})^{-1}\nm \leq e^{N_1^{\frac{3}{4}}}\leq e^{N_1^{\frac{9}{10}}}.\]
Applying  Theorem \ref{MSA1} will ensure that, for any $\si\notin \cup_{N_1\in[\sqrt[4]{N},N]}\tilde{\Sigma}_{N_1}$,  the  $(c_N,N)$-LDT holds for $c_N=c_1-N^{-\vartheta}>\frac{c_1}{2}$. 
Therefore,
\[\Sigma_N\subset \bigcup_{N_1\in[\sqrt[4]{N},N]}\tilde{\Sigma}_{N_1},\]
and hence,
\begin{align*}
	\meas(\Sigma_N)  &\leq \meas(\bigcup_{N_1\in[\sqrt[4]{N},N]}\tilde{\Sigma}_{N_1})\\
	&\leq\sum_{\sqrt[4]{N}\leq N_1\leq N}e^{-\frac{N_1^{\frac{3}{4}}}{|\log\delta|^2}}\\
	&\leq\sum_{\sqrt[4]{N}\leq N_1\leq N}e^{-N_1^{\frac{3}{4}-\frac{1}{5}}}\leq e^{-N^{\frac{1}{30}}}.
\end{align*}

\subsection{LDT in large scales}
Finally,  we prove Theorem \ref{LDT} for the large scales $N\geq |\log\delta|^K$ via the MSA induction. Recall that,  we choose  $K=K_1^{100},K_2=K_1^5$ to be sufficiently large, and  the frequency   $\omega$  satisfies \eqref{Diophantine on omega frequency} and \eqref{weak second Melnikov}.

	For any $N_3>|\log\delta|^K$, we take $N_1, N_2$ satisfying $$N_2=N_1^{K_1},\ N_2^{K_1}\leq N_3\leq N_2^{2K_1}.$$  
	We assume LDT holds for  scales $N_1, N_2$, and aim to prove LDT holds at the scale $N_3.$\smallskip

	\textbf{Step 1.  Number  on bad $N_1$-size boxes near  $j=0$}\\
	We first  consider  the case when $|j_1|\leq 2N_1$.  By the LDT at scale $N_1$,  there exists a set $\Sigma_{N_1}$ with $\meas(\Sigma_{N_1})\leq e^{-N_1^{\frac{1}{30}}}$ such that,  for all $\si\notin\Sigma_{N_1}$ and any $\La_{N_1}\in \mcE_{N_1}^0$, $\La_{N_1}(j_1)$ is good (in the sense of  having good control  on Green's functions as in LDT). Since the operator is T\"oplitz with respect to 
	$n\in\Z^b$,  we have  that if  $\si+n_1\cdot \omega\notin \Sigma_{N_1}$,  then  $(n_1,j_1)+\La_{N_1}$ is  good. Now,  observe that 
	\[\nm \cdot \nm\leq \nm\cdot\nm_{HS}, \]
	where $\nm\cdot\nm_{HS}$ is the standard  Hilbert-Schmidt norm. If we replace the norm \eqref{LDT L2 norm} in LDT with  
	\begin{equation}\label{Hilbert-Schmidt}
		\nm (R_{\La_N(j_0)}T(\sigma)R_{\La_N(j_0)})^{-1} \nm_{HS} \leq e^{N^{\frac{9}{10}}},
	\end{equation} 
	then  
	\[\Sigma_{N_1}=\{\si:\ {\rm  for \ some} \ \La_{N_1}(j_1), \ \eqref{LDT L2 norm} \ {\rm or} \ \eqref{LDT off diagonal} \ {\rm fails\ with}\  N=N_1, j_0=j_1\}\]
	is a subset of 
	\[\Sigma_{N_1}'=\{\si:\ {\rm for \ some} \ \La_{N_1}(j_1), \ \eqref{Hilbert-Schmidt} \ {\rm or} \ \eqref{LDT off diagonal} \ {\rm fails\ with}\ N=N_1, j_0=j_1\}.\]
	Notice that  we also have 
	\[\nm (R_{\La_{N_1}(j_0)}T(\sigma)R_{\La_{N_1}(j_0)})^{-1} \nm_{HS}\leq  (\#\La_{N_1}) \nm (R_{\La_{N_1}(j_0)}T(\sigma)R_{\La_{N_1}(j_0)})^{-1} \nm \]
	and $(\#\La_{N_1})e^{N_1^{\frac{17}{20}}}\leq e^{N_1^{\frac{9}{10}}}$. Thus,  if we replace the upper bound in \eqref{LDT L2 norm} by $e^{N_1^{\frac{17}{20}}}$ and obtain a set  $\Sigma_{N_1}''$ such that
	\[\Sigma_{N_1}'\subset \Sigma_{N_1}'',\quad \meas(\Sigma_{N_1}'')\leq e^{-N_1^{\frac{1}{30}}}.\]
	Therefore, we only need to ensure that $\si+n_1\cdot \omega\notin \Sigma_{N_1}'$. By Cramer's rule,  we  know  that $\Sigma_{N_1}'$ is  a semi-algebraic set of degree at most 
	$N_1^C$. Moreover, Lemma \ref{basu} tells us that $\Sigma_{N_1}'$ has at most $N_1^C$ many connected components, each with measure less than $e^{-N_1^{\frac{1}{30}}}$. In conclusion, there exist $N_1^C$ intervals $I_k$ of size $e^{-N_1^{\frac{1}{30}}}$, such that $\Sigma_{N_1}\subset\cup_k I_k$. 
	
	However, the Diophantine condition \eqref{Diophantine on omega frequency} of $\omega$  implies  that for any 
	nonzero $n$ with $|n|\leq 2N_3$, we have 
	\[|\nomega|\geq e^{-(2N_3)^{\frac{1}{K_2}}}\geq e^{-2N_1^{\frac{2K_1^2}{K_2}}\gg e^{-N_1^{\frac{1}{30}}}}.\]
	Therefore, for any $|j_1|\leq 2N_1$, there is at most one $|n_1|\leq N_3$ such that $\si+n_1\cdot\omega\in I_k$. This shows  that there are  at most $N_1^{C(b)}\ll N_3^{1-}$  many  bad $N_1$-boxes  in this case, provided $K_1(b)\gg1$. \smallskip
  
	\textbf{Step 2. Number on  bad $N_1$-size   boxes away from  $j=0$}\\
	When $|j_1|\geq 2N_1$, we will show that there are most two disjoint bad boxes of size $N_1$. First, if a box $(n,j)+\La_{N_1}$ is bad,  then using  Lemma \ref{Neumann expansion} and \eqref{perturbation Op} (whose $\max\{|j|,|j'|\}$-term ensures $\epsilon_2=C_1\delta e^{-2c_1N_1}$ for applying Lemma \ref{Neumann expansion})  shows  that there must be some $(n_1,j_1)\in (n,j)+\La_{N_1}$, so that either 
	\[|\si+n_1\cdot\omega+\mu_{j_1}|\leq 2e^{-N_1^{\frac{9}{20}}},\]
	or
	\[|\si+n_1\cdot\omega-\mu_{j_1}|\leq 2e^{-N_1^{\frac{9}{20}}}.\]
	Otherwise,   there are indeed more than three bad $N_1$ size boxes. By the Pigeonhole principle and without loss of generality,  we can assume there are two different $(n_1,j_1),(n_2,j_2)\in [-N_3,N_3]^b\times[-3N_3,3N_3]^d$ such that they are all from  the $D_+$ sector, i.e.,  
	\[|\si+n_1\cdot\omega+\mu_{j_1}|\leq 2e^{-N_1^{\frac{9}{20}}},\ |\si+n_2\cdot\omega+\mu_{j_2}|\leq 2e^{-N_1^{\frac{9}{20}}}.\]
	But this gives 
	\begin{align*}
		4e^{-N_1^{\frac{9}{20}}} &\geq |\si+n_1\cdot\omega+\mu_{j_1}|+|\si+n_2\cdot\omega+\mu_{j_2}| \\
		      &\geq |(n_1-n_2)\cdot\omega+\mu_{j_1}-\mu_{j_2}| \geq  e^{-N_3^{\frac{1}{K_2}}},
	\end{align*}
	which contradicts the weak second Melnikov's condition \eqref{weak second Melnikov}.\smallskip

		\textbf{Step 3.  The application of  Lemma \ref{Cartan lemma}}\\
	To apply Lemma \ref{Cartan lemma}, one must verify the prior  estimate condition \eqref{mc3} via  LDT at an additional scale $N_2$.  So,  let $\tilde{\Sigma}_{N_2}\subset \R$ be such that for some $(n,j)\in [-N_3,N_3]^b\times[-3N_3,3N_3]^d$,  either $|\si+\nomega+\mu_j|\leq 2e^{-N_2^{\frac{9}{20}}}$ or $|\si+\nomega-\mu_j|\leq 2e^{-N_2^{\frac{9}{20}}}$. Clearly,  we have 
	\[\meas(\tilde{\Sigma}_{N_2})\leq N_3^C e^{-N_2^{\frac{9}{20}}}\leq \frac{1}{2}e^{-N_2^{\frac{1}{31}}}.\]
	If  $\si\notin\tilde{\Sigma}_{N_2}$, combining Lemma \ref{Neumann expansion} and  the same argument  in  \textbf{Step 2}  shows that  for all $(n_1,j_1)\in [-N_3,N_3]^b\times[-3N_3,3N_3]^d$ with $|j_1|\geq 2N_2$ and all $\La_{N_2}\in \mcE^0_{N_2}$, $(n_1,j_1)+\La_{N_2}$ contained in $[-N_3,N_3]^b\times[-3N_3,3N_3]^d$ are good. 
	
	At the same time, by LDT at the scale $N_2$,  we have the corresponding set $\Sigma_{N_2}$. Let 
	\[\hat{\Sigma}_{N_2}=\{\si:\ {\rm for\ some}\ n\in [-N_3,N_3]^b,\si+\nomega\in \Sigma_{N_2}\}.\]
	Then by  LDT, for any $\si\notin \hat{\Sigma}_{N_2},(n_1,j_1)\in [-N_3,N_3]^b\times[-3N_3,3N_3]^d,|j_1|\leq 2N_2$ and $\La_{N_2}\in \mcE^0_{N_2}$, we have $(n_1,j_1)+\La_{N_2}\subset [-N_3,N_3]^b\times[-3N_3,3N_3]^d$ is good. Thus,   
	\[\meas(\hat{\Sigma}_{N_2})\lesssim N_3^b \meas(\Sigma_{N_2})\leq \frac{1}{2}e^{-N_2^{\frac{1}{31}}}.\]
	Hence, we have proven that for all $\si\notin \tilde{\Sigma}_{N_2}\cup \hat{\Sigma}_{N_2}$, the  $N_2$ size  blocks contained  in $[-N_3,N_3]^b\times[-3N_3,3N_3]^d$ are all good, and 
	\[\meas(\tilde{\Sigma}_{N_2}\cup \hat{\Sigma}_{N_2})\leq e^{-N_2^{\frac{1}{31}}}.\]
	Applying  Lemma \ref{MSA L2} for every scale $N\in [N_3^\frac{1}{4},N_3]$ and $\tilde{\La}\in \mcR_{N}^{N_3^{\frac{1}{4}}}$, we obtain  that for all 
	$\si\notin \tilde{\Sigma}_{N_2}\cup \hat{\Sigma}_{N_2}$,  
	\[\nm (R_{\tilde{\La}}TR_{\tilde{\La}})^{-1}\nm\leq 4(N_2+1)^{b+d}e^{N_2^{\frac{9}{10}}}.\]
	
	Now we can apply Lemma \ref{Cartan lemma} on $\tilde{\La}$.  Due to \eqref{(2.4)} abd $|\nomega+\mu_j|\lesssim N_3^{d+3}$,  we only need to consider $\si\in [-N_3^{d+4},N_3^{d+4}]$. Pave it with size $0<\xi\ll e^{-N_3^{\frac{1}{30}}}$ intervals. Take $\xi_1=1,B_1=N^{d+4}$. The conclusions of \textbf{Step 1, Step 2} show  that there are at most $N_1^{C(b)}$ many disjoint $N_1$-size bad boxes in $\tilde{\La}$, and applying Lemma \ref{MSA L2}  gives  us $M=N_1^C$ and $B_2=4(N_1+1)^{b+d}e^{N_1^{\frac{9}{10}}}$. Finally, 
	take $B_3=4(N_2+1)^{b+d}e^{N_2^{\frac{9}{10}}}$
	and note that
	\[e^{-N_2^{\frac{1}{31}}}\ll 10^{-3}\xi_1 (1+B_1)^{-1}(1+B_2)^{-1}.\]
	Then the condition \eqref{mc3} holds true. Taking further  $\epsilon=e^{N^{\frac{9}{10}}}$ yields  a set $\tilde{\Sigma}_{\tilde{\La}}$ with 
	\[\meas(\tilde{\Sigma}_{\tilde{\La}})\leq N^{d+4} e^{-\mcO(1)\frac{N^{\frac{9}{10}}}{N_1^C N_2^{\frac{9}{10}}}}\leq e^{-N_3^{\frac{1}{5}}}.\]
	Taking account of  all $\tilde{\Sigma}_{\tilde{\La}}$ for all $N$ and all $\tilde{\La}$ leads to a set $\tilde{\Sigma}_{N_3}$ with 
	\[\meas(\tilde{\Sigma}_{N_3})\leq N_3^Ce^{-N_3^{\frac{1}{5}}}\leq e^{-N_3^{\frac{1}{4}}},\]
    and for any $\si\notin \tilde{\Sigma}_{N_3}$,  $N\in [N_3^{\frac{1}{4}},N_3]$, each $\tilde{\La}\in \mcR_{N}^{N_3^{\frac{1}{4}}}$ contained  in $[-N_3,N_3]^b\times[-3N_3,3N_3]^d$ is good. \smallskip
    
    	\textbf{Step 4. Off-diagonal exponential decay  estimate}\\
	We have already established the sub-exponential growth bound required in Theorem  \ref{MSA1} for $\sigma$ outside the set $\tilde{\Sigma}_{N_3}$.  To establish the off-diagonal decay, it remains to verify the sublinear bound. For this purpose, let $N_0=N_3^{\frac{1}{4}}$. Since there are at most $N_1^{C(b)}$ many bad boxes of size $N_1$ and $N_0\gg N_1$, applying the resolvent identity implies that there are at most $N_1^C$ many  size $N_0$ disjoint bad boxes (with exponential decay rate being  $c_{N_1}-N_1^{-\vartheta}$, because the resolvent identity needs  the off-diagonal exponential decay of Green's function at the  scale $N_1$) contained in $[-N_3,N_3]^b\times[-3N_3,3N_3]^d$. Moreover, we have 
	\[N_1^C\leq \frac{N_3^{\frac{3}{8}}}{N_3^{\frac{1}{4}}}\]
	by the choice of $N_1,N_2,N_3$ and $K_1$. This inequality provides the sublinear bound required by Theorem \ref{MSA1}. Applying  Theorem \ref{MSA1} then leads to $\Sigma_{N_3}\subset \tilde{\Sigma}_{N_3}$ with  
	\[\meas(\Sigma_{N_3})\leq\meas(\tilde{\Sigma}_{N_3})\leq e^{-N_3^{\frac{1}{30}}},\]
	and the $(c_{N_3},N_3)$-LDT holds for $c_{N_3}=c_{N_1}-N_1^{-\vartheta}$.\smallskip
	
	\textbf{Step 5. Completion of the induction}\\
	By the above arguments, we initially take $N_2^{K_1}=\hat{N}=|\log\delta|^K$ and prove LDT in scale intervals  $[\hat{N},\hat{N}^2]$, then in  $[\hat{N}^2,\hat{N}^4]$, in  $[\hat{N}^4,\hat{N}^8]$, and so on. Since  the  induction scales  grow  like $\hat{N}^{2^s}$ and $N_1\sim N_3^{\frac{1}{K_1^2}}$,  direct  calculation implies  that for $N_3\in [\hat{N}^{2^s},\hat{N}^{2^{s+1}}]$,  the decay rate  $c_{N_3}$ can be 
	\begin{equation}\label{c_N_3 decay}
		c_{N_3}=c_1-\mcO(1)\sum_{l=1}^{s}\frac{1}{|\log\delta|^{\vartheta K 2^l}}\geq c_1-\mcO(1)\sum_{s=1}^{\infty}\frac{1}{|\log\delta|^{\vartheta K 2^s}} \geq \frac{c_1}{2} 
	\end{equation}
	We finish the proof of Theorem \ref{LDT} in large scales, and thus the whole LDT.

\begin{rmk}
	
	(1) We summarize the constant $c_N$ in $(c_N,N)$-LDT as follows.
	   \begin{itemize}
		\item In small scales,  $N\leq |\log\delta|^{10}$, take $c_N=\frac{9}{10}c_1\leq c_1-N^{\frac{9}{20}}$.
		\item In intermediate scales,  $|\log\delta|^{10}<N\leq |\log\delta|^K$, take $c_N= c_1-N^{-\vartheta}$.
		{\item In large scales $N> |\log\delta|^{K}$,   $c_N=c_1-\sum_{j>|\log\delta|^K}^{N}\gamma_j$ with $\sum_{j>|\log\delta|^K}^{\infty}\gamma_j<\frac{c_1}{2}$. (Many $\gamma_j$ may be equal to zero by \eqref{c_N_3 decay}.)}
	   \end{itemize}
	(2)  From \eqref{Diophantine condition} and \eqref{weak second Melnikov} in Theorem \ref{LDT}, in order for  LDT to hold at all scales, it suffices to remove a set of  $\omega$ with measure  at most  
	\[e^{-\frac{1}{2}|\log\delta|^{\frac{K}{K_2}}}\leq e^{-|\log\delta|^{K_1^{50}}}\ll \delta.\]
	which suffices for the following nonlinear analysis.
\end{rmk}

\section{Nonlinear analysis:  Proof of Theorem \ref{Main Thm}}\label{NLsect}
Fix $\bfa\in\DC$.  In the following, we always assume  $0< \delta\leq \tilde{\delta} $,   $0<\varepsilon\leq \tilde{\varepsilon}(\delta)$ and $\theta\in \Theta_{\bfa}$  such that both Theorem $\ref{seperative spectrum}$ and Theorem \ref{LDT} hold true.  Recalling  \eqref{phiidef} and \eqref{muidef}, and since  we set $\eta=\frac12$ in Theorem \ref{KPS}, we have 
\begin{equation}\label{Eigenfunction in Nonlinear analysis}
	\nm\phi_j\nm_{\ell^2(\Z^d)}=1,\ |\phi_j(j)-1|<\varepsilon^{\frac{1}{2}},\ |\phi_j(x)|<\varepsilon^{\frac{1}{2}|x-j|_1}=e^{-\kappa|x-j|_1} \ {\rm for}\  {\forall}j,x\in \Z^d,
\end{equation}
where  $\kappa=\frac{1}{2}|\log\varepsilon|$.

In the following, we will introduce the CWB scheme in the present context, and prove our main theorem. 

\subsection{Lyapunov-Schmidt decomposition}
We aim to  construct  solutions  for  NLS \eqref{NLS}  of  the  form 
\begin{equation}\label{QP solu}
	u(t,x)=\sum_{(n,j)\in \Z^b\times\Z^d}\hat{u}(n,j)e^{in\cdot\omega t}\phi_{j}(x),
\end{equation}
which should be perturbations  of  solutions  \eqref{solu to LS}  for  \eqref{LS} with  the  amplitude  $\mathbf{a}=(a_k)_{k=1}^b$.  
Substituting  $(\ref{QP solu})$ into NLS $(\ref{NLS})$ will lead to the  system of nonlinear equations on $\Z^b\times\Z^d$
\begin{equation}\label{lattice nonlinear equation}
	(n\cdot\omega+\mu_j)\hat u(n, j)+ \delta W_{\hat u}(n, j)=0, \  (n, j)\in\mathbb Z^{b}\times\mathbb Z^d, 
\end{equation}
where 
\begin{align*}
	W_{\hat u} (n, j)=&\sum_{n'+\sum_{m=1}^p (n_m-n_m')=n \atop {n', n_m,n_m'\in\Z^b}} \sum_{\ j', j_m,j_m'\in\Z^d } \hat{u}(n',j')\prod_{m=1}^p\hat {u}(n_m,j_m) \overline{\hat{u}(n'_m,j'_m) }\nonumber\\
	     &\ \ \ \times\left(\sum_{x\in\Z^d}\phi_{j}(x)\phi_{j'}(x)\prod_{m=1}^p\phi_{j_m} (x) {\phi_{j_m'}(x)}\right),
\end{align*}
with $\overline{(\cdot)}$ denoting the complex conjugation. 
Note that  the form $(\ref{QP solu})$ are closed under multiplication and complex conjugation.
To simplify the notation, we write $u$ for $\hat u$, namely,   $u(n, j)=\hat u(n, j)$.  Let $v$ be the complex conjugation  of $u$:   $v(n, j)=\overline{{\hat u}(-n, j)}$. Then  we can rewrite   $W_u$  as 
\begin{align}\label{Lattice nonlinear term}
	W_{u} (n, j)=&\sum_{n'+\sum_{m=1}^p (n_m+n_m')=n \atop {n', n_m,n_m'\in\Z^b}} \sum_{\ j', j_m,j_m'\in\Z^d } u(n',j')\prod_{m=1}^p u(n_m,j_m) v(n'_m,j'_m) \nonumber\\
	     &\ \ \ \times\left(\sum_{x\in\Z^d}\phi_{j}(x)\phi_{j'}(x)\prod_{m=1}^p\phi_{j_m} (x) {\phi_{j_m'}(x)}\right).
\end{align}
Taking the conjugate of system (\ref{lattice nonlinear equation}) yields
\begin{equation}\label{conjugation lattice nonlinear equation}
	(-n\cdot\omega+\mu_j) v(n, j)+ \delta \widetilde{W}_{u}(n, j)=0, \, (n, j)\in\mathbb Z^{b}\times\mathbb Z^d,
\end{equation}
where 
\begin{align}\label{conjugation Lattice nonlinear term}
	\widetilde{W}_{u} (n, j)=\overline{W_u(-n,j)}=&\sum_{n'+\sum_{m=1}^p (n_m+n_m')=n \atop {n', n_m,n_m'\in\Z^b}} \sum_{\ j', j_m,j_m'\in\Z^d } v(n',j')\prod_{m=1}^p v(n_m,j_m) u(n'_m,j'_m) \nonumber\\
	     &\ \ \ \times \left(\sum_{x\in\Z^d}\phi_{j}(x)\phi_{j'}(x)\prod_{m=1}^p\phi_{j_m} (x) {\phi_{j_m'}(x)}\right).
\end{align}
In fact, $W_u(n, j),\widetilde{W}_{u} (n, j)$ are functions of $u$ and $v$ (can view them as vectors in $\C^{\Z^{b+d}}$).  Since  $v$ is the conjugation of $u$,  we only indicate the dependence on $u$ for simplicity.
Combining  $(\ref{lattice nonlinear equation})$ and $(\ref{conjugation lattice nonlinear equation})$ shows 
\begin{align}\label{Lattice equation system}
	\begin{split}
		\left\{
			\begin{array}{l}
				(D_+u)(n,j)+\delta W_u(n,j)=0, \\
				(D_-v)(n,j)+\delta \widetilde{W}_u(n,j)=0,
			\end{array}\right.
	\end{split}
\end{align}
where $D_{\pm}$ are the diagonal  matrices with entries
\[D_{\pm}(n,j;n',j')=(\pm n\cdot \omega +\mu_j)\delta_{(n,j),(n',j')}.\]

Now view $y=(u,v)$ as functions on $\{\pm\}\times \Z^{b}\times\Z^d$, with 
\[y(+,n,j)=u(n,j),\ y(-,n,j)=v(n,j).\]
Define the resonant set 
\begin{equation}\label{resonant set}
	\mcS=\{(+,-e_k,\beta_k),(-,e_k,\beta_k):\ k=1,2,\cdots,b \}
\end{equation}
and denote 
\[\mcS^c=(\{\pm\}\times\Z^{b+d})\setminus \mcS. \]
Write $D$ for the diagonal matrix  with   diagonal blocks being  $D_\pm$ and write \eqref{Lattice equation system} in the form
$F(y)=F(u,v)=0$. Since $v$ is the conjugation  of $u$, we simply write $F(u)=0$ for $F(u,v)=0$. 

By using   Lyapunov-Schmidt decomposition (for more details, cf. Sect. 2 of \cite{Bou98}), we can divide   $F(u)=0$ into the P-equation 
\begin{equation}\label{P equation}
	F(u)\big|_{\mcS^c}=0,
\end{equation}
and the Q-equations 
\begin{equation}\label{Q equation}
	F(u)\big|_{\mcS}=0.
\end{equation}
The remaining content of this section is devoted to solving both  P-equation and Q-equation to get the solution $u$ (and $\omega$). Here,  we emphasize that, we  always fix the values  of $(u,v)$ on the resonant set to be  
\begin{equation}\label{Fixed amplitude on resonant set}
	u(+,-e_k,\beta_k)=v(-,e_k,\beta_k)=a_k,\ 1\leq k\leq b. 
\end{equation}

\subsubsection{The P-equations}
The P-equation \eqref{P equation} is of  infinite dimensional and will be solved by a Newton scheme as in \cite{Bou98}, starting from the initial approximation $u^{(0)}=u_0$. To do this, we first linearize the system $F(u)$ 
\[F(u+\Delta_{cor}\begin{pmatrix} u\\v\end{pmatrix})\big|_{\mcS^c}=F(u)\big|_{\mcS^c}+F'_{\mcS^c}(u)\cdot \Delta_{cor}\begin{pmatrix} u\\v\end{pmatrix} +\mcO(\nm \Delta_{cor}\begin{pmatrix} u\\v\end{pmatrix}\nm^{2-})\]
Then the linearized equation will become
\[\Delta_{cor} \begin{pmatrix} u\\v\end{pmatrix}=-[F'_{\mcS^c}(u)]^{-1} F(u)|_{\mcS^c},\]
where the linearized operator $F'(u)$ at $u$  defined on $\ell^2(\{\pm\}\times \Z^{b+d})$ has the form
\[F'(u)=D+\delta \mcW_{u}\]
with 
\begin{equation}\label{mcW}
	\mathcal W_u= \begin{pmatrix} \frac{ \partial{W_u}}{\partial u}&  \frac{ \partial{W_u}}{\partial v}\\  \frac{ \partial{\widetilde{W}_u}}{\partial u} & \frac{ \partial{\widetilde{W}_u}}{\partial v}  \end{pmatrix}=\begin{pmatrix} \mcW_{+,+} &  \mcW_{+,-}\\  \mcW_{-,+} & \mcW_{-,-}  \end{pmatrix}.
\end{equation}
For example, direct computation shows 
\begin{align} \label{Explicit mcW}
	\mcW_{+,+} (n,j;n',j') &= \frac{ \partial{W_u}}{\partial u}(n,j;n',j') \\
	&=(p+1) \sum_{\sum_{m=1}^p (n_m+n_m')=n-n' \atop { n_m,n_m'\in\Z^b}} \sum_{\  j_m,j_m'\in\Z^d } \prod_{m=1}^p u(n_m,j_m) v(n'_m,j'_m) \nonumber\\
	\notag &\ \ \ \times \left(\sum_{x\in\Z^d}\phi_{j}(x)\phi_{j'}(x)\prod_{m=1}^p\phi_{j_m} (x) {\phi_{j_m'}(x)}\right).
\end{align}

\begin{prop}\label{generate the nonlinear operator}
	\begin{itemize}
	\item[(1)]    For any $j,j'\in \Z^d,k,n,n'\in \Z^b,r,r'\in\{\pm\}$,  we have 
	\[\mcW_{r,r'}(n,j;n',j')=\mcW_{r,r'}(n+k,j;n'+k,j').\]
	\item [(2)]  Assume that $|u(n,j)|\leq e^{-c(|n|+|j|)}$ and $0<c\ll\kappa=\frac{1}{2}|\log\varepsilon|$. Then 
	\[|\mcW_{r,r'}(n,j;n',j')|\leq C_1(|n-n'|+1)^{C_1} e^{-c|n-n'|-\frac{\kappa}{2}|j-j'|-2pc\max\{|j|,|j'|\}},\ C_1=C_1(p,b)>1.\]
\end{itemize}
\end{prop}
\begin{proof}
	(1) This follows directly from the definition of $\mcW_u$, see e.g., the form \eqref{Explicit mcW}.

	(2)  Since $v(n,j)=\overline{u(-n,j)}$, we also have $|v(n,j)|\leq e^{-c(|n|+|j|)}$. From the definition of  $\mcW_{r,r'}$ (cf. e.g.,   \eqref{Explicit mcW}), we have
	\begin{align*}
		|\mcW_{r,r'}(n,j;n',j')| &\le C(p) \sum_{\sum_{m=1}^p (n_m+n_m')=n-n' \atop { n_m,n_m'\in\Z^b}} \sum_{\  j_m,j_m'\in\Z^d }  e^{-c\sum_{m=1}^{p}(|n_m|+|j_m|+|n'_m|+|j'_m|)}\nonumber\\
		 &\ \ \ \times \left|\sum_{x\in\Z^d}\phi_{j}(x)\phi_{j'}(x)\prod_{m=1}^p\phi_{j_m} (x) {\phi_{j_m'}(x)}\right|:=A\times B, 
	\end{align*}
	where 
	\[A=\sum_{\sum_{m=1}^p (n_m+n_m')=n-n' \atop { n_m,n_m'\in\Z^b}} e^{-c\sum_{m=1}^{p}(|n_m|+|n'_m|)} \]
	and 
	\[B=\sum_{ j_m,j_m'\in\Z^d }  e^{-c\sum_{m=1}^{p}(|j_m|+|j'_m|)} \left|\sum_{x\in\Z^d}\phi_{j}(x)\phi_{j'}(x)\prod_{m=1}^p\phi_{j_m} (x) {\phi_{j_m'}(x)}\right|.\]
    
    Indeed, for $A$, we have   
	\begin{align*}
		A &= \sum_{n_1,\cdots,n_p,n'_1, \cdots,  n'_{p-1}\in\Z^b} e^{-c(|n_1|+\cdots +|n_p|+|n'_1|+\cdots|n'_{p-1}|+|n-n'-n_1-\cdots-n'_{p-1}|)}\\
		  &\leq \bigg(\sum_{|n_1|,\cdots,|n'_{p-1}|\leq |n-n'|} + \sum_{|n_1|> |n-n'| \atop n_2,\cdots, n'_{p-1}\in\Z^b}+\sum_{|n_2|> |n-n'| \atop n_1,\cdots, n'_{p-1}\in\Z^b}+\cdots +\sum_{|n'_{p-1}|> |n-n'| \atop n_1,\cdots, n'_{p-2}\in\Z^b} \bigg)\\
		    & \qquad\cdot e^{-c(|n_1|+\cdots +|n_p|+|n'_1|+\cdots|n'_{p-1}|+|n-n'-n_1-\cdots-n'_{p-1}|)}\\
		  &\leq (2|n-n'|+1)^{(2p-1)b}e^{-c|n-n'|} +(2p-1) \big(  \sum_{ k\in \Z^b}e^{-c|k|} \big)^{2p-2}\cdot\sum_{|k|>|n-n'|}e^{-c|k|} \\
		  &\leq C(1+|n-n'|)^{C}e^{-c|n-n'|}
	\end{align*}
	with some $C=C(p,b)>1$. Moreover, by \eqref{Eigenfunction in Nonlinear analysis}, we have 
	\begin{align*}
		B &\leq \sum_{ j_m,j_m'\in\Z ^d}  e^{-c\sum_{m=1}^{p}(|j_m|+|j'_m|)}    \sum_{x\in\Z^d }e^{-\kappa|x-j|-\kappa|x-j'|-\kappa\sum_{m=1}^{p}(|x-j_m|+|x-j'_m|)}
	\end{align*}
	and 
	\begin{align*}
		\sum_{j\in \Z^d} e^{-c|j|-\kappa|x-j|} &\leq \sum_{j\in \Z^d} e^{-c(|j|+|x-j|)-\frac{\kappa}{2}|x-j|}\\
		   &\leq \sum_{j\in \Z^d} e^{-c|x|-\frac{\kappa}{2}|x-j|}\\
		   & = C(\varepsilon,d)e^{-c|x|}.
	\end{align*}
	The first inequality follows from  $c\ll\kappa$; the last equality is due to $\kappa=\frac{1}{2}|\log\varepsilon|,C(\varepsilon,d)=\sum_{j\in\Z^d}e^{-\frac{\kappa}{2}|j|}>1$. In addition,  since we have  already assumed  $0<\varepsilon\leq \tilde{\varepsilon}(\delta)$ for $0<\delta\leq \tilde{\delta}\ll1$, we can further reduce  $\tilde{\delta}$ to ensure that $\varepsilon$ is sufficiently small, such that $C(\varepsilon,d)\leq 2$. Hence,  we have   (without loss of generality, assume $|j|=\max\{|j|,|j'|\}$)
	\begin{align*}
		B &\leq 2 \sum_{x\in\Z^d}e^{-2pc|x|-\kappa|x-j|-\kappa|x-j'|} \\
		 &\leq 2  e^{-\frac{\kappa}{2}|j-j'|} \sum_{x\in\Z^d}e^{-2pc|x|-\frac{\kappa}{2}|x-j|-\frac{\kappa}{2}|x-j'|}\\
		 & \leq 2 e^{-\frac{\kappa}{2}|j-j'|} \sum_{x\in\Z^d}e^{-2pc|x|-3pc|x-j|}\\
		 &\leq 2 e^{-\frac{\kappa}{2}|j-j'|-2pc|j|}\sum_{x\in\Z^d}e^{-pc|x-j|}\\
		 &\leq 2 e^{-\frac{\kappa}{2}|j-j'|-2pc|j|}.
	\end{align*}
	
Combining the above estimates  for $A$ and $B$ completes the proof.
\end{proof}

When solving the P-equation, we must keep in mind that the unknown variables are $\mathbf{a}$ and $\omega$, and our goal is to determine $u^{(r+1)}$ from  $u^{(r)}$. In this Newton scheme, we initially choose $\omega$ close to $\omega^{(0)}$, and the $0$-th approximation $u^{(0)}$ satisfies
\[u^{(0)}(-e_k,\beta_k)=a_k,k=1,2,\cdots,b;\  u^{(0)}(n,j)=0\ {\rm otherwise}.\]
 Since we seek solutions close to $(u^{(0)},v^{(0)})$, which has compact support in $\{\pm\}\times\Z^{b+d} $, we adopt a multi-scale Newton scheme as follows. 
Let $M$ be a large constant which will be determined later. By the choice of $(u^{(0)},v^{(0)})$, we already have $(u^{(0)},v^{(0)})\big|_{\mcS^c}=0$. Assume we have already obtained the $r$-th approximation $u^{(r)}=u^{(r)}(\omega,\mathbf{a})$, which is supported in
\[(\{\pm\}\times[-M^r,M^r]^{b+d})\setminus \mcS\subset \{\pm\}\times\Z^{b+d}.\]
At iteration step $(r+1)$, choose an appropriate scale $N=M^{r+1}$ and denote by $F'_N$  the restriction of $F'$  on   
\[(\{\pm\}\times[-N,N]^{b+d})\setminus \mcS\subset \{\pm\}\times\Z^{b+d},\]
where  $F'_N=F'_N(u^{(r)})$ is  evaluated at $(u^{(r)},v^{(r)})$. Moreover, denote by  $F_N$ the restriction of  $F$  on 
\[(\{\pm\}\times[-N,N]^{b+d})\setminus \mcS\subset \{\pm\}\times\Z^{b+d}.\]
Then  the $(r+1)$-th correction is 
\[\Delta_{cor} \begin{pmatrix} u^{(r+1)}\\v^{(r+1)}\end{pmatrix}=-[F'_{N}(u^{(r)})]^{-1} F_N(u^{(r)}),\ N=M^{r+1}\]
and the $(r+1)$-th approximation is
\[u^{(r+1)}(\omega,\mathbf{a})=u^{(r)}(\omega,\mathbf{a})+\Delta_{cor} u^{(r+1)}(\omega,\mathbf{a}),\]
\[v^{(r+1)}(\omega,\mathbf{a})=v^{(r)}(\omega,\mathbf{a})+\Delta_{cor} v^{(r+1)}(\omega,\mathbf{a}), \]
for all $r=0,1,2, \cdots$.

At  the $(r+1)$-th step,  $\mathcal W_{u^{(r)} (\omega,\mathbf{a})}$ (depending on $u^{(r)}(\omega,\amplitude)$) is a function of $\omega$ and $\amplitude$.  We write $T_{u^{(r)}}(\si,\omega,\amplitude)$ for the operator 
$F'(u^{(r)})=D(\si)+\delta \mathcal W_{u^{(r)}(\omega,\amplitude)}, $
and $\tilde{T}_{u^{(r)}}(\si,\omega,\amplitude)$ for 
$F'(u^{(r)})=D(\si)+\delta \mathcal W_{u^{(r)}(\omega,\amplitude)}$  restricted to $(\{\pm\}\times \Z^{b+d})\backslash \mcS$,  where
\begin{equation}\label{glinearD2}
	{D}(\si)=\begin{pmatrix}  \text {diag }(\si+n\cdot\omega+\mu_j)
		&0\\
		0&  \text {diag }(-\si-n\cdot\omega+\mu_j)
	\end{pmatrix}, \  (n,j)\in\Z^{b+d}.
\end{equation}
For simplicity, write  $\tilde{T}_{u^{(r)}}(\omega,\amplitude)$ for $\tilde{T}_{u^{(r)}}(0,\omega,\amplitude)$ and ${T}_{u^{(r)}}(\omega,\amplitude)$ for ${T}_{u^{(r)}}(0,\omega,\amplitude)$.
 
The analysis of the linearized operators $F'_N$ uses Theorem \ref{seperative spectrum} for small scales; for large scales, it
 also uses Theorem \ref{LDT} and semi-algebraic type  projection arguments to convert estimates on $\si$ into that on $\omega$, and finally on  $\amplitude$.

\subsubsection{The Q-equation}
The Q-equation  is of  $2b$ dimensions. The symmetry information  leads to $b$ equations only. These equations are used to relate $\omega$ to $\mathbf{a}$, so that the removal of $\omega$-parameters ultimately turns into the removal of $\mathbf{a}$-parameters. Recall that we have fixed the amplitude $u(n,j)$ on resonant set $\mcS$ as \eqref{Fixed amplitude on resonant set}. These $b$ equations are then seen as equations for the frequencies $\omega$ instead, and we have 
\begin{align}\label{Q-equation Explicit}
	\omega_k=\mu_{\beta_k}+\delta\frac{W_u(-e_k,\beta_k)}{a_k}, \ k=1,2,\cdots,b.
\end{align}
When $u=u^{(0)}$, let us compute the terms in the Q-equation \eqref{Q-equation Explicit}. For $k=1,2,\cdots,b$, we have 
\begin{align}\label{Q-equation nonlinear term}
	W_{u^{(0)}} (-e_k,\beta_k)=&\sum_{n'+\sum_{m=1}^p (n_m+n_m')=-e_k \atop {n', n_m,n_m'\in\Z^b}} \sum_{\ j', j_m,j_m'\in\Z^d } u^{(0)}(n',j')\prod_{m=1}^p u^{(0)}(n_m,j_m) v^{(0)}(n'_m,j'_m) \nonumber\\
	     &\ \ \ \times \left(\sum_{x\in\Z^d}\phi_{ \beta_k }(x)\phi_{j'}(x)\prod_{m=1}^p\phi_{j_m} (x) {\phi_{j_m'}(x)}\right).
\end{align}
Since $u^{(0)}$ has the support $\{(-e_k,\beta_k)\}_{k=1}^b$, 
one has 
\[j',j_m,j'_m\in\{\beta_k \}_{k=1}^b\ {\rm for}\ m=1,2,\cdots,b.\]
In fact, by \eqref{Q-equation nonlinear term}, $W_{u^{(0)}} (-e_k,\beta_k)={\rm Poly}(\mathbf{a})$ is a polynomial in $\mathbf{a}$ with coefficients
\[\sum_{x\in\Z^d}\phi_{\beta_k}(x)\phi_{j'}(x)\prod_{m=1}^p\phi_{j_m} (x) {\phi_{j_m'}(x)}.\]
It is  easy to see  (by \eqref{Eigenfunction in Nonlinear analysis})
\begin{align}\label{Main coefficient}
1\geq \sum_{x\in\Z^d}\phi_{\beta_k}(x)\phi_{j'}(x)\prod_{m=1}^p\phi_{j_m} (x) {\phi_{j_m'}(x)} &=\sum_{x\in\Z^d}|\phi_{\beta_k}(x)|^{2p+2}\\
 \notag    &\geq |\phi_{\beta_k}(\beta_k)|^{2p+2}\\
\notag	 &\geq (1-\varepsilon^{\frac{1}{2}})^{2p+2}\geq 1-(2p+2)\varepsilon^{\frac{1}{2}}. 
\end{align}
Moreover, except for the case $j'=j_m=j'_m=\beta_k,m=1,\cdots,b$,  we assume without loss of generality $j'\neq \beta_k$ and we have still by \eqref{Eigenfunction in Nonlinear analysis}  
\begin{align}\label{not main coefficient}
	\sum_{x\in\Z^d}\phi_{\beta_k}(x)\phi_{j'}(x)\prod_{m=1}^p\phi_{j_m} (x) {\phi_{j_m'}(x)} &\leq \sum_{x\in\Z^d}e^{-\kappa(|x-\beta_k|_1+|x-j'|_1+|x-j_1|_1+\cdots+|x-j'_b|_1)}\\
\notag		 &\leq \sum_{x\in\Z^d}e^{-\kappa(|x-\beta_k|_1+|x-j'|_1)}\\
\notag		 &\leq e^{-\frac{\kappa}{2}|j'-\beta_k|_1} \sum_{x\in\Z^d}e^{-\frac{\kappa}{2}(|x-\beta_k|_1+|x-j'|_1)}\\
	\notag	 &\leq 2e^{-\frac{\kappa}{2}}=2\varepsilon^{\frac{1}{4}}.
\end{align}
Denote $A_k=\sum_{x\in\Z^d}|\phi_{\beta_k}(x)|^{2p+2}$. From \eqref{Q-equation nonlinear term}, \eqref{Main coefficient} and \eqref{not main coefficient}, the leading contribution in the sum of \eqref{Q-equation nonlinear term} is when 
\[(n',j'),(n_m,j_m),(n'_m,j'_m)=(-e_k,\beta_k).\]
Then by \eqref{Q-equation Explicit}, we have 
\begin{equation}\label{approimation omega^1}
	(\omega^{(1)}_k=)\ \omega_k=\mu_{\beta_k}+\delta(A_k a_{k}^{2p}+\mcO(1)\varepsilon^{\frac{1}{4}}).
\end{equation}
Since $1-(2p+2)\varepsilon^{\frac{1}{2}}\leq A_k\leq 1,{\bf a}\in [1,2]^b$, the above expression gives 
\[\omega\in \Omega_0:= [\mu_{\beta_1},\mu_{\beta_1}+2^{2p+1}\delta]\times[\mu_{\beta_2},\mu_{\beta_2}+2^{2p+1}\delta]\times\cdots\times[\mu_{\beta_b},\mu_{\beta_b}+2^{2p+1}\delta].\]
This indicates  that we can take the constant $C_2=2^{2p+1}$ in Theorem \ref{LDT}.  

Assume that after $r$ iteration steps, we obtain a $C^1$-function $u^{(r)}(\omega,\mathbf{a})$ on $\Omega_0\times [1,2]^b$. Substituting $(u^{(r)}(\omega,\mathbf{a}),v^{(r)}(\omega,\amplitude))$ into \eqref{Q-equation Explicit} and applying the implicit function theorem yield  the next $(r+1)$-approximation about frequency $\omega$ 
\[\omega=\omega^{(r+1)}(\amplitude)\]
and its reverse
\[\amplitude=\amplitude^{(r+1)}(\omega).\]
Again,  since   $u^{(r)}|_{\mcS}=\bf a$,  similar argument leading to \eqref{approimation omega^1}  implies  that, for some $C^1$-functions $f_k,k=1,\cdots,b$, we have 
\begin{equation}\label{approimation omega^r+1}
	\omega^{(r+1)}_k(\amplitude)=\mu_{\beta_k}+\delta(A_k a_k^{2p}+f_k(a_1,a_2,\cdots,a_b)). 
\end{equation} 
Denote  by $\Gamma_r={\rm graph}(\omega^{(r+1)}(\amplitude))\subset \Omega_0\times[1,2]^b$  the graph of $(\omega,\amplitude)$ at step $r$. Denote by $P_x$ the projection onto the $x$-variable, where $x=\amplitude,\omega$ or $(\omega,\amplitude)$.

\subsection{The inductive hypothesis and its proof}
Let $M$ be a large integer, and denote by $B(0, R)$ the
 $\ell^\infty$ ball on $\mathbb Z^{b+d}$ centered  at the origin with radius $R$.
 Set 
 $$r_0=\left\lfloor \frac{|\log\delta|^{\frac{3}{4}} } {\log M}\right\rfloor.$$

We first state the induction hypothesis, which can be verified if $0< \delta\leq \tilde{\delta} $,   $0<\varepsilon\leq \tilde{\varepsilon}(\delta)$ and $\theta\in \Theta_{\bfa}$.  Given its significance, we refer to it as the induction theorem. In the following,  $C$ represents a large constant and $c>0$   a small constant,  both of which  depend only on $b,d$. These constants may vary in different parts of the context.

\begin{thm}\label{inductive}
	For $r\geq 1$, we assume that the following holds:
\begin{itemize}
	\item[\bf Hi.] ${u}^{(r)}( \omega, \amplitude)$  is a $C^1$ map on  $\Omega_0\times [1,2]^b$, and
	$\text{supp } u^{(r)}\subseteq B(0, M^r)$ (note that $\text{supp } u^{(0)}\subset B(0, \frac{M}{10})$).
	\item[\bf Hii.] $\Vert \Delta_{cor}  u^{(r)}\Vert\leq \delta_r$,
	$\Vert \partial \Delta_{cor}  u^{(r)}\Vert\leq \bar\delta_r$,
	where $\partial$ denotes $\partial_x$, $x$ stands for $\omega_k$, $a_k$,  $k=1, 2, \cdots, b$
	and $\Vert\,\Vert:=\sup_{(\omega, a)}\Vert\,\Vert_{\ell^2(\Z^{b+d})}$.
	\item[\bf Hiii.] $|  {u^{(r)}}(n,j)|\leq  C  e^{ -c(|n|+|j|)}$.
       \begin{rmk}\label{gammar}
		The constant $c>0$ will decrease slightly along the iterations,  but remain  bounded away from zero. This will become clear in the proof (cf. also  \cite[Sect. V. A]{KLW24}). We 
		also remark that, to put $u^{(r)}$ in Q-equation  and apply implicit function theorem  solving  $\omega^{(r+1)}$, the $u^{(r)}$ will be extended to the entire parameter space $(\omega,\amplitude)\in\Omega_0\times [1,2]^b$ by using the  standard dilation argument (cf. \cite{Bou98},(10.33)--(10.37)). Hence,  by the relationship 
		\[\omega_k^{(r+1)}=\mu_{\beta_k}+\delta\frac{W_{u^{(r)}}(-e_k,\beta_k)}{a_k},\  k=1,2,\cdots,b\]
		together with hypothesis \textbf{Hii.}, one  has  
		\[|\omega^{(r+1)}(\amplitude)-\omega^{(r)}(\amplitude)|\lesssim \delta \Vert u^{(r)}-u^{(r-1)}\Vert \lesssim \delta\cdot\delta_r.\]
		That is to say $\Vert\Gamma_r-\Gamma_{r-1}   \Vert\lesssim \delta\cdot\delta_r$.
	   \end{rmk}
	\item[\bf Hiv.] 
	There exists  a collection $\mcI_r$ of disjoint open sets $I$ (in $(\omega, \amplitude)$) of size $M^{-r^{10C}}$ when $r\geq r_0$ (the total number of open sets is therefore bounded above by $M^{r^{10C}}$),
	 such that:    for any  $(\omega, \amplitude)\in \bigcup_{I\in \mcI_r}I$ when $r\geq r_0$,  and  $(\omega, \amplitude)\in \Omega_0\times [1,2]^b$ when $1\leq r\leq r_0-1$), we have 
	 	\begin{enumerate}
		\item  $u^{(r)}(\omega,\amplitude)$ is a  rational function in $(\omega, \amplitude)$  of degree at most $M^{r^3}$;
		\item 
		\begin{equation}\label{F}
		\Vert F( u^{(r)})\Vert\leq \kappa_r,\\
		\Vert \partial F( u^{(r)})\Vert\leq \bar\kappa_r; \end{equation}
		\item Let $r\leq r_0-1$. Then 
		 \begin{equation}\label{small r Green norm}
					       \nm(R_{[-M^r,M^r]^{b+d}}\tilde{T}_{u^{(r-1)}}(\omega,\amplitude)R_{[-M^r,M^r]^{b+d}})^{-1}\nm\leq 2\delta^{-\frac{1}{8}}, 
		    \end{equation} 
			and for $s,s'\in \{\pm\}$, 
			\begin{align}\label{small r Green off-diagonal}
				&\ \ \ |\big((R_{[-M^r,M^r]^{b+d}}\tilde{T}_{u^{(r-1)}}R_{[-M^r,M^r]^{b+d}})^{-1}\big)_{s,s'}(n,j;n',j')-\big(D(0)^{-1}\big)_{s,s'}(n,j;n',j')|\\
				 \notag  &\leq \delta^{\frac{1}{2}}e^{-c(|n-n'|+|j-j'|)},
			\end{align}
			where $D(0)$ is the diagonal component of $\tilde{T}_{u^{(r-1)}}$ restricted to  $B(0,M^r)$, which is given by \eqref{glinearD2}; 
        \item Let $r\geq r_0$. Then  
		\begin{equation}\label{big r Green norm}
			\nm(R_{[-M^r,M^r]^{b+d}}\tilde{T}_{u^{(r-1)}}(\omega,\amplitude)R_{[-M^r,M^r]^{b+d}})^{-1}\nm\leq M^{r^C},
         \end{equation} 
            and for $|n-n'|+|j-j'|>r^C, s,s'\in\{\pm\}$,  
            \begin{align}\label{big r Green off-diagonal}
           |\big((R_{[-M^r,M^r]^{b+d}}\tilde{T}_{u^{(r-1)}}(\omega,\amplitude) R_{[-M^r,M^r]^{b+d}})^{-1}\big)_{s,s'}(n,j;n',j')  | &\leq e^{-c(|n-n'|+|j-j'|)}; 
               \end{align} 
		\item 
				
				 Each $I\in {\mcI}_r$ is contained in some  $I'\in {\mcI}_{r-1}$ for  $r\geq r_0$,
				and 
					\begin{equation}\label{gme1b1111}
					\meas (P_{\amplitude}(\Gamma_{r-1}\cap (\bigcup_{I'\in {\mcI}_{r-1}} I'\backslash \bigcup_{I\in {\mcI}_{r}}I)))\leq e^{-|\log \delta|^{ K_1^{80}}} +M^{-\frac{r_0}{2^b}},\  r=r_0;
				\end{equation} 
			and 
				\begin{equation}\label{gme1b11}
				\meas (P_{\amplitude}(\Gamma_{r-1}\cap (\bigcup_{I'\in {\mcI}_{r-1}} I'\backslash \bigcup_{I\in {\mcI}_{r}}I)))\leq M^{-\frac{r}{2^b}},\  r\geq r_0+1;
				\end{equation} 
					\item For  $(\omega, \amplitude)\in \bigcup_{I\in \mcI_r}I$ with $r\geq r_0$, $\omega$  satisfies  the conditions \eqref{Diophantine on omega frequency} and \eqref{weak second Melnikov} for $n\neq 0$ and  $\tilde{N}\in [(\log\frac{1}{\delta})^K, M^{r}]$; 
	
\end{enumerate}
		\item[\bf Hv.]
		The iterations  hold  with 
		\begin{equation}\label{conv}
		\delta_r= \delta^{\frac{1}{2}}M^{-(\frac{4}{3})^r}, \, \bar\delta_r= \delta^{\frac{1}{8}} M^{-\frac{1}{2}(\frac{4}{3})^r}; \ \kappa_r= \delta^{\frac{3}{4}} M^{-(\frac{4}{3})^{r+2}}, \, \bar\kappa_r= \delta^{\frac{3}{8}} M^{-\frac{1}{2}(\frac{4}{3})^{r+2}}. \end{equation}
\end{itemize}

\end{thm}

In the following,  we will prove Theorem \ref{inductive}: The general scheme is as follows:  for small scales, we use \eqref{(2.5)}--\eqref{(2.6)} of Theorem \ref{seperative spectrum} and Theorem \ref{Neumann expansion} to establish \eqref{small r Green norm} and \eqref{small r Green off-diagonal}; for larger scales,  we employ \eqref{(2.4)} of Theorem \ref{seperative spectrum}, Theorem \ref{LDT} and semi-algebraic type  projection lemma.

\begin{proof}[Proof of Theorem \ref{inductive}]

We begin with an important projection lemma on semi-algebraic sets, which has played an essential role in the study of both linear and nonlinear Anderson localization. 

\begin{lem}[cf. Lemma 9.9, \cite{Bou05}]\label{projection lemma}
	Let $\mathcal{S}\subset [0, 1]^{d_1}\times [0, 1]^{d_2}:=[0, 1]^{d}$, be a semi-algebraic set of degree $B$ and $\text{\rm meas}_{d} S <\eta, \log B\ll
	\log 1/\eta$.
	Denote by $(x, y)\in [0, 1]^{d_1}\times [0, 1]^{d_2}$ the product variable.
	Fix $\epsilon>\eta^{1/d}$.
	Then there is a decomposition
	$$
	\mathcal{S} =\mathcal{S}_1 \bigcup \mathcal{S}_2
	$$
	with $\mathcal{S}_1$ satisfying
	$$
	\text{\rm meas}_{d_1}({\rm Proj}_x \mathcal{S}_1)\leq B^C\epsilon,
	$$
	and $\mathcal{S}_2$ the transversality property
	$$
	\text{\rm meas}_{L}(\mathcal{S}_2\cap L)\leq  B^C\epsilon^{-1} \eta^{1/d} 
	$$
	for any $d_1$-dimensional hyperplane $L$ in  $[0, 1]^{d_1+d_2}$ such that
	\begin{equation}\label{transversality projection}
			\max_{1\leq l\leq d_1}|{\rm Proj}_L (e_l)|\leq  \frac 1{100}\epsilon,
	\end{equation}
	where $e_l$ are the basis vectors for the $x$-coordinates.
\end{lem}

	Assume that   each induction hypothesis holds for all scales up to $r$. We will prove that it holds for $r+1$. 
	
	The Newton scheme gives 
	\begin{equation}\label{Newton multiscale}
	\Delta_{cor} \begin{pmatrix} u^{(r+1)}\\v^{(r+1)}\end{pmatrix}=-[F'_{N}(u^{(r)})]^{-1} F_N(u^{(r)}),\quad 
	 N=M^{r+1},
	\end{equation}
    and by \textbf{Hiv.}(1), $u^{(r)}(\omega,\amplitude)$ is  a rational function of degree at most $M^{r^3}$ about $(\omega,\amplitude)$. By the convolution structure of $W_u$, we have 
	\[\deg F(u^{(r)})\leq (2p+1)M^{r^3},\ \deg F'(u^{(r)})\leq 2p M^{r^3}.\]
	Then by  Cramer's rule, we  have 
	\[[F'_{N}(u^{(r)})]^{-1}=\frac{A}{\det([F'_{N}(u^{(r)})]^{-1})},\]
	where $A$ is the adjacent matrix of $[F'_{N}(u^{(r)})]^{-1}$. Hence,
	\[\deg [F'_{N}(u^{(r)})]^{-1}\leq 2N\deg F'(u^{(r)})\leq (4p+2)M^{r+1}M^{r^3}.\]
	Thus,   using \eqref{Newton multiscale} shows 
	\[\deg\big( \Delta_{cor} \begin{pmatrix} u^{(r+1)}\\v^{(r+1)}\end{pmatrix}  \big)\leq (4p+2)M^{r+1}M^{r^3}+(2p+1)M^{r^3} \leq M^{(r+1)^3}  \]
	This proves the \textbf{Hiv.}(1) for $r+1$. 
	
	For $r\leq r_0-1$, we directly check those hypotheses hold  by Neumann series argument. Now $M^{r+1}\leq M^{r_0}\leq e^{|\log\delta|^{\frac{3}{4}}}$. Recall that \eqref{(2.5)} and \eqref{(2.6)} ensure  that the diagonals  of $\tilde{T}_{u^{(r)}}$ are bounded from below by $2\delta^{\frac{1}{8}}$.  Then  by induction  hypothesis \textbf{Hiii.} together with Proposition \ref{generate the nonlinear operator}. (2), we have 
	\[|\mcW_{r,r'}(n,j;n',j')|\leq C_1C^p(|n-n'|+1)^{C_1}e^{-c(|n-n'|+|j-j'|)}.\]
	Choose $\epsilon_1=2\delta^{\frac{1}{8}},\epsilon_2=C_1C^p\delta,S=[-N,N]^{b+d}$ in Theorem \ref{Neumann expansion}. Then  \eqref{Verification Neumann}  can be verified by 
	\[4|S|^2 (\diam(S)+1)^C \epsilon_2\epsilon_1^{-1}\lesssim\delta^{\frac{7}{8}}e^{(2b+2d+C_1)|\log\delta|^{\frac{3}{4}}}\leq \delta^{\frac{3}{4}}\ll \frac{1}{2}\]
	provided  $0<\delta\ll1.$ Then applying Theorem \ref{Neumann expansion} will ensure  that \eqref{small r Green norm} and \eqref{small r Green off-diagonal} hold  for $r+1$  for any $(\omega,\amplitude)\in \Omega_0\times[1,2]^b$. Clearly, for $r+1=r_0$, we have 
	\[r_0^C\sim |\log\delta|^{\frac{3}{4}C}\Rightarrow 2\delta^{\frac{1}{8}}\ll M^{r_0^C}.\]
	Therefore,  when $r=r_0$,  \eqref{big r Green norm} and \eqref{big r Green off-diagonal} automatically hold since  \eqref{small r Green norm} and \eqref{small r Green off-diagonal}. 
	
	Thus, we are  in a position to treat the case $r\geq r_0$.
	
	For $r\geq r_0$, it suffices  to establish \textbf{Hiv.}(4) and \textbf{Hiv.}(5), and the other induction  assumptions can be verified in the standard way as in  \cite[Sect. V]{KLW24} and \cite{Bou05}. For the constants relations \textbf{Hv.},  we refer to \cite[Appendix E]{LW24}. The strategy of the proof is as follows.  Take $N=M^{r+1}$ and $N_1=(\log N)^C$ with a  large constant $C>1$. By $r\geq r_0$, we have $\log N_1\geq |\log\delta|^{\frac{3}{4}C} $. Take  
	\[r_1=\lfloor \frac{2\log N_1}{\log\frac{4}{3}}\rfloor+1,\]
	which is the minimum integer such that $(\frac{4}{3})^{r_1}\geq N_1^2$, and hence $\delta_{r_1}<e^{-N_1^2}$. We have a decomposition 
	\[[-N,N]^{b+d}=[-M^r,M^r]^{b+d}\cup \mcR_1\cup\mcR_2,\]
	where 
	\[\mcR_1=\{(n,j)\in [-N,N]^{b+d}:\ |j|> 2N_1\},\]
	\[\mcR_2=\{(n,j)\in [-N,N]^{b+d}:\ |n|\geq \frac{M^r}{2},|j|\leq 2N_1\}.\]
	Now pick one $I\in \mcI_{r_1}$ of size $M^{-r_1^{10C}}$ (with at most $M^{r_1^{10C}}$ many). Let $I_1=P_{\omega}(I\cap \Gamma_{r_1})$ with $\Gamma_{r_1}={\rm graph} (\omega^{(r_1+1)}(\amplitude))$. For any $k\geq |\log\delta|^K$, denote by $DC(k)$ all $\omega\in \Omega_0$  satisfying \eqref{Diophantine on omega frequency} and \eqref{weak second Melnikov} for any $|\log\delta|^K\leq \tilde{N}\leq k$. Denote $DC^{r+1}=DC([(r+1)\log M ]^C)\setminus DC((r\log M)^C)$. Let $I_2=I_1\cap DC^{r+1}$.\smallskip
   
 \textbf{Analysis in $[-M^r,M^r]^{b+d}$}\\
	By the induction  hypothesis at step $r$, we have that \eqref{big r Green off-diagonal} and \eqref{big r Green norm} hold  at step $r$. Also, by \textbf{Hii.}, we have, if $|n-n'|+|j-j'|\leq (r+1)^{5C}\log M,$ then 
	\begin{align}\label{perturbation 1}
	&\ \ \ 	|(\tilde{T}_{u^{r}}(\omega,\amplitude))_{s,s'}(n,j;n',j')- (\tilde{T}_{u^{r-1}}(\omega,\amplitude))_{s,s'}(n,j;n',j')| \\
		\notag    &\leq C\Vert \Delta_{cor}u^{(r)}\Vert\\
		\notag	&\leq C\delta_r \leq CM^{-(r+1)^{10C}} \\
		\notag	&\leq M^{-10(r+1)^C}e^{-c(|n-n'|+|j-j'|)}.
	\end{align}
If $|n-n'|+|j-j'|>(r+1)^{5C}\log M$, by \textbf{Hiii.} and Proposition \ref{generate the nonlinear operator}.  (2), we have 
	\begin{align}\label{perturbation 2}
		&\ \ \ |(\tilde{T}_{u^{r}}(\omega,\amplitude))_{s,s'}(n,j;n',j')-(\tilde{T}_{u^{r-1}}(\omega,\amplitude))_{s,s'}(n,j;n',j')| \\
		 \notag   &\leq C|n-n'|^Ce^{-c(|n-n'|+|j-j'|)}\\
		\notag	&\leq M^{-10(r+1)^C}e^{-(c-(r+1)^{-3C})(|n-n'|+|j-j'|)}. 
	\end{align}
	The above estimates  imply  (for simplicity,  we hide the changes  of exponential decay rate)
	\begin{equation}\label{perturbation by approimation to operator 1}
		|(\tilde{T}_{u^{r}})_{s,s'}(n,j;n',j')- (\tilde{T}_{u^{r-1}})_{s,s'}(n,j;n',j')|\leq M^{-10(r+1)^C}e^{-c(|n-n'|+|j-j'|)}. 
	\end{equation}
	Hence,  by \eqref{big r Green off-diagonal}--\eqref{big r Green norm} at step $r$ and using Lemma \ref{Neumann expansion}, we have that for any $(\omega,\amplitude)\in \cup_{I\in \mcI_r}I$,  one can replace $u^{(r-1)}$ in \eqref{big r Green norm}--\eqref{big r Green off-diagonal} by $u^{(r)}$. That is to say,
	\begin{equation}\label{center block r+1 Green norm}
					\nm(R_{[-M^r,M^r]^{b+d}}\tilde{T}_{u^{(r)}}(\omega,\amplitude)R_{[-M^r,M^r]^{b+d}})^{-1}\nm\leq M^{r^C},
	\end{equation}
	and for $|n-n'|+|j-j'|>r^C,$
	\begin{equation}\label{center block r+1 Green off-diag}
			 |\big((R_{[-M^r,M^r]^{b+d}}\tilde{T}_{u^{(r)}}(\omega,\amplitude) R_{[-M^r,M^r]^{b+d}})^{-1}\big)_{s,s'}(n,j;n',j')  | \leq e^{-c(|n-n'|+|j-j'|)}.
	\end{equation}
	Thus, we  have obtained  the desired  Green's function estimates on  $[-M^r,M^r]^{b+d}$.\smallskip

	\indent \textbf{Analysis in $\mcR_1$}\\
	Assume $|j_0|>2N_1$. In this case, we denote 
	\[\tilde{I}_2^r=\{\omega\in I_2:\ {\exists}(n,j)\in [-N,N]^{b+d}\ {\rm s.t.,}\ |\nomega+\mu_j|\leq 2e^{-N_1^{\frac{9}{10}}} \ {\rm or} \ |\nomega-\mu_j|\leq 2e^{-N_1^{\frac{9}{10}}}\}.\]
	Recall that \eqref{(2.4)} ensures that $|\mu_j|\geq \frac{1}{2N^{d+2}}> 2e^{-N_1^{\frac{9}{10}}} $. Thus, the condition in the definition of $\tilde{I}_2^r$ does not hold  if $n=0$. It's easy to check that 
	\begin{equation}\label{meas tildeI_2^r}
		\meas(\tilde{I}_2^r)\leq N^C e^{-N_1^{\frac{9}{10}}}\ll \delta^b M^{-\frac{r}{2}}. 
	\end{equation}
	Suppose that $\omega\notin \tilde{I}_2^r$. Due to $|j_0|>2N_1$, the off-diagonal part of $R_{(n_0,j_0)+\La_{N_1}}T_{u^{(r_1)}}R_{(n_0,j_0)+\La_{N_1}}$ is sufficiently small. Applying Theorem \ref{Neumann expansion}  yields that  
	for any $\omega\in I_2\setminus \tilde{I}_2^r,(\omega,\amplitude)\in \Gamma_{r_1}$, $(n_0,j_0)\in \mcR_1$ and $\La_{N_1}\in \mcE^0_{N_1}$ with $(n_0,j_0)+\La_{N_1}\subset [-N,N]^{b+d}$, we have good Green's function estimates in the sense of 
	\begin{equation}\label{mcR_1 r_1 Green norm}
					\nm(R_{(n_0,j_0)+\La_{N_1}}\tilde{T}_{u^{(r_1)}}(\omega,\amplitude)R_{(n_0,j_0)+\La_{N_1}})^{-1}\nm\leq e^{N_1^{\frac{9}{10}}}, 
	\end{equation}
	and for $|n-n'|+|j-j'|>\sqrt{N_1},$
	\begin{equation}\label{mcR_1 r_1 Green off-diag}
			 |\big[(R_{(n_0,j_0)+\La_{N_1}}\tilde{T}_{u^{(r_1)}}(\omega,\amplitude) R_{(n_0,j_0)+\La_{N_1}})^{-1}\big]_{s,s'}(n,j;n',j')  | \leq e^{-c(|n-n'|+|j-j'|)}. 
	\end{equation}
Thus, we have established desired  Green's function estimates  on boxes of   size $N_1$  contained in $\mcR_1$.\smallskip
	
	\indent \textbf{Analysis in $\mcR_2$}\\
	Assume now $|j_0|\leq 2N_1$ and  $|n_0|\geq \frac{M^r}{2}$. We fix $\omega\in I_2, (\omega,\amplitude)\in \Gamma_{r_1} $, i.e.,  $\amplitude=\amplitude^{(r_1+1)}(\omega)$. Since  $I_2\in \Omega_0\cap DC([(r+1)\log M]^C)=\Omega_0\cap DC(N_1)$, we  apply LDT Theorem \ref{LDT} to obtain a set 
	$X_{N_1}=X_{N_1}(\omega)$ (depending on $\omega$)  about $\si$, such that 
	\[\meas(X_{N_1})\leq e^{-N_1^{\frac{1}{20}}},\]
	and for all $\si\notin X_{N_1}$, $|j_0|\leq 2N_1$ and $\La_{N_1}\in \mcE^0_{N_1}$, we have 
    \begin{equation}\label{mcR_2 LDT Green norm}
		\nm(R_{\La_{N_1}(j_0)}\tilde{T}_{u^{(r_1)}}(\si,\omega,\amplitude^{(r_1+1)}(\omega))R_{\La_{N_1}(j_0)})^{-1}\nm\leq e^{N_1^{\frac{9}{10}}}, 
    \end{equation}
    and for $|n-n'|+|j-j'|>\sqrt{N_1},$
   \begin{equation}\label{mcR_2 LDT Green off-diag}
       |\big((R_{\La_{N_1}(j_0)}\tilde{T}_{u^{(r_1)}}(\si,\omega,\amplitude^{r_1+1}(\omega)) R_{\La_{N_1}(j_0)})^{-1}\big)_{s,s'}(n,j;n',j')  | \leq e^{-c(|n-n'|+|j-j'|)}. 
   \end{equation}
    Moreover, if we take
    \begin{align*}
	&\ \ \  I_{N_1}(\omega)\\
	&=\left\{\si:\ {\exists}|j|\leq 3N_1,|n|\leq N_1\ {\rm s.t.,}\ |\si+\nomega+\mu_j|\leq N_1^{10(b+d+C_1)}\delta \ {\rm or} \ |\si+\nomega+\mu_j|\leq N_1^{10(b+d+C_1)}\delta \right\},
	\end{align*}
	where $C_1=C_1(p,b)>1$ is given in Proposition \ref{generate the nonlinear operator}.  (2).   If  $\si\notin I_{N_1}$, then applying  Theorem \ref{Neumann expansion}  shows , for any $|j_0|\leq 2N_1$ and $\La_{N_1}\in \mcE^0_{N_1}$,  we have 
	\[		\nm(R_{\La_{N_1}(j_0)}\tilde{T}_{u^{(r_1)}}(\si,\omega,\amplitude^{(r_1+1)}(\omega))R_{\La_{N_1}(j_0)})^{-1}\nm\leq \frac{1}{\delta}\leq e^{N_1^{\frac{9}{10}}},\]
	and for $|n-n'|+|j-j'|>\sqrt{N_1},$
	\begin{align*}
			&\ \ \ |\big((R_{\La_{N_1}(j_0)}\tilde{T}_{u^{(r_1)}}(\si,\omega,\amplitude^{r_1+1}(\omega)) R_{\La_{N_1}(j_0)})^{-1}\big)_{s,s'}(n,j;n',j') |\\
			&\leq \frac{4}{\delta} e^{-c(|n-n'|+|j-j'|)}\\
			&\leq e^{-(c-\frac{|\log\delta|}{\sqrt{N_1}})(|n-n'|+|j-j'|)}\\
			&\leq e^{-(c-N_1^{-\frac{1}{4}})(|n-n'|+|j-j'|)}\\ 
			&\leq e^{-(c-(r+1)^{-\frac{C}{4}})(|n-n'|+|j-j'|)}. 
	\end{align*}
	For simplicity, we also hide the changes  of  the exponential rate. Therefore, \eqref{mcR_2 LDT Green norm} and \eqref{mcR_2 LDT Green off-diag} hold, which  implies that $X_{N_1}(\omega)\subset I_{N_1}(\omega)$. Since $\Omega_0$ is of size $C_2\delta$ and $I_{N_1}(\omega)$ only involves $|n|\leq N_1$, we can assume that $X_{N_1}(\omega)$ is in a union of a collection of intervals (independent of $\omega$) of size $\delta$ with total number $N_1^C$. We just pick one  interval $\Sigma$ of them. Let $\mcX_{N_1}(\omega,\si)\in I_2\times \Sigma$ be such that,  there exist some $\La_{N_1}\in \mcE^0_{N_1}$ and $|j_0|\leq 2N_1$ such that either \eqref{mcR_2 LDT Green norm} or \eqref{mcR_2 LDT Green off-diag} is not true. 
	It's easy  to see that 
	\[\mcX_{N_1}(\omega,\si)=\bigcup_{\omega\in I_2}\{\omega\}\times (I_{N_1}(\omega)\cap \Sigma),\]
	and by Fubini's  theorem,  
	\begin{equation}\label{fubini}
		\meas(\mcX_{N_1})\leq C_2^b \delta^b e^{-N_1^{\frac{1}{30}}}\leq \delta^b e^{-N_1^{\frac{1}{31}}}. 
	\end{equation}
	Moreover, we claim that $\mcX_{N_1}$ is a semi-algebraic set of degree at most $N_1^C M^{Cr_1^3}$. Indeed, this can be seen as follows. Let $\tilde{X}_{N_1}\subset \Omega_0\times [1,2]^b \times\R$ be such that there exist some $\La_{N_1}\in \mcE^0_{N_1},|j_0|\leq 2N_1$ such that,  one of the following is not true: 
	\begin{equation*}
		\nm(R_{\La_{N_1}(j_0)}\tilde{T}_{u^{(r_1)}}(\si,\omega,\amplitude)R_{\La_{N_1}(j_0)})^{-1}\nm\leq e^{N_1^{\frac{9}{10}}},
    \end{equation*}
    or  for $|n-n'|+|j-j'|>\sqrt{N_1}$, 
   \begin{equation*}
       |\big((R_{\La_{N_1}(j_0)}\tilde{T}_{u^{(r_1)}}(\si,\omega,\amplitude) R_{\La_{N_1}(j_0)})^{-1}\big)_{s,s'}(n,j;n',j')  | \leq e^{-c(|n-n'|+|j-j'|)}. 
   \end{equation*}
   Then $\mcX_{N_1}=P_{(\omega,\si)}(\tilde{X}_{N_1}\cap(\Gamma_{r_1}\times \R))$. By the similar  truncation arguments  as in Section 3.4, \textbf{Step 1} and \textbf{Hiv.} (1), we know that $\tilde{X}_{N_1}$ is a semi-algebraic set of degree at most $N_1^C M^{Cr_1^3}$. Simultaneously,    $\Gamma_{r_1}$ is given by the Q-equation 
   \[\omega_k^{(r_1+1)}=\mu_{\beta_k}+\delta\frac{W_{u^{(r_1)}}(-e_k,\beta_k)}{a_k}, \ k=1,2,\cdots,b,\]
   which are rational equations of degree at most $CM^{r_1^3}$. Then the degree of $\tilde{X}_{N_1}\cap(\Gamma_{r_1}\times \R)$ will not exceed $N_1^C M^{Cr_1^3}$. Hence,  by Lemma \ref{projection semi algebraic} (i.e., the Tarski-Seidenberg principle),  we have that  $\mcX_{N_1}$ is of degree at most $N_1^C M^{Cr_1^3}$.\\
   \\ 
   Next we will apply Lemma \ref{projection lemma} to remove $\omega$. The  argument  is similar to  that, such as  (3.26) in \cite {Bou07} and  (142) in \cite{LW24}. Our aim is to estimate $\meas(P_{\omega}(\mcX_{N_1}\cap L))$ for hyperplane $L=\{(\omega,k\cdot \omega)\}$ for all $k=(k_1,k_2,\cdots,k_b)\in [-N,N]^b,|k|\geq \frac{M^r}{2}$. We remark that, although $|k|$ is large, we not necessarily have the transversality \eqref{transversality projection}, because $|\cdot|$ is the infinity norm 
   but  the transversality \eqref{transversality projection} needs 
   \[\min_{1 \leq i\leq b}|k_i|\geq 100 \epsilon^{-1}.\]
   So, we need to discuss all possibilities via $b$ steps of inductions. Without loss of generality,   we  assume that $k_1\leq k_2\leq \cdots \le k_b$. {Moreover, since all parameters are of $\mathcal{O}(\delta)$ order, we remark  that Lemma \ref{projection lemma} admits the rescaling. Thus, if we rescale  the semi-algebraic set by size $\delta$, then we can obtain a  measure  bound  involving  the  $\delta^b$ factor. }
   From \eqref{fubini}, we just take $\eta=e^{-N_1^{\frac{1}{31}}}$. 
   Let 
     \[\epsilon_1=M^{-\frac{r}{2^{b-1}}},\epsilon_2=M^{-\frac{r}{2^{b-2}}},\cdots,\epsilon_{b-1}=M^{-\frac{r}{2}},\epsilon_b=2M^{-r}.\] 
  If $k_1\geq \epsilon_1^{-1}$, then \eqref{transversality projection} holds for $L$ with $\epsilon=100\epsilon_1$.  Applying  Lemma \ref{projection lemma} enables us to decompose $\mcX_{N_1}$ into $\mcS_1^{(1)}$ with 
	                    \[\meas(P_{\omega}\mcS_1^{(1)})\leq 100 \delta^b N_1^C M^{Cr_1^3} \epsilon_1 \]
						and $\mcS_2^{(1)}$ with 
						\[{\rm meas}_L(\mcS_2^{(2)}\cap L)\leq \frac{1}{100} \delta^b N_1^C M^{Cr_1^3} e^{-\frac{1}{(b+d)}N_1^{\frac{1}{31}}} \epsilon_1^{-1}. \]
						Again  by  \eqref{transversality projection}, we have 
						\[\meas(P_{\omega}(\mcS_2^{(2)}\cap L))\leq \frac{1}{100} \delta^b N_1^C M^{Cr_1^3} e^{-\frac{1}{(b+d)}N_1^{\frac{1}{31}}}.\]
   If $k_1<\epsilon_1^{-1}$, but $k_2\geq \epsilon_2^{-1}$, then we fix $\omega_1$ in $\omega$ and consider  $\tilde{\omega}=(\omega_2,\omega_3,\cdots,\omega_b)$.  Since the restriction on  $\omega_1$ just gives an additional equation of degree $1$, we have that $\mcX_{N_1}\big|_{\omega_1}$ is a semi-algebraic set of degree at most $ N_1^C M^{Cr_1^3}$. Moreover, 
            since  $\meas(X_{N_1}(\omega))\leq e^{-N_1^{\frac{1}{30}}}$,  we have by Fubini's  theorem,  
			\[\meas(\mcX_{N_1}\big|_{\omega_1}) \leq \delta^{b-1} e^{-N_1^{\frac{1}{31}}}.\]   
			Then  we can apply   Lemma \ref{projection lemma}   for the $\tilde{\omega}$-coordinate. Then \eqref{transversality projection} will hold for $L$ with $\epsilon=100\epsilon_2$. Thus, using  Lemma \ref{projection lemma} enables  us to decompose $\mcX_{N_1}\big|_{\omega_1}$ into $\mcS_1^{(2)}$ with  
			\begin{equation}\label{proj S_1^2}
				\meas(P_{\tilde{\omega}}\mcS_1^{(2)})\leq 100 \delta^{b-1} N_1^C M^{Cr_1^3} \epsilon_2,
			\end{equation}
			and $\mcS_2^{(2)}$ with 
			\[\meas(P_{\tilde{\omega}}(\mcS_2^{(2)}\cap L))\leq \frac{1}{100} \delta^{b-1} N_1^C M^{Cr_1^3} e^{-\frac{1}{(b+d)}N_1^{\frac{1}{31}}}. \]   
	        Finally, by Fubini's  theorem, we integrate \eqref{proj S_1^2} on $\omega_1$ and count all possibilities  of $k_1$ (totally $\epsilon_1^{-1}$ many). Then we get a set 
			\[\meas(\cup_{\omega_1}P_{\tilde{\omega}}\mcS_1^{(2)})\leq 100 \delta^b N_1^C M^{Cr_1^3} \epsilon_2\epsilon_1^{-1}.\]
	The iterations can continue and will  stop at $b$-th step, since $k_b=|k|\geq \frac{M^r}{2}=\epsilon_b^{-1}$. 
Moreover, the above estimates have already accounted for all possible $k$, i.e., all possible hyperplane $L$. Combining  all estimates across all cases, and also over all possible open sets $\Sigma$ (totally $N_1^C$), we obtain a set  (in  $\omega$)  $I^{r}_2\subset I_2$ such that 
\begin{align}\label{meas I_2^r}
	\meas(I_2^r) &\leq \delta^b N_1^C M^{Cr_1^3}(\epsilon_1+\epsilon_1^{-1}\epsilon_2+\cdots+\prod_{l=1}^{b-1}\epsilon_l^{-1}\epsilon_b)+\delta^b N^b N_1^C M^{Cr_1^3} e^{-\frac{1}{(b+d)}N_1^{\frac{1}{31}}}\\
	\notag &\leq \delta^b M^{-\frac{r}{2^{b-1}}}N^C_1 M^{Cr_1^3}, 
\end{align}
and for any $\omega\in I_2\setminus I_2^r,(\omega,\amplitude)\in \Gamma_{r_1}$, and  $(n_0,j_0)\in \mcR_2,\La_{N_1}\in \mcE^0_{N_1}$ with $(n_0,j_0)+\La_{N_1}\subset [-N,N]^{b+d}$, we have 
\begin{equation}\label{mcR_2 r_1 Green norm}
	\nm(R_{(n_0,j_0)+\La_{N_1}}\tilde{T}_{u^{(r_1)}}(\omega,\amplitude^{(r_1+1)}(\omega))R_{(n_0,j_0)+\La_{N_1}})^{-1}\nm\leq e^{N_1^{\frac{9}{10}}},
\end{equation}
and for $|n-n'|+|j-j'|>\sqrt{N_1},$
\begin{equation}\label{mcR_2 r_1 Green off-diag}
   |\big((R_{(n_0,j_0)+\La_{N_1}}\tilde{T}_{u^{(r_1)}}(\omega,\amplitude^{r_1+1}(\omega)) R_{(n_0,j_0)+\La_{N_1}})^{-1}\big)_{s,s'}(n,j;n',j')  | \leq e^{-c(|n-n'|+|j-j'|)}. 
\end{equation}
Therefore, we have established  the desired estimates on  Green's function on boxes of  size $N_1$ contained  in $\mcR_2$.\smallskip

 By summarizing  the above estimates: \eqref{mcR_1 r_1 Green norm}--\eqref{mcR_1 r_1 Green off-diag} in  \textbf{$\mcR_1$} and \eqref{mcR_2 r_1 Green norm}--\eqref{mcR_2 r_1 Green off-diag}   \textbf{$\mcR_2$}, we claim  that, for any $\La_{N_1}\in \mcE^0_{N_1}$,  $(n_0,j_0)\in [-N,N]^{b+d}$ with $\max\{|n_0|,|j_0|\}\geq \frac{M^r}{2}$, and  $\omega\in I_2\setminus(I_2^r\cup \tilde{I}_2^r),(\omega,\amplitude)\in \Gamma_{r_1}$, we have 
\begin{equation}\label{mcR r_1 Green norm}
	\nm(R_{(n_0,j_0)+\La_{N_1}}\tilde{T}_{u^{(r_1)}}(\omega,\amplitude^{(r_1+1)}(\omega))R_{(n_0,j_0)+\La_{N_1}})^{-1}\nm\leq e^{N_1^{\frac{9}{10}}},
\end{equation}
and for $|n-n'|+|j-j'|>\sqrt{N_1},$
\begin{equation}\label{mcR r_1 Green off-diag}
   |\big((R_{(n_0,j_0)+\La_{N_1}}\tilde{T}_{u^{(r_1)}}(\omega,\amplitude^{(r_1+1)}(\omega)) R_{(n_0,j_0)+\La_{N_1}})^{-1}\big)_{s,s'}(n,j;n',j')  | \leq e^{-c(|n-n'|+|j-j'|)}.
\end{equation}
From  $\Vert \Gamma_r-\Gamma_{r-1}\Vert\lesssim \delta\cdot\delta_{r}$  in Remark \ref{gammar}, we  have  
\begin{align*}
	\Vert \Gamma_r-\Gamma_{r_1} \Vert & \lesssim \delta\sum_{s=r_1+1}^{r}\delta_{s}\lesssim \delta_{r_1}\leq Ce^{-N_1^2}.
\end{align*}
And $\Vert \Delta_{cor}u^{(r)}\Vert \leq \delta_r$ gives 
\[\Vert u^{(r)}-u^{(r_1)}\Vert \leq \sum_{s=r_1+1}^{r}\delta_s\lesssim \delta_{r_1}\leq Ce^{-N_1^2}.\]
Then for any $\omega\in I_2\setminus (I_2^r\cup \tilde{I}_2^r)$ and $(\omega,\amplitude)\in \Gamma_r ,$ 
we have  for some $\tilde{\omega}\in I_2\setminus(I_2^r\cup\tilde{I}_2^r),(\tilde{\omega},\tilde{\amplitude})\in \Gamma_{r_1},$
\[|(\omega,\amplitude)-(\tilde{\omega},\tilde{\amplitude})|\leq \Vert \Gamma_r-\Gamma_{r_1}\Vert\leq Ce^{-N_1^2}.\]
So,  if we change $u^{(r_1)}$ to $u^{(r)}$, and change parameter from $\Gamma_{r_1}$ to $\Gamma_r$ in \eqref{mcR r_1 Green norm}--\eqref{mcR r_1 Green off-diag}, we have 
\begin{align*}
	|\big(\tilde{T}_{u^{(r)}}(\omega,\amplitude)-\tilde{T}_{u^{(r_1)}}(\tilde{\omega},\tilde{\amplitude})\big)_{s,s'}(n,j;n',j')| &\lesssim |u^{(r)}(\omega,\amplitude)-u^{(r_1)}(\tilde{\omega},\tilde{\amplitude})|\\ 
	    &\lesssim \Vert u^{(r)}-u^{(r_1)}\Vert+\Vert\partial u^{(r)}\Vert\cdot |(\omega,\amplitude)-(\tilde{\omega},\tilde{\amplitude})|\\ 
		&\leq Ce^{-N_1^2}.
\end{align*}
Hence, by the same argument as in  \eqref{perturbation 1}, \eqref{perturbation 2} and \eqref{perturbation by approimation to operator 1}, we obtain
\begin{equation}\label{mcR r Green norm}
	\nm(R_{(n_0,j_0)+\La_{N_1}}\tilde{T}_{u^{(r)}}(\omega,\amplitude^{(r+1)}(\omega))R_{(n_0,j_0)+\La_{N_1}})^{-1}\nm\leq e^{N_1^{\frac{9}{10}}},
\end{equation}
and for $|n-n'|+|j-j'|>\sqrt{N_1},$
\begin{equation}\label{mcR r Green off-diag}
   |\big[(R_{(n_0,j_0)+\La_{N_1}}\tilde{T}_{u^{(r)}}(\omega,\amplitude^{r+1}(\omega)) R_{(n_0,j_0)+\La_{N_1}})^{-1}\big]_{s,s'}(n,j;n',j')  | \leq e^{-c(|n-n'|+|j-j'|)}
\end{equation}
hold  for any $\La_{N_1}\in \mcE^0_{N_1}$, $(n_0,j_0)\in [-N,N]^{b+d}$ with $\max{|n_0|,|j_0|}\geq \frac{M^r}{2}$, and  $\omega\in I_2\setminus(I_2^r\cup \tilde{I}_2^r),(\omega,\amplitude)\in \Gamma_{r}$.\smallskip

Now,  we know that for any $\omega\in I_2\setminus (I_2^r\cup \tilde{I}^r_2),(\omega,\amplitude)\in \Gamma_r$,  Green's functions estimates   (involving $u^{(r)}$) \eqref{center block r+1 Green norm}, \eqref{center block r+1 Green off-diag}, \eqref{mcR r Green norm} and \eqref{mcR r Green off-diag} hold true. By applying the  resolvent identity as  in \cite{HSSY24} (cf. Lemma 3.6), 
we have that for all $(\omega,{\bf a})\in \cup_{I\in \mcI_r}I$ and $\omega\in I_2\setminus(I_2^r\cup \tilde{I}_2^r)$,
\begin{equation}\label{total r+1 Green norm}
	\nm(R_{[-M^{r+1},M^{r+1}]^{b+d}}\tilde{T}_{u^{(r)}}(\omega,\amplitude)R_{[-M^{r+1},M^{r+1}]^{b+d}})^{-1}\nm\leq\frac{1}{2} M^{(r+1)^C},
\end{equation}
and for $|n-n'|+|j-j'|>r^C,$
\begin{equation}\label{total r+1 Green off-diag}
|\big((R_{[-M^{r+1},M^{r+1}]^{b+d}}\tilde{T}_{u^{(r)}}(\omega,\amplitude) R_{[-M^{r+1},M^{r+1}]^{b+d}})^{-1}\big)_{s,s'}(n,j;n',j')  | \leq \frac{1}{2} e^{-c(|n-n'|+|j-j'|)}.
\end{equation}
Moreover, if we perturb  $(\omega,\amplitude)$ to $|(\omega_1,\amplitude_1)-(\omega,\amplitude)|\leq M^{-(1+r)^{10C}}$, the same argument as in  \eqref{perturbation 1}, \eqref{perturbation 2} and \eqref{perturbation by approimation to operator 1} will ensure that the Green's function estimates \eqref{total r+1 Green norm}--\eqref{total r+1 Green off-diag} still hold for $(\omega_1,\amplitude_1)$ {(with no  factor $\frac12$  on the right hand side)}. This implies that estimates  \eqref{total r+1 Green norm}--\eqref{total r+1 Green off-diag} essentially remain the same in a $M^{-(1+r)^{10C}}$-neighborhood of $(\omega,\amplitude)$. Finally, it suffices to  take account of  all such neighborhoods  and all possible $I_2$ to generate $\mcI_{r+1}$. We have thus proved \textbf{Hiv.} (4) for $r+1$.\smallskip

Finally, we turn to  \textbf{Hiv.} (5). On one hand, the Diophantine condition required by  Theorem \ref{LDT} makes that, for  $r=r_0-1$,
\begin{equation}\label{Diophantine begin step meas}
	\meas(DC((r_0\log M )^C)\setminus DC(|\log\delta|^K))\leq \delta^2 e^{-|\log\delta|^{K_1^{90}}},
\end{equation}
and for $r\geq r_0$, 
\begin{equation}\label{Diophantine r step meas}
	\meas(DC[(r+1)\log M ]^C)\setminus DC((r\log M)^C)\leq e^{-r^C}.
\end{equation}
By counting all possible $I_1$ (or $I_2$) (the total number is bounded by $M^{r_1^{10C}}$),  combining  measure estimates \eqref{meas tildeI_2^r},\eqref{meas I_2^r},\eqref{Diophantine begin step meas} and \eqref{Diophantine r step meas}, and $\cup_{I'\in \mcI_r}I' \subset \cup_{I_1\in \mcI_{r_1}}I_1 $,  we have that for $r\geq r_0$, 
\begin{equation}\label{proj omega r}
	\text{meas} (P_{\omega}(\Gamma_{r}\cap (\bigcup_{I'\in {\mcI}_{r}} I'\backslash \bigcup_{I\in {\mcI}_{r+1}}I)))\leq \delta^b M^{r_1^C} N_1^C M^{-\frac{r}{2^{b-1}}}+ e^{-r^C}\leq \delta^b M^{r_1^C} N_1^C M^{-\frac{r}{2^{b-1}}},
\end{equation}
and for $r=r_0-1,$
\begin{equation}\label{proj omega r_0}
	\text{meas} (P_{\omega}(\Gamma_{r}\cap (\bigcup_{I'\in {\mcI}_{r}} I'\backslash \bigcup_{I\in {\mcI}_{r+1}}I)))\leq \delta^b M^{r_1^C}N_1^C M^{-\frac{r}{2^{b-1}}}+ \delta^2 e^{-|\log \delta|^{K_1^{90}}}.
\end{equation}
Since  $\Gamma_r$ is obtained by solving the Q-equation  \eqref{Q-equation Explicit} and satisfies \eqref{approimation omega^r+1},  it's easy to get that $\det(\frac{\partial \omega}{\partial \amplitude})\sim\delta^b$. Hence,  we can transfer the estimates  \eqref{proj omega r} and \eqref{proj omega r_0} into projection on $\amplitude$: for $r\geq r_0,$
\begin{equation}\label{proj amplitude r}
	\text{meas} (P_{\amplitude}(\Gamma_{r}\cap (\bigcup_{I'\in {\mcI}_{r}} I'\backslash \bigcup_{I\in {\mcI}_{r+1}}I))) \leq \delta^{-b}\delta^b M^{r_1^C} N_1^C M^{-\frac{r}{2^{b-1}}}\leq M^{-\frac{r}{2^b}},
\end{equation}
and for $r=r_0-1,$
\begin{align}\label{proj amplitude r_0}
	\text{meas} (P_{\amplitude}(\Gamma_{r}\cap (\bigcup_{I'\in {\mcI}_{r}} I'\backslash \bigcup_{I\in {\mcI}_{r+1}}I))) & \leq \delta^{-b}\delta^b M^{r_1^C} N_1^C M^{-\frac{r}{2^{b-1}}} +\delta^{-b+2}e^{-|\log\delta|^{K_1^{90}}}\\
	 \notag &\leq M^{-\frac{r}{2^b}}+ e^{-|\log\delta|^{K_1^{80}}}.
\end{align}
That is \textbf{Hiv.} (5) for $r+1$.

Thus, we have finished the proof of  induction theorem. 
\end{proof}

Based on Theorem \ref{inductive}, we can prove our main theorem on the existence of Anderson localized states.

\begin{proof}[Proof of Theorem \ref{Main Thm}]

By \textbf{Hiv.} (5),  we can construct the set $\mcR=\mcR_{\bfa,\delta,\varepsilon,\theta}$ by taking intersections on \eqref{proj amplitude r}--\eqref{proj  amplitude r_0} for all $r\geq 1$ in Theorem \ref{inductive},  and 
\begin{align*}
	\meas([1,2]^b\setminus \mcR) &\leq \sum_{r\geq r_0-1}M^{-\frac{r}{2^b}}+e^{-|\log\delta|^{K_1^{80}}} \\
	          &\leq 2M^{-\frac{r_0-1}{2^b}}+e^{-|\log\delta|^{10}}\\
			  &\leq e^{-|\log\delta|^{\frac{1}{2}}}.
\end{align*}
Thus, Theorem \ref{Main Thm} follows immediately from the  Theorem \ref{inductive}.
\end{proof}

\appendix{}
\section{Lemma \ref{KPS} via  Rellich function analysis}\label{CSZapp}
In this section, we provide a proof of Lemma \ref{KPS}  (cf. \cite{KPS24}) based on MSA type Rellich function estimates from  \cite{CSZ24}. 
Roughly speaking, Rellich functions are certain eigenvalue functions (in $\theta$) of finite volume restrictions of $H(\theta).$
 We begin with some results proven in \cite{CSZ24}.   First, we consider  the induction parameters: 
\begin{itemize}
	\item $l_0=1,\delta_0=\varepsilon^{\frac{1}{20}},l_1=[|\log \delta_0|^4],\gamma_0=\frac{1}{2}|\log \varepsilon|.$
	\item For $n\geq 1$, $l_{n+1}=l_{n}^2,\delta_n=e^{-l_n^{\frac{2}{3}}},\gamma_n=\prod_{k=1}^{n}(1-l_k^{-\frac{1}{80}})\gamma_0.$
	\item  The above settings  imply  there is some absolute constant $\gamma_\infty$ such that 
	        \[\gamma_n\geq \gamma_{\infty}=\prod_{k=1}^{\infty}(1-l_k^{-\frac{1}{80}})\gamma_0\geq \frac{1}{2}\gamma_0.\]
\end{itemize}

 Recall that $Q_l\subset\Z^d$ denotes a ball in the $\ell^\infty$-norm of center $0$ and radius $l>0.$ We have 
\begin{state}[Rellich function estimates, \cite{CSZ24}]
Let  $H(\theta)=\varepsilon\Delta+V(\theta+j\cdot\bfa)\delta_{j,j{'}}$ with a  Lipschitz monotone potential  $V(\theta)$ (here $L$ is the Lipschitz constant). Then we can construct Rellich functions  and  establish quantitative Green's function estimates as follows. 
\begin{itemize}
	\item \textbf{(Induction  blocks)} At the $n$-th step, we can construct    a  block $B_n\subset \Z^d$ satisfying 
	       \[ B_0=\{0\}, \ B_{1}=Q_{l_1},\]
	       \[Q_{l_n}\subset B_n\subset Q_{l_n+50l_{n-1}}\ {\rm for}\ n\geq 2.\]

	\item \textbf{(Rellich functions)} The Dirichlet restriction $H_{B_n}(\theta)$ has a $1$-periodic, real-valued Rellich function $E_n(\theta)$ satisfying the following properties:\\
	       (1) $E_0(\theta)=V(\theta)$;\\
		   (2) For $n\geq 1$,  
		      \[|E_n(\theta)-E_{n-1}(\theta)|\leq e^{-l_{n-1}},\]
			  where $e^{-l_0}:=\varepsilon$. In fact, $E_n(\theta)$ is the unique eigenvalue of $H_{B_n}(\theta)$ in $\disk(E_{n-1}(\theta),10\delta_{m-1}):=\{E:\ |E-E_{n-1}(\theta)|\leq 10\delta_{m-1}\};$\\
		   (3) For all $n\geq 0, E_n(\theta)$ is non-decreasing  and Lipschitz monotone with a  Lipschitz constant $L$, i.e.,
		     \[E_n(\theta_1)-E_n(\theta_2)\geq L(\theta_1-\theta_2),\ 0\leq \theta_2\leq \theta_1<1.\] 
	\item \textbf{(Good and bad sets.)} For $j\in \Z^d$,  denote $B_n(j):=B_n+j$. Fix $\theta\in \R, E\in \mathbb{R}$. Define the $n$-th  step resonant  set as 
	   \[S_n(\theta,E)=\big\{ j\in\Z^d:\ |E_n(\theta+j\cdot \bfa)-E|<\delta_n\big\}.\]
	   We have  $S_{n+1}(\theta,E)\subset S_n(\theta,E)$.  Let $\La\subset \Z^d$ be a finite set. Related to $(\theta,E)$, we say\\
	   (1) $\La$ is $n$-nonresonant if $\La\cap S_n(\theta,E)=\emptyset$;\\
	   (2) $\La$ is $n$-regular if $j\in \La\cap S_k(\theta,E)\Rightarrow B_{k+1}(j)\subset\La$ for $0\leq k\leq n-1$;\\
	   (3) $\La$ is $n$-good if it is both $n$-nonresonant and $n$-regular. \\
	   Particularly,  if $j\in S_n(\theta,E)$, then $B_n(j)\backslash\{j\}$ is $(n-1)$-good.
	\item \textbf{(Green's function estimates)} Fix $\theta\in \R, E\in \mathbb{R}$. Assuming  that a finite set $\La$ is $n$-good related to $(\theta,E)$, then for  the  Green's function $G_{\La}(\theta,E)=(H_{\La}(\theta)-E)^{-1}$ we have 
	   \begin{align}
		\label{Green l2 estimate}\nm G_{\La}(\theta,E)\nm&\leq 10\delta_n^{-1},\\
	\label{Green off-diag estimate}|G_{\La}(\theta,E)(x,y)|&\leq e^{-\gamma_n |x-y|_1}  \ {\rm for}\quad |x-y|_1\geq l_n^{\frac{5}{6}}.
	   \end{align}
	   We remark that the above estimates hold for all $n\geq 0$.
\end{itemize}
\end{state}

The above statements hold for $\alpha\in\DC, \theta\in\T$ and $0<\varepsilon\leq \varepsilon_0(\gamma,\tau,d,L)$. In the following,  an argument for ``${\forall} \theta\in\R$'' typically represents  an almost everywhere one,  because we  will remove  those $\theta$ such that $|V(\theta+j\cdot\bfa)|=\infty$. 
We have the following steps. \smallskip

\textbf{Step 1. Approximating  the eigenvalue with Rellich function}\\
Now fix any $\ct \in \R$.  Assume that $\phi(\ct)$ is an eigenfunction of $H(\ct)$ with an eigenvalue $\mu(\ct)$.  Since  $\theta$ has been fixed, we just omit the dependence on $\theta$ of $H,\phi,\mu$. Choose one localization center $l\in \Z^d$ of $\phi(\ct)$. That is,
\[|\phi(l)|=\nm \phi\nm_{\infty}.\]
Without loss of generality,  we assume $\phi(l)=1$ and hence $|\phi(j)|\leq 1\ {\rm for}\ {\forall}j\in\Z^d$. Consider 
\[(H-\mu)\phi=0.\]
From $[(H-\mu)\phi](l)=0$, we obtain
\[(V(\theta+l\cdot\bfa)-\mu(\theta))\phi(l)+\varepsilon\sum_{|j-l|_1=1}\phi(j)=0.\]
Thus, we get  
\begin{align*}
		|V(\theta+l\cdot\bfa)-\mu(\theta)| &\leq \varepsilon \sum_{|j-l|_1=1}\phi(j)\\
		                                   &\leq 2d\varepsilon \ll \delta_0=\varepsilon^{\frac{1}{20}}
\end{align*}
provided $0<\varepsilon<c(d)\ll1$. Hence, $l\in S_0(\ct,\mu)$. We will  prove inductively  that 
\begin{equation}\label{localization center is bad}
	l\in S_n(\ct,\mu) \ {\rm for}\ {\forall } n\geq 0.  
\end{equation}
Now $(\ref{localization center is bad})$  holds for $n=0$ as shown above.

Assume $l\in S_n(\ct,\mu)$. Since  $l$ is the localization center,  the same argument  as in Claim 4.7 in  \cite{CSZ24} shows 
\begin{equation}\label{CSZ Claim 4.7}
	S_{n+1}(\ct,\mu)\cap Q_{50l_{n+1}}(l)\neq \emptyset.
\end{equation}
Moreover,  the Diophantine condition on $\bfa$ ensures the \textbf{separation  property} of $S_n(\ct,\mu)$:  for $j\neq j{'}\in S_n(\ct,\mu)$, we have 
\begin{align*}
	2\delta_n  &>|E_n(\ct+j\cdot \bfa)-\mu|+|E_n(\ct+j{'}\cdot \bfa)-\mu|\\
	           &\geq|E_n(\ct+j\cdot \bfa)-E_n(\ct+j{'}\cdot \bfa)| \\
			   &\geq L\nm (j-j{'})\cdot \bfa\nmt\\
			   &\geq \frac{L\gamma}{|j-j{'}|^{\tau}} \geq \frac{L\gamma}{|j-j{'}|_1^{\tau}}.
\end{align*}
This shows 
\[|j-j{'}|>(\frac{L\gamma}{2\delta_n})^{\frac{1}{\tau}}=(\frac{L\gamma}{2})^{\frac{1}{\tau}}e^{\frac{1}{\tau}l_n^{\frac{2}{3}}}\gg 100l_{n+1}=100l_n^2\]
provided  $0<\varepsilon<c(L,\gamma,\tau)\ll1$. Thus, from the  induction assumption, we have 
\[S_n(\theta,\mu)\cap Q_{50l_{n+1}}(l) = \{l\}.\]
By $S_{n+1}(\theta,\mu)\subset S_n(\theta,\mu)$ and $(\ref{CSZ Claim 4.7})$, we  obtain
\[S_{n+1}(\theta,\mu)\cap Q_{50l_{n+1}}(l) = \{l\}.\]
This  means $l\in S_{n+1}(\ct,\mu)$, and $(\ref{localization center is bad})$ is proved.\smallskip

\textbf{Step 2. Exponential decay of eigenfunctions}\\
We fix any $\ct\in\R$ and assume that the eigenfunction $\phi$ has a localization center $l$ (with the eigenvalue $\mu$). We further assume that $\nm\phi\nm_{\ell^2(\Z^d)}=1$, which implies $\nm \phi\nm_{\infty}\leq  1$.  For any $x\neq l\in \Z^d$, we have

\textbf{(Case 1)} $1\leq |x-l|_1\leq \frac{1}{2}l_1$. Then  $|x-l|\leq |x-l|_1\leq \frac{1}{2}l_1$. Let $\La=B_1(l)\backslash\{l\}$. Since  $(\ref{localization center is bad})$,  we have  $l\in S_1(\ct,\mu)$. Then $\La$ is 0-good.  Moreover, 
\begin{align}
	\label{1-far 1}|x-w|_1&\geq |x-w|>\frac{1}{2}l_1\gg l_0^{\frac{5}{6}}\  {\rm for}\ {\forall} w\in \pa_- B_1(l),\\
	\label{1-far 2}|x-l|_1&\geq 1= l_0^{\frac{5}{6}},
\end{align}
where  for a set $B\subset\Z^d,$
\[\pa_+B=\big\{w\in \Z^d\setminus B:\ \exists w'\in B\ {\rm s.t.,} \ |w'-w|_1=1\big\},\]
\[\pa_-B=\big\{w\in B:\ \exists w'\in \Z^d\setminus B\ {\rm s.t.,} \ |w'-w|_1=1\big\}.\]
For simplicity, we denote $G_{\La}=G_{\La}(\theta,\mu)$. Then alppying the Poisson's  formula and  $(\ref{Green off-diag estimate})$ yields 
\begin{align*}
	|\phi(x)| &=\varepsilon \bigg| \sum_{\substack{w\in \pa_-\La,w'\in \pa_+\La \\ |w-w'|_1=1}}G_{\La}(x,w)\phi(w')\bigg| \\
              &\leq \varepsilon \sum_{|w-l|_1=1} |G_{\La}(x,w)|+\varepsilon\sum_{\substack{w\in \pa_- B_1(l),w'\in \pa_+ B_1(l) \\ |w-w'|_1=1}}|G_{\La}(x,w)||\phi(w')| \\
			  &\leq \varepsilon \sum_{|w-l|_1=1} e^{-\gamma_0 |x-w|_1}+ 2d\varepsilon  \sum_{w\in \pa_- B_1(l)}e^{-\gamma_0|x-w|_1}\\
			  &\leq 10\delta_0^{-1} \varepsilon \sum_{|w-l|_1=1} e^{-\gamma_0 (|x-l|_1-1)}+ 2d\varepsilon  \sum_{w\in \pa_- B_1(l)}e^{-\gamma_0|x-l|_1}\\
			  &\leq (20d \delta_0^{-1}\varepsilon e^{\gamma_0}+ 2d\varepsilon\cdot \# (\pa_- B_1))e^{-\gamma_0|x-l|_1}\\
			  &\leq e^{-\frac{1}{2}\gamma_0|x-l|_1}\leq e^{-\frac{1}{2}\gamma_{\infty}|x-l|_1}.
\end{align*}

\textbf{(Case 2)}  $\frac{1}{2}l_n<|x-l|_1<\frac{1}{2}l_{n+1}$ for some $n\geq 1$. In this case, we  have $|x-l|\leq |x-l|_1\leq \frac{1}{2}l_{n+1}$. Let $\La= B_{n+1}(l)\backslash\{l\}$. Since  $(\ref{localization center is bad})$,  we know $l\in S_{n+1}(\ct,\mu)$. Then $\La$ is $n$-good. Moreover, we have 
\begin{align*}
|x-w|_1&\geq |x-w|>\frac{1}{2}l_{n+1} \gg l_n^{\frac{5}{6}}\ {\rm for}\ {\forall} w\in \pa_- B_{n+1}(l),\\
 |x-l|_1&\geq \frac{1}{2}l_n \Rightarrow |x-w|_1\geq \frac{1}{2}l_n-1\gg l_n^{\frac{5}{6}}\ {\rm for}\ {\forall} |w-l|_1=1.
\end{align*}
Hence, by applying Poisson's formula and using similar estimates  as in {\bf  (Case 1)},  we obtain  
\begin{align*}
	|\phi(x)| &=\varepsilon \bigg| \sum_{\substack{w\in \pa_-\La,w'\in \pa_+\La \\ |w'-w|_1=1}}G_{\La}(x,w)\phi(w')\bigg| \\
              &\leq \varepsilon \sum_{|w-l|_1=1} |G_{\La}(x,w)|+\varepsilon\sum_{\substack{w\in \pa_- B_{n+1}(l),w'\in \pa_+ B_{n+1}(l) \\ |w'-w|_1=1}}|G_{\La}(x,w)||\phi(w')| \\
			  &\leq (2d\varepsilon e^{\gamma_n}+ 2d\varepsilon\cdot \# (\pa_- B_n))e^{-\gamma_n|x-l|_1}\\
			  &\leq e^{-\frac{1}{2}\gamma_n|x-l|_1}\leq e^{-\frac{1}{2}\gamma_{\infty}|x-l|_1}
\end{align*}

Combining estimates in the above two cases yields 
\[|\phi(x)|\leq e^{-\frac{1}{2}\gamma_{\infty}|x-l|_1}\leq e^{-\frac{1}{4}\gamma_0|x-l|_1}=\varepsilon^{\frac{1}{8} |x-l|_1}\ {\rm for}\ {\forall} x\neq l.\]
Recall that  $\phi$ is normalized. We have 
\begin{align*}
	1-|\phi(l)|^2  &=\sum_{x:\ x\neq l}|\phi(x)|^2\\
	            &\leq \sum_{x:\ x\neq l}\varepsilon^{\frac{1}{4}|x-l|_1}\\
				&\leq \sum_{k=1}^{\infty} (2k+1)^d\varepsilon^{\frac{k}{4}} \leq \varepsilon^{\frac{1}{5}}
\end{align*}
provided $0<\varepsilon\leq c(d)\ll1$.  Thus, $|\phi(l)-1|\leq \varepsilon^{\frac{1}{5}}$.\smallskip

\textbf{Step 3. The completeness of eigenfunctions}\\
We have proven that for any fixed $\theta$ and any normalized  eigenfunction $\phi(\theta)$ (with corresponding  eigenvalue $\mu(\theta)$) with localization center $l$,  
\begin{align}
\label{Approximate By Rellich}l\in S_n(\theta,\mu(\theta))&\Rightarrow |\mu(\theta)-E_n(\theta+l\cdot\bfa)|<\delta_n\ {\rm for}\ {\forall}n\geq 0,\\
	\label{Explicit localization} |\phi(\theta, l)-1|&\leq \varepsilon^{\frac{1}{5}} ,\  |\phi(\theta, x)|\leq \varepsilon^{\frac{1}{8} |x-l|_1}\ {\rm for}\ {\forall} x\neq l\in \Z^d,
\end{align}
and hence, $l$ is the unique localization center of $\phi(\theta) $. Below we are  intended to show 
\begin{claim}\label{Simple relabeling.}
	For any $\theta \in \R$ and each lattice site $j \in \Z^d$, there exists   a  unique eigenfunction $\phi_j(\theta)$ of $H(\theta)$ peaking at $j$, namely,  $j$ is the unique localization center of $\phi_j(\theta)$ and if there is another eigenfunction $\psi(\theta)$ peaking at $j$, then $\psi(\theta)=\phi_j(\theta)$. 
\end{claim}
\begin{proof}[Proof of Claim A.1.]
	Denote by $\{\psi_s(\theta)\}_{s\in \Z^d}$  the complete set of orthogonal  normalized eigenfunctions  (without relabelling)  of $H(\theta)$ in $\ell^2(\Z^d)$ (this is a consequence of main theorem in \cite{CSZ24}). For simplicity,  we hide the $\ct$ dependence. By Parseval's  equality, if $\{e_x\}_{x\in\Z^d}$ is the standard basis of $\ell^2(\Z^d)$, we have  for any $x\in\Z^d,$
	\begin{align}\label{parse}
	\nm e_x\nm_{\ell^2(\Z^d)}^2=\sum_{s\in \Z^d}|\psi_s(x)|^2=1.
	\end{align}
	Suppose that there are two different $\psi_{s_1},\psi_{s_2}$ with  the same localization center $x$. Then by $(\ref{Explicit localization}),$
	\[ \sum_{s\in \Z^d}|\psi_s(x)|^2\geq |\psi_{s_1}(x)|^2+|\psi_{s_2}(x)|^2\geq 2-2\varepsilon^{\frac{1}{5}}>1,\]\
	which contradicts \eqref {parse}. 
	
	Next, suppose there is no $\psi_s$ peaking at $x$.  Then by $(\ref{Explicit localization})$ and the fact  that different eigenfunctions cannot  peak at  the same lattice site, 
	\[\sum_{s\in \Z^d}|\psi_s(x)|^2 \leq \sum_{l\in \Z^d:\ l\neq x}\varepsilon^{\frac{1}{8}|l-x|_1}\leq \sum_{k=1}^{\infty}(2k+1)^d\varepsilon^{\frac{k}{8}}\ll 1,\]
	which also contradicts \eqref{parse}.
\end{proof}

Thus, Claim $\ref{Simple relabeling.}$, together with $(\ref{Explicit localization})$, ensures the completeness of $\{\phi_j(\theta)\}_{j\in\Z^d}$. In other words, we construct a relabelling of the orthogonal eigenfunctions basis of $H(\theta)$ such that the index matches with the unique localization center.

Finally, for each $\theta$,  pick  the unique eigenfunction $\phi(\theta)$ whose localization center is exactly $0$, and denote its eigenvalue as $E(\theta)$. Then $(\ref{Explicit localization})$ becomes 
\[	|\phi(\theta, 0)-1|\leq \varepsilon^{\frac{1}{5}} , |\phi(\theta, x)|\leq \varepsilon^{\frac{1}{8} |x|_1}\ {\rm for}\ 
 {\forall} x\in \Z^d\setminus\{0\}. \]
Moreover, by $(\ref{Approximate By Rellich})$, 
\begin{align*}
	E(\theta_1)-E(\theta_2) &\geq E_n(\ct_1)-E_n(\ct_2)-|E(\theta_1)-E_n(\ct_1)|-|E(\theta_2)-E_n(\ct_2)|\\
                         &\geq L(\theta_1-\ct_2)-2\delta_n \ {\rm for}\  0\leq\theta_2\leq \ct_1<1. 
\end{align*}
Letting $n\rightarrow\infty$ in the above inequality,  we have $E(\theta_1)-E(\theta_2)\geq L(\theta_1-\ct_2)$. Therefore, $\phi(\theta),E(\theta)$ are  what we want to find in Lemma $\ref{KPS}$.

\begin{rmk}
	The only difference between what we prove in this section and that in Lemma  $\ref{KPS}$ is  the exponential decay rate. Here, it  is $\frac{1}{8}$, while  in Lemma $\ref{KPS}$, it  is $(1-\eta)$ and can be arbitrarily  close to $1$ provided that $0<\varepsilon\ll1$). This minor issue, however, can also  be resolved (cf.  Remark 1.5,   \cite{CSZ24}). 
	\end{rmk}

\section{Relabeling the localized eigenfunctions}\label{Hallapp}

This section presents a more general method for relabelling the localized eigenfunctions of  Schr\"odinger operators in  higher dimensions.  First, assume that the Schr\"odinger operator $H$ on $\Z^d$ exhibits Anderson localization, meaning it has a pure point spectrum with exponentially decaying eigenfunctions $\phi=\{\phi(j)\}_{j\in\Z^d}\in \ell^2(\Z^d)$).  For $\phi\in\ell^2(\Z^d)$, denote
 the set of localization centers  of $\phi$ by 
 \[\LC(\phi):=\left\{l\in \Z^d:\ |\phi(l)|=\max_{j\in\Z^d}|\phi(j)|\right\}.\] 
 
We emphasize that in the case of Lipschitz monotone potentials, Lemma $\ref{KPS}$ shows  that each eigenfunction has  only one localization center. However,  for Schr\"odinger operators with general potentials,  the set $\LC(\phi)$ can contain many elements.  Therefore, we require the eigenfunctions to have a semi-uniform localization control:
 
\begin{assumption}[Semi-uniform Localization Eigenfunctions (SLUE)]\label{SULEass}
We say that $H$ has SLUE property, if  there are  some constants  $C,c,q>0$ such that  
\begin{equation}\label{uniform localization control}
	|\phi(x)|\leq C(1+|l|)^q e^{-c|x-l|}\ {\rm for}\ {\forall}l\in \LC(\phi)\ {\rm and }\ \forall x\in\Z^d
\end{equation} 
holds for all eigenfunctions  $\phi$ (of $H$) with $\nm \phi\nm_{\ell^2(\Z^d)}=1$, where $C,c,q$ are all independent of $\phi$.
\end{assumption}
 If  the bound \eqref{uniform localization control} becomes  $Ce^{-c|x-l|}$,  we call such property the  ``uniform localization eigenfunctions" (ULE).  Both SULE and ULE were  discussed in \cite{dJLS96}. 
 
Some important models satisfying SULE include: 

\begin{ex} 
	\begin{itemize}
		\item \textbf{The i.i.d. random potentials case} 
		  (cf. sect.1.6,  \cite{GK14}).  Assume that the potential $\{V_j\}_{j\in \Z^d}$consists of i.i.d. random variables with a bounded distribution density $g$. Consider 
		 \[H=\varepsilon\Delta+V_j\delta_{j,j'}.\]
		For any fixed small $varepsilon$(in fact, in one dimension, such a restriction is not necessary), Anderson localization holds. Consequently, there exist some $c,q>0$ such that, with probability one,
		            \[|\phi(x)|\leq C_V(1+|l|)^q e^{-c|x-l|}\quad {\forall}l\in \LC(\phi)\ {\rm and}\ \forall x\in\Z^d\]
		for any normalized eigenfunction $\phi$,  where  $\mathbb{E}(C_V)<\infty$. 		
		
		\item \textbf{The Lipschitz monotone potentials case} (cf. \cite{KPS24}).    
		In fact, Lemma $\ref{KPS}$ shows that even ULE property  holds in this case.  	
		\end{itemize}
\end{ex}

Below, we aim to construct a relabelling of eigenfunctions(cf. Theorem \ref{relabelling thm})  under the assumption that the SULE property holds true.
The construction of the relabelling is inspired by the one-dimensional relabelling argument in \cite{LW24} (cf. Appendix A), where the authors used the standard ordering on $\Z$ to define the relabelling. However, for $d>1$,  there is no such  an ordering  on  $\Z^d$.  Instead, we will use the annulus structure to construct the relabelling map.\smallskip

\textbf{Step 1.  Equidistribution of localization centers}\\
For scale $L\in \mathbb{N}_+$ and $k\in\Z^d$, let  
\[\La_{k\nearrow L}=(k_1,k_1+L]\times(k_2,k_2+L]\times \cdots \times(k_d,k_d+L]\subset \Z^d.\]
Denote by $\{\psi_s\}_{s\in \Z^d}$  the   orthogonal  basis given by normalized eigenfunctions of $H$ (without relabelling). We have 
\begin{lem}\label{equi-distri of LC}
	Assume $H$ satisfies  Assumption \ref{SULEass}. Then for any small $0<\epsilon<1$, there exists  some $L_{\epsilon}$ sufficiently large, depending  only on $\epsilon,C,c,q,d$, such that,  for any $L\geq L_{\epsilon}$ and $k\in [-L^4,L^4]^d$,  we have 
	\begin{equation}
		(1-\epsilon)^d L^d\leq \#\left\{s\in\Z^d:\ \LC(\psi_s)\cap \LakL\neq \emptyset\right\}\leq (1+\epsilon)^d L^d. 
	\end{equation}
\end{lem}
\begin{proof}[Proof of Lemma B.3.]
	For each $s\in\Z^d$,  if $\LC(\psi_s)\cap\LakL\neq \emptyset$,  we take one $l_s\in \LC(\psi_s)\cap\LakL$; otherwise,  take an arbitrary one  $l_s\in \LC(\psi_s)$ so that $l_s\notin\LakL$.  By Parseval's  equality, we have 
	\[\sum_{l\in \Z^d}|\psi_s(l)|^2=1\ {\rm for}\ {\forall} s\in\Z^d, \ \sum_{s\in\Z^d}|\psi_s(l)|^2=1\ {\rm for}\ {\forall}l\in\Z^d.\]
	Assume $L\gg1$  and  denote by 
	\[\tLakL=(k_1-\epsilon L,k_1+(1+\epsilon)L]\times \cdots \times (k_d-\epsilon L,k_d+(1+\epsilon)L]\]
	the extension  of $\LakL$ by   $\epsilon L$. Then for $s\in \{s\in\Z^d:\ \LC(\psi_s)\cap \LakL\neq \emptyset\}$ such that $l_s\in \LakL\subset Q_{2L^4}$, we have by \eqref{uniform localization control} 
	\begin{align*}
		\sum_{l \notin \tLakL }|\psi_s(l)|^2 &\lesssim_{C,q} \sum_{\substack{ m\in \Z^d \\ |m|\geq \epsilon L}}L^{8q} e^{-2c|m|}\\
		                            &\lesssim_{C,q,d} L^{8q} \sum_{\substack{m\in \Z \\ m\geq \epsilon L}}m^{d-1}e^{-2cm}\\
									&\lesssim_{C,q,d} L^{8q} \sum_{\substack{m\in \Z \\ m\geq \epsilon L}}m^{d-1}e^{-\frac{2}{3}cm-\frac{4}{3}c\epsilon L}\\
									&\lesssim_{C,c,q,d} L^{8q}e^{-\frac{4}{3}c\epsilon L} \leq e^{-c\epsilon L}
	\end{align*}
provided  $L>C(\epsilon,C,c,q,d)\gg1$. Hence,
	\[\sum_{l\in \tLakL}|\psi_s(l)|^2=1-\sum_{l\notin \tLakL}|\psi_s(l)|^2\geq 1-e^{-c\epsilon L}\]
	holds for any $s\in \{s\in\Z^d:\ \LC(\psi_s)\cap \LakL\neq \emptyset\}$. Taking the  summation on  those  $s$ leads to 
	\begin{align*}
		(1-e^{-c\epsilon L}) &\#\{s\in\Z^d:\ \LC(\psi_s)\cap \LakL\neq \emptyset\}\\
		 &\leq \sum_{\substack{l\in\tLakL \\ s:\ \LC(\psi_s)\cap \LakL\neq \emptyset}}|\psi_s(l)|^2\\
		 &\leq \sum_{\substack{l\in \tLakL \\ s\in\Z^d}}|\psi_s(l)|^2 =\# \tLakL \\
		 &\leq (1+2\epsilon)^d L^d.
	\end{align*}
	That is,
	\begin{equation}\label{equi-distri upper bound}
		\#\{s\in\Z^d:\ \LC(\psi_s)\cap \LakL\neq \emptyset\}\leq \frac{(1+2\epsilon)^d}{1-e^{-c\epsilon L}}L^d\leq (1+3\epsilon)^d L^d
	\end{equation}
	provided   $L>C(\epsilon,c,d)\gg1$.Next, denote by 
	\[\hLakL=(k_1+\epsilon L,k_1+(1-\epsilon)L]\times \cdots \times (k_d+\epsilon L,k_d+(1-\epsilon)L]\]
	 the shrink of $\LakL$ by  $\epsilon L$. Then for any fixed $l\in \hLakL$,  we have since  \eqref{uniform localization control}
	\begin{align}\label{before lower bound}
		\sum_{s:\ \LC(\psi_s)\cap \LakL =\emptyset} |\psi_s(l)|^2 
		      &\lesssim_C \sum_{s:\ \LC(\psi_s)\cap \LakL =\emptyset} (1+|l_s|)^{2q}e^{-2c|l-l_s|} \\
			\notag	 &\lesssim_C \sum_{m=0}^{\infty}\sum_{\substack{\ s:\ mL\leq |l_s|\leq (m+1)L \\ |l_s-l|\geq \epsilon L}} (1+|l_s|)^{2q}e^{-2c|l-l_s|}\\
			\notag	 &\lesssim_{C,q} \sum_{m\leq 10L^4} \sum_{\ s:\ |l_s|\leq 20 L^5} L^{10q}e^{-2c\epsilon L} \\
			\notag	 &   \ \  +\sum_{m>10L^4}\sum_{\ s:\ mL\leq |l_s|\leq (m+1)L} (1+|l_s|)^{2q}e^{-2c|l-l_s|}
	\end{align}
	In the above summation, $l_s\in \mathcal{LC}(\psi_s)$ is the element having the smallest length. The summation of second term after the last inequality doesn't involve $|l_s-l|>\epsilon L$ because
	\[m>10L^4,l\in \hLakL,|k|\leq L^4,|l_s|\geq mL\Rightarrow |l_s-l|\geq \frac{1}{2}mL> \epsilon L.\]
	So, this summation index is automatically holds. 
	By \eqref{equi-distri upper bound}  (with scale replaced by $20L^5$),  we  have shown  
	\[\#\{s:\ |l_s|\leq 20L^5\}\leq (1+3\epsilon)^d(20L^5)^d.\]
	Hence, 
	\begin{align*}
		\sum_{m\leq 10L^4} \sum_{\ s:\ |l_s|\leq 20 L^5} L^{10q}e^{-2c\epsilon L} &\lesssim L^4\cdot (1+3\epsilon)^d(20L^5)^d L^{10q}e^{-2c\epsilon L}\\
		      &\lesssim_{q,d}L^{4+5d+10q}e^{-2c\epsilon L}\\
			  &\leq e^{-c\epsilon L}.
	\end{align*}
	The last inequality needs that $L>C(\epsilon,c,q,d)\gg1$. On the other hand, applying again  \eqref{equi-distri upper bound}  (with scale replaced by $(m+1)L$) implies  
	\begin{align*}
		\sum_{m>10L^4} & \sum_{\ s:\ mL\leq |l_s|\leq (m+1)L} (1+|l_s|)^{2q}e^{-2c|l-l_s|} \\
		           &\leq\sum_{m> 10L^4}(1+3\epsilon)^d((m+1)L)^d(1+(m+1)L)^{2q}e^{-cmL}\\
				   &\lesssim_{q,d}\sum_{m>10L^4}(mL)^{d+2q}e^{-cmL}\\
				   &\leq e^{-\frac{1}{2}c\cdot 10L^5}=e^{-5cL^5}
	\end{align*}
	provided  $L>C(c,q,d)\gg1$. So,  \eqref{before lower bound}  satisfies  
	\[ \sum_{s:\ \LC(\psi_s)\cap \LakL =\emptyset} |\psi_s(l)|^2\lesssim_{C,q} e^{-c\epsilon L}+e^{-5cL^5}\leq e^{-\frac{1}{2}c\epsilon L}. \]
    Taking  summation on all $l\in \hLakL$ yields 
	\begin{align*}
		e^{-\frac{1}{2}\epsilon L}\#\hLakL &\geq \sum_{\substack{s:\ \LC(\psi_s)\cap \LakL =\emptyset \\ l\in\hLakL}} |\psi_s(l)|^2 \\
		  &= \#\hLakL- \sum_{\substack{s:\ \LC(\psi_s)\cap \LakL \neq \emptyset \\ l\in\hLakL}} |\psi_s(l)|^2\\
		  &\geq \#\hLakL- \sum_{\substack{s:\ \LC(\psi_s)\cap \LakL \neq \emptyset \\ l\in\Z^d}} |\psi_s(l)|^2\\ 
		  &= \#\hLakL -\#\{s\in\Z^d:\ \LC(\psi_s)\cap \LakL\neq \emptyset\}.
	\end{align*}
	That is,
	\begin{align}\label{equi-distri lower bound}
		\#\{s\in\Z^d:\ \LC(\psi_s)\cap \LakL\neq \emptyset\} &\geq (1-e^{-\frac{1}{2}c\epsilon L})\#\hLakL \\
		       \notag          &\geq (1-e^{-\frac{1}{2} c \epsilon L})(1-2\epsilon)^d L^d\\
		       \notag &\geq (1-3\epsilon)^d L^d
	\end{align}
    provided  $L>C(\epsilon,c)\gg1$. Finally, combining \eqref{equi-distri upper bound} and   \eqref{equi-distri lower bound}, and just replacing $\epsilon$ with $\frac{\epsilon}{3}$  will  finish the proof. Summarizing  all restrictions  on the largeness of $L$ shows  that $L_\epsilon$  can depend only  on $\epsilon,C,c,q,d$.
\end{proof}

\begin{rmk}
\begin{itemize}
\item[]{(1)} From  the above proof,  one  has  that the $L^4$ can  be replaced by $C(\epsilon) L$ for any large constant $C(\epsilon)$ depending  only on $\epsilon$ (for example $\frac{100}{\epsilon}$).   
 Also, It suffices to choose $L_{\epsilon}\gg1$   to ensure 
	\[(C(\epsilon)L_{\epsilon})^{2q} e^{-\frac{1}{4}c\epsilon L_{\epsilon}}<1.\]
	This  will be helpful in \textbf{Step 2} in the following. In fact,  largeness of  $L_{\epsilon}$ causes no influence in Remark B.1 (2) below, because we always fix $\epsilon$ and let $c\to+\infty$. 
	\item[]{(2)} In fact, it's easy to see that $\lim_{c\to +\infty}L_{\epsilon}(\epsilon,C,c,q,d)=0$. This can be intuitively understood. If $c\gg 1$, then the eigenfunctions will decay very fast and become more like the Dirac function $e_l$. The localization centers of the standard basis $\{e_l\}_{l\in\Z^d}$ are  strictly equidistributed on $\Z^d$, which indicates  the $L_{\epsilon}$ can be very small. Usually, in \cite{CSZ24} and \cite{KPS24}, $c\sim |\log\varepsilon|\to\infty$ as $\varepsilon\to0$, where $H=\varepsilon\Delta+V\delta_{n,n'}$.
    \item[]{(3)}  Moreover, we can choose $L_{\epsilon}\gg1$  such that,  if $\psi_s$ has a  localization center $|x_s|\leq \frac{10}{\epsilon}L_{\epsilon}$, then $\LC(\psi_s)\subset Q_{L_{\epsilon}}(x_s)$. 
    Indeed,  since  $\nm\psi_s\nm_{\ell^2(\Z^d)}=1$, consider 
	\[\sum_{x\notin Q_{L_{\epsilon}}(x_s)} |\psi_s(x)|^2:=G.\]
	If we can prove
	\[G< \frac{\sum_{x\in Q_{L_{\epsilon}}(x_s)} |\psi_s(x)|^2}{\# Q_{L_{\epsilon}}(x_s)}=\frac{1-G}{(2L_{\epsilon}+1)^d},\]
	then this together with  the fact that  $x_s$ is the maximum point,  shows for any $x\notin Q_{L_{\epsilon}}(x_s),$
	\[|\psi_s(x)|^2\leq G< \frac{\sum_{x\in Q_{L_{\epsilon}}(x_s)} |\psi_s(x)|^2}{\# Q_{L_{\epsilon}}(x_s)}\leq |\psi_s(x_s)|^2\]
	which implies that $x$ cannot be a maximum  point. Thus, it suffices  to prove  $G< \frac{1}{1+(1+2L_{\epsilon})^d}$. 
	In fact,  applying  \eqref{uniform localization control} deduces 
	\begin{align*}
		\sum_{x\notin Q_{L_{\epsilon}}(x_s)} |\psi_s(x)|^2 &\lesssim_{C,\epsilon,q}L^{2q}_{\epsilon}\sum_{x:\ |x-x_s|> L_{\epsilon}}e^{-2c|x-x_s|}\\
		&\lesssim_{C,\epsilon,q,d,c}L^{2q}_{\epsilon}e^{-cL_{\epsilon}}\\
		&\ll \frac{1}{1+(1+2L_{\epsilon})^d}
	\end{align*}
	provided  $L_{\epsilon}\gg1.$ This will ensure $\LC(\psi_s)\subset Q_{L_{\epsilon}}(x_s)$.
\end{itemize}
\end{rmk}
Note that  Lemma \ref{equi-distri of LC} reveals that the localization centers are nearly equidistributed on $\Z^d$.\smallskip

 \textbf{Step 2. Relabelling eigenfunctions}\\ 
In this step, for small $\epsilon$, denote  by $L_{\epsilon}$   the scale given in Lemma \eqref{equi-distri of LC}. The aim of this step is to prove

\begin{thm}[{\bf Relabelling Theorem}]\label{relabelling thm}
Suppose  that  Assumption \ref{SULEass} holds true and  $0\leq \epsilon\ll 1$. Then there is a relabelling map
	\[f:\{\psi_s\}_{s\in \Z^d}\leftrightarrow \{\phi_j\}_{j\in \Z^d}\]
such that
\[ {\forall}j\in\Z^d \ {\rm with}\  |j|\geq L''_{\epsilon}, \ \exists \iota_j\in \LC(\phi_j) \  {\rm s.t.,}\  |j-\yo_j|\leq \epsilon |j|,\]
where $L''_{\epsilon}\geq 500$ is another large scale depending only on $\epsilon,C,c,q,d$. That is to say, after relabelling, $\phi_j$ has some localization center close to  $j$. 
\end{thm}

\begin{proof}[Proof of Theorem \ref{relabelling thm}.]
	Assume $\epsilon\ll 1$ and take $L'_{\epsilon}=\frac{10}{\epsilon}L_{\epsilon}\geq 201$.  For    $\psi_s$,  choose one element $x_s$ such that 
	\[|x_s|=\min\{|x|:\ x\in\LC(\psi_s)\}.\]
	Then  for any $K>0,$
	\[|x_s|\leq K\Leftrightarrow \LC(\psi_s)\cap Q_K\neq \emptyset.\]
	We construct the relabelling map at different scales  step by step. For convenience, for an odd scale $L$, denote $S_L=Q_{\frac{L-1}{2}}\subset \Z^d$. In particular, $\# S_L=L^d$. Additionally, for  $r>0$, denote by  $\lfloor r\rfloor_o$ the largest odd integer that does not exceed $r$. Without loss of generality, we can assume that both $L_{\epsilon}$ and $L_{\epsilon}'$ are odd. 
	
	(1) First,  take $L_1=L'_{\epsilon}\geq 201$. By Lemma \ref{equi-distri of LC}, we have 
	\[(1-\epsilon)^d L^d_1\leq \#\{\psi_s:\ |x_s|\leq \frac{L_1-1}{2}\}:=N_1\leq (1+\epsilon)^d L_1^d.\]
	Take  $K_1=\lfloor N_1^{\frac{1}{d}}\rfloor_o$  to be the biggest odd scale such that $\# S_{K_1}= K_1^d \leq N_1$. Then 
	\begin{equation}\label{estimate on K_1}
			(1-\epsilon)L_1-2<\lfloor (1-\epsilon)L_1\rfloor_o\leq K_1\leq \lfloor (1+\epsilon)L_1\rfloor_o\leq (1+\epsilon)L_1 
	\end{equation}
	We can  construct (adding in  sites near the boundary of $S_{K_1}$)  a set  with 
	\[S_{K_1}\subset \tilde{S}_{K_1} \subset S_{K_1+2},\ \#\tilde{S}_{K_1}=N_1.\]
	Then there is a  relabelling for    $j\in \tilde{S}_{K_1}$ to match with $\{\psi_s:\ |x_s|\leq \frac{L_1-1}{2}\}$. That is,
	\[f:\{j:\ j\in\tilde{S}_{K_1}\}\leftrightarrow\{\psi_s:\ |x_s|\leq \frac{L_1-1}{2}\}.\]
	
	(2) Next, take $L_2=3L_1$. By Lemma \ref{equi-distri of LC}, we have 
	\[(1-\epsilon)^d L^d_2\leq \#\{\psi_s:\ |x_s|\leq \frac{L_2-1}{2}\}:=N_2\leq (1+\epsilon)^d L_2^d\]
	Take $K_2=\lfloor N_2^{\frac{1}{d}}\rfloor_o$ to be the biggest odd scale such that $\# S_{K_2}= K_2^d \leq N_2$. Then 
	\begin{equation}\label{estimate on K_2}
			(1-\epsilon)L_2-2<\lfloor (1-\epsilon)L_2\rfloor_o\leq K_2\leq \lfloor (1+\epsilon)L_2\rfloor_o\leq (1+\epsilon)L_2.  
	\end{equation}
	We can again  construct a set 
	\[S_{K_2}\subset \tilde{S}_{K_2} \subset S_{K_2+2},\ \#\tilde{S}_{K_2}=N_2\]
	From  \eqref{estimate on K_1} and  \eqref{estimate on K_2}, we have 
	\begin{align}\label{diff scale boundary dist}
		\dist(\partial_-\tilde{S}_{K_2},\partial_-\tilde{S}_{K_1} ) &\geq \frac{1}{2}((1-\epsilon)L_2-2-(1+\epsilon)L_1)\\
		 \notag   &=(1-2\epsilon)L_1-1\geq (1-\frac{2}{10})\times 201-1>10.
	\end{align}
	This indicates that the trim near the boundary of $S_{K_2}$ has no influence on the previous that of  $S_{K_1}$. For a better geometric understanding,  we refer to FIGURE 1.
	
	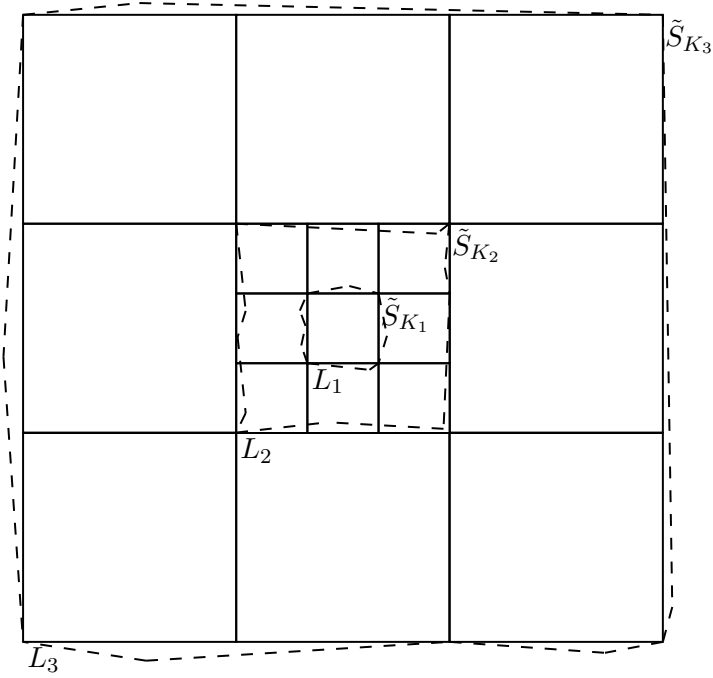
\begin{figure}[htbp]
		\centering
		
	\tikzset{every picture/.style={line width=0.75pt}} 

\begin{tikzpicture}[x=0.75pt,y=0.75pt,yscale=-0.8,xscale=0.8]

\draw   (252.44,201.51) -- (297.31,201.51) -- (297.31,245.52) -- (252.44,245.52) -- cycle ;
\draw   (252.44,245.52) -- (297.31,245.52) -- (297.31,289.54) -- (252.44,289.54) -- cycle ;
\draw   (252.44,289.54) -- (297.31,289.54) -- (297.31,333.55) -- (252.44,333.55) -- cycle ;
\draw   (297.31,201.51) -- (342.18,201.51) -- (342.18,245.52) -- (297.31,245.52) -- cycle ;
\draw   (297.31,245.52) -- (342.18,245.52) -- (342.18,289.54) -- (297.31,289.54) -- cycle ;
\draw   (297.31,289.54) -- (342.18,289.54) -- (342.18,333.55) -- (297.31,333.55) -- cycle ;
\draw   (342.18,201.51) -- (387.05,201.51) -- (387.05,245.52) -- (342.18,245.52) -- cycle ;
\draw   (342.18,245.52) -- (387.05,245.52) -- (387.05,289.54) -- (342.18,289.54) -- cycle ;
\draw   (342.18,289.54) -- (387.05,289.54) -- (387.05,333.55) -- (342.18,333.55) -- cycle ;
\draw   (387.05,201.51) -- (521.41,201.51) -- (521.41,333.3) -- (387.05,333.3) -- cycle ;
\draw   (387.05,333.55) -- (521.41,333.55) -- (521.41,465.34) -- (387.05,465.34) -- cycle ;
\draw   (252.44,333.55) -- (386.8,333.55) -- (386.8,465.34) -- (252.44,465.34) -- cycle ;
\draw   (118.08,333.55) -- (252.44,333.55) -- (252.44,465.34) -- (118.08,465.34) -- cycle ;
\draw   (118.08,201.51) -- (252.44,201.51) -- (252.44,333.3) -- (118.08,333.3) -- cycle ;
\draw   (387.05,69.72) -- (521.41,69.72) -- (521.41,201.51) -- (387.05,201.51) -- cycle ;
\draw   (252.44,69.72) -- (386.8,69.72) -- (386.8,201.51) -- (252.44,201.51) -- cycle ;
\draw   (118.08,69.72) -- (252.44,69.72) -- (252.44,201.51) -- (118.08,201.51) -- cycle ;
\draw  [dash pattern={on 4.5pt off 4.5pt}]  (292.57,256.34) -- (297.31,245.52) ;
\draw  [dash pattern={on 4.5pt off 4.5pt}]  (296.16,266.02) -- (292.57,256.34) ;
\draw  [dash pattern={on 4.5pt off 4.5pt}]  (293.46,278.35) -- (296.16,266.02) ;
\draw  [dash pattern={on 4.5pt off 4.5pt}]  (297.31,289.54) -- (293.46,278.35) ;
\draw  [dash pattern={on 4.5pt off 4.5pt}]  (297.31,245.52) -- (323.74,240.82) ;
\draw  [dash pattern={on 4.5pt off 4.5pt}]  (323.74,240.82) -- (342.18,245.52) ;
\draw  [dash pattern={on 4.5pt off 4.5pt}]  (342.18,245.52) -- (347.5,272.23) ;
\draw  [dash pattern={on 4.5pt off 4.5pt}]  (347.5,272.23) -- (342.18,289.54) ;
\draw  [dash pattern={on 4.5pt off 4.5pt}]  (297.31,289.54) -- (336.47,293.72) ;
\draw    (342.18,289.54) -- (336.47,293.72) ;
\draw  [dash pattern={on 4.5pt off 4.5pt}]  (252.44,201.51) -- (258.42,256.52) ;
\draw  [dash pattern={on 4.5pt off 4.5pt}]  (258.42,256.52) -- (253.33,273.05) ;
\draw  [dash pattern={on 4.5pt off 4.5pt}]  (253.33,273.05) -- (258.42,320.99) ;
\draw  [dash pattern={on 4.5pt off 4.5pt}]  (258.42,320.99) -- (252.44,333.3) ;
\draw  [dash pattern={on 4.5pt off 4.5pt}]  (252.44,201.51) -- (378.89,207.76) ;
\draw  [dash pattern={on 4.5pt off 4.5pt}]  (386.8,201.51) -- (378.89,207.76) ;
\draw  [dash pattern={on 4.5pt off 4.5pt}]  (386.8,201.51) -- (383.98,229.25) ;
\draw  [dash pattern={on 4.5pt off 4.5pt}]  (383.98,229.25) -- (387.05,245.52) ;
\draw  [dash pattern={on 4.5pt off 4.5pt}]  (387.05,245.52) -- (383.13,330.91) ;
\draw  [dash pattern={on 4.5pt off 4.5pt}]  (383.13,330.91) -- (310.17,326.78) ;
\draw  [dash pattern={on 4.5pt off 4.5pt}]  (252.44,333.3) -- (310.17,326.78) ;
\draw  [dash pattern={on 4.5pt off 4.5pt}]  (118.08,69.72) -- (105.71,285.45) ;
\draw  [dash pattern={on 4.5pt off 4.5pt}]  (105.71,285.45) -- (118.08,465.34) ;
\draw  [dash pattern={on 4.5pt off 4.5pt}]  (118.08,69.72) -- (191.4,62.29) ;
\draw  [dash pattern={on 4.5pt off 4.5pt}]  (191.4,62.29) -- (521.41,69.72) ;
\draw  [dash pattern={on 4.5pt off 4.5pt}]  (195.64,477.21) -- (118.08,465.34) ;
\draw  [dash pattern={on 4.5pt off 4.5pt}]  (195.64,477.21) -- (386.8,465.34) ;
\draw  [dash pattern={on 4.5pt off 4.5pt}]  (387.05,465.34) -- (484.93,472.25) ;
\draw  [dash pattern={on 4.5pt off 4.5pt}]  (521.41,465.34) -- (484.93,472.25) ;
\draw  [dash pattern={on 4.5pt off 4.5pt}]  (521.41,69.72) -- (527.35,443.32) ;
\draw  [dash pattern={on 4.5pt off 4.5pt}]  (527.35,443.32) -- (521.41,465.34) ;

\draw (298.18,291.22) node [anchor=north west][inner sep=0.75pt]   [align=left] {$L_1$};
\draw (253.31,335.23) node [anchor=north west][inner sep=0.75pt]   [align=left] {$L_2$};
\draw (118.95,467.02) node [anchor=north west][inner sep=0.75pt]   [align=left] {$L_3$};
\draw (341.76,246.62) node [anchor=north west][inner sep=0.75pt]   [align=left] {$\tilde{S}_{K_1}$};
\draw (386.63,202.6) node [anchor=north west][inner sep=0.75pt]   [align=left] {$\tilde{S}_{K_2}$};
\draw (520.99,70.82) node [anchor=north west][inner sep=0.75pt]   [align=left] {$\tilde{S}_{K_3}$};

\end{tikzpicture}

		\caption{The construction of relabelling map in each annulus.} 
	\end{figure}

Now we have $\tilde{S}_{K_1}\subset\tilde{S}_{K_2}$, and by the above steps, in scale $L_1$,  we have already relabelled  the eigenfunctions  in $\{\psi_s:\ |x_s|\leq \frac{L_1-1}{2}\}$ with  the set $\tilde{S}_{K_1}$. So,  we need to construct a relabelling map 
\[\{j:\ j\in \tilde{S}_{K_2}\backslash \tilde{S}_{K_1} \}\leftrightarrow \{\psi_s:\ \frac{L_1-1}{2}<|x_s|\leq \frac{L_2-1}{2}\}.\]
Notice that, by the construction of $\tilde{S}_{K_1}$ and $\tilde{S}_{K_2}$,  we  have 
\[\# (\tilde{S}_{K_2}\backslash \tilde{S}_{K_1})=\# \{\psi_s:\ \frac{L_1-1}{2}<|x_s|\leq \frac{L_2-1}{2}\}=N_2-N_1.\]
Moreover, the target relabelling map must have the property that for each $s$ such that $|x_s|>\frac{L_1-1}{2}=L'_{\epsilon}$,  we can relabel $\psi_s$ with index $j\in \tilde{S}_{K_2}\backslash \tilde{S}_{K_1}$ to become $\phi_j$ and   
{\[\exists \iota_j\in \LC(\psi_s),|\iota_j-j|\leq \frac{1}{|\log\epsilon|}|j|\Leftrightarrow \LC(\psi_s)\cap Q_{|\log\epsilon|^{-1}|j|}(j)\neq \emptyset.\]}
In fact, this requirement provides an "accessible edge" between the set of eigenfunctions  $\{\psi_s\}_{s\in\Z^d}$ and index set $\{j\}_{j\in\Z^d}$. This motivates the use of  {\bf Hall's marriage theorem} to prove the existence of the target relabelling map. This technique was first used in \cite{EK16}. However, to verify Hall's condition, \cite{EK16} established estimates on the  $\ell^2$-operator norm of the product of two orthogonal projections of finite dimensions, which provides a comparison of the dimensions of the corresponding subspaces. Unfortunately, this method becomes invalid because we aim to relabel eigenfunctions of an {\bf infinite-dimensional system}, rather than a finite one. In this paper, we verify the Hall's condition directly by the equidistribution argument (cf. Lemma \ref{equi-distri of LC})  and  certain geometric considerations. We believe that this approach is novel. The choice of $\frac{1}{|\log\epsilon|}$ is for some technical reasons.

Denote 
\begin{align*}
&\tilde{S}_{K_2}\backslash \tilde{S}_{K_1}=\{j:\ j\in \tilde{S}_{K_2}\backslash \tilde{S}_{K_1}\}:= \mcA,\\
&\left\{\psi_s:\ \frac{L_1-1}{2}<|x_s|\leq \frac{L_2-1}{2}\right\} :=  \mcB.
\end{align*}
Define the edges  set to be 
\[\mathbb{E}:=\{(j,\psi_s):\ \LC(\psi_s)\cap Q_{|\log\epsilon|^{-1}|j|}(j)\neq \emptyset\}.\]
Then $(\mcA,\mcB;\mathbb{E})$ generates  a bipartite graph, with $\#\mcA=\#\mcB$. We only need to verify the Hall's condition. For any  $j\in\mcA$ and $U\subset \mcA$, we define  
\begin{align*}
&\mcN(j):=\{\psi_s\in\mcB:\ (j,\psi_s)\in \mathbb{E}\},\\
&\mcN(U):=\bigcup_{j\in U}\mcN(j)=\{\psi_s\in \mcB:\  \LC(\psi_s)\cap\big(\cup_{j\in U} Q_{|\log\epsilon|^{-1}|j|}(j) \big)\neq \emptyset\}.
\end{align*}

\begin{lem}[{\bf Hall's condtion}]
	\[\#\mcN(U)\geq \# U\ {\rm for}\ \forall  U\subset \mathcal A.\]
\end{lem}
\begin{proof}[Proof of Lemma B.5. ]
By \eqref{estimate on K_1}, \eqref{estimate on K_2}, the choices  of $L_1,L_2$ and $0\leq \epsilon\ll 1$, we have 
\[\mcA=\tilde{S}_{K_2}\backslash \tilde{S}_{K_1}\subset S_{(1+\epsilon)L_2+2}\backslash S_{(1-\epsilon)L_1-2}\subset S_{(\frac{30}{\epsilon}+50)L_{\epsilon}}\backslash S_{(\frac{10}{\epsilon}-20)L_{\epsilon}}.\]
Pave the annulus $S_{(\frac{30}{\epsilon}+50)L_{\epsilon}}\backslash S_{(\frac{10}{\epsilon}-20)L_{\epsilon}}$ with$\#(S_{(\frac{30}{\epsilon}+50)}\backslash S_{(\frac{10}{\epsilon}-20)})$ many $L_{\epsilon}$-size cubes $C_k$ with the homothetic transformation (cf. FIGURE 2): 
\[S_{(\frac{30}{\epsilon}+50)L_{\epsilon}}\backslash S_{(\frac{10}{\epsilon}-20)L_{\epsilon}} =\bigcup_{k\in S_{(\frac{30}{\epsilon}+50)}\backslash S_{(\frac{10}{\epsilon}-20)}}C_k.\]
\begin{figure}[htbp]
	\centering

\tikzset{every picture/.style={line width=0.75pt}} 

\begin{tikzpicture}[x=0.75pt,y=0.75pt,yscale=-0.7,xscale=0.7]

\draw   (101.19,125.97) -- (146.15,125.97) -- (146.15,168.18) -- (101.19,168.18) -- cycle ;
\draw   (101.19,168.18) -- (146.15,168.18) -- (146.15,210.39) -- (101.19,210.39) -- cycle ;
\draw   (101.19,210.39) -- (146.15,210.39) -- (146.15,252.6) -- (101.19,252.6) -- cycle ;
\draw   (146.15,210.39) -- (191.12,210.39) -- (191.12,252.6) -- (146.15,252.6) -- cycle ;
\draw   (191.12,210.39) -- (236.09,210.39) -- (236.09,252.6) -- (191.12,252.6) -- cycle ;
\draw   (101.19,83.75) -- (146.15,83.75) -- (146.15,125.97) -- (101.19,125.97) -- cycle ;
\draw   (56.22,125.97) -- (101.19,125.97) -- (101.19,168.18) -- (56.22,168.18) -- cycle ;
\draw   (56.22,168.18) -- (101.19,168.18) -- (101.19,210.39) -- (56.22,210.39) -- cycle ;
\draw   (56.22,210.39) -- (101.19,210.39) -- (101.19,252.6) -- (56.22,252.6) -- cycle ;
\draw   (101.19,252.6) -- (146.15,252.6) -- (146.15,294.81) -- (101.19,294.81) -- cycle ;
\draw   (236.09,210.39) -- (281.06,210.39) -- (281.06,252.6) -- (236.09,252.6) -- cycle ;
\draw   (146.15,252.6) -- (191.12,252.6) -- (191.12,294.81) -- (146.15,294.81) -- cycle ;
\draw   (191.12,252.6) -- (236.09,252.6) -- (236.09,294.81) -- (191.12,294.81) -- cycle ;
\draw   (56.22,83.75) -- (101.19,83.75) -- (101.19,125.97) -- (56.22,125.97) -- cycle ;
\draw   (56.22,252.6) -- (101.19,252.6) -- (101.19,294.81) -- (56.22,294.81) -- cycle ;
\draw   (236.09,252.6) -- (281.06,252.6) -- (281.06,294.81) -- (236.09,294.81) -- cycle ;
\draw   (11.25,83.75) -- (56.22,83.75) -- (56.22,125.97) -- (11.25,125.97) -- cycle ;
\draw   (56.22,41.54) -- (101.19,41.54) -- (101.19,83.75) -- (56.22,83.75) -- cycle ;
\draw   (11.25,125.97) -- (56.22,125.97) -- (56.22,168.18) -- (11.25,168.18) -- cycle ;
\draw   (11.25,168.18) -- (56.22,168.18) -- (56.22,210.39) -- (11.25,210.39) -- cycle ;
\draw   (11.25,210.39) -- (56.22,210.39) -- (56.22,252.6) -- (11.25,252.6) -- cycle ;
\draw   (11.25,252.6) -- (56.22,252.6) -- (56.22,294.81) -- (11.25,294.81) -- cycle ;
\draw   (56.22,294.81) -- (101.19,294.81) -- (101.19,337.02) -- (56.22,337.02) -- cycle ;
\draw   (101.19,294.81) -- (146.15,294.81) -- (146.15,337.02) -- (101.19,337.02) -- cycle ;
\draw   (146.15,294.81) -- (191.12,294.81) -- (191.12,337.02) -- (146.15,337.02) -- cycle ;
\draw   (191.12,294.81) -- (236.09,294.81) -- (236.09,337.02) -- (191.12,337.02) -- cycle ;
\draw   (236.09,294.81) -- (281.06,294.81) -- (281.06,337.02) -- (236.09,337.02) -- cycle ;
\draw   (101.19,41.54) -- (146.15,41.54) -- (146.15,83.75) -- (101.19,83.75) -- cycle ;
\draw   (57.67,422.94) -- (102.64,422.94) -- (102.64,465.15) -- (57.67,465.15) -- cycle ;
\draw   (102.64,422.94) -- (147.61,422.94) -- (147.61,465.15) -- (102.64,465.15) -- cycle ;
\draw   (147.61,422.94) -- (192.58,422.94) -- (192.58,465.15) -- (147.61,465.15) -- cycle ;
\draw   (11.25,41.54) -- (56.22,41.54) -- (56.22,83.75) -- (11.25,83.75) -- cycle ;
\draw   (11.25,294.81) -- (56.22,294.81) -- (56.22,337.02) -- (11.25,337.02) -- cycle ;
\draw   (12.71,422.94) -- (57.67,422.94) -- (57.67,465.15) -- (12.71,465.15) -- cycle ;
\draw   (333.77,238.17) -- (372.85,238.17) -- (372.85,229.77) -- (398.9,246.57) -- (372.85,263.37) -- (372.85,254.97) -- (333.77,254.97) -- cycle ;
\draw   (542.59,63.84) .. controls (542.59,61.17) and (544.76,59) .. (547.44,59) .. controls (550.12,59) and (552.29,61.17) .. (552.29,63.84) .. controls (552.29,66.51) and (550.12,68.67) .. (547.44,68.67) .. controls (544.76,68.67) and (542.59,66.51) .. (542.59,63.84) -- cycle ;
\draw   (542.59,106.05) .. controls (542.59,103.38) and (544.76,101.21) .. (547.44,101.21) .. controls (550.12,101.21) and (552.29,103.38) .. (552.29,106.05) .. controls (552.29,108.72) and (550.12,110.88) .. (547.44,110.88) .. controls (544.76,110.88) and (542.59,108.72) .. (542.59,106.05) -- cycle ;
\draw   (542.59,148.26) .. controls (542.59,145.59) and (544.76,143.42) .. (547.44,143.42) .. controls (550.12,143.42) and (552.29,145.59) .. (552.29,148.26) .. controls (552.29,150.93) and (550.12,153.1) .. (547.44,153.1) .. controls (544.76,153.1) and (542.59,150.93) .. (542.59,148.26) -- cycle ;
\draw   (543.12,190.47) .. controls (543.12,188.1) and (545.06,186.18) .. (547.44,186.18) .. controls (549.82,186.18) and (551.76,188.1) .. (551.76,190.47) .. controls (551.76,192.84) and (549.82,194.77) .. (547.44,194.77) .. controls (545.06,194.77) and (543.12,192.84) .. (543.12,190.47) -- cycle ;
\draw   (542.59,232.68) .. controls (542.59,230.01) and (544.76,227.85) .. (547.44,227.85) .. controls (550.12,227.85) and (552.29,230.01) .. (552.29,232.68) .. controls (552.29,235.35) and (550.12,237.52) .. (547.44,237.52) .. controls (544.76,237.52) and (542.59,235.35) .. (542.59,232.68) -- cycle ;
\draw   (587.56,232.68) .. controls (587.56,230.01) and (589.73,227.85) .. (592.41,227.85) .. controls (595.08,227.85) and (597.25,230.01) .. (597.25,232.68) .. controls (597.25,235.35) and (595.08,237.52) .. (592.41,237.52) .. controls (589.73,237.52) and (587.56,235.35) .. (587.56,232.68) -- cycle ;
\draw   (632.53,232.68) .. controls (632.53,230.01) and (634.7,227.85) .. (637.38,227.85) .. controls (640.05,227.85) and (642.22,230.01) .. (642.22,232.68) .. controls (642.22,235.35) and (640.05,237.52) .. (637.38,237.52) .. controls (634.7,237.52) and (632.53,235.35) .. (632.53,232.68) -- cycle ;
\draw   (677.5,232.68) .. controls (677.5,230.01) and (679.67,227.85) .. (682.34,227.85) .. controls (685.02,227.85) and (687.19,230.01) .. (687.19,232.68) .. controls (687.19,235.35) and (685.02,237.52) .. (682.34,237.52) .. controls (679.67,237.52) and (677.5,235.35) .. (677.5,232.68) -- cycle ;
\draw   (677.5,274.89) .. controls (677.5,272.22) and (679.67,270.06) .. (682.34,270.06) .. controls (685.02,270.06) and (687.19,272.22) .. (687.19,274.89) .. controls (687.19,277.57) and (685.02,279.73) .. (682.34,279.73) .. controls (679.67,279.73) and (677.5,277.57) .. (677.5,274.89) -- cycle ;
\draw   (632.53,274.89) .. controls (632.53,272.22) and (634.7,270.06) .. (637.38,270.06) .. controls (640.05,270.06) and (642.22,272.22) .. (642.22,274.89) .. controls (642.22,277.57) and (640.05,279.73) .. (637.38,279.73) .. controls (634.7,279.73) and (632.53,277.57) .. (632.53,274.89) -- cycle ;
\draw   (587.56,274.89) .. controls (587.56,272.22) and (589.73,270.06) .. (592.41,270.06) .. controls (595.08,270.06) and (597.25,272.22) .. (597.25,274.89) .. controls (597.25,277.57) and (595.08,279.73) .. (592.41,279.73) .. controls (589.73,279.73) and (587.56,277.57) .. (587.56,274.89) -- cycle ;
\draw   (542.5,273.89) .. controls (542.5,271.22) and (544.67,269.06) .. (547.34,269.06) .. controls (550.02,269.06) and (552.19,271.22) .. (552.19,273.89) .. controls (552.19,276.57) and (550.02,278.73) .. (547.34,278.73) .. controls (544.67,278.73) and (542.5,276.57) .. (542.5,273.89) -- cycle ;
\draw   (497.53,273.89) .. controls (497.53,271.22) and (499.7,269.06) .. (502.38,269.06) .. controls (505.05,269.06) and (507.22,271.22) .. (507.22,273.89) .. controls (507.22,276.57) and (505.05,278.73) .. (502.38,278.73) .. controls (499.7,278.73) and (497.53,276.57) .. (497.53,273.89) -- cycle ;
\draw   (452.56,273.89) .. controls (452.56,271.22) and (454.73,269.06) .. (457.41,269.06) .. controls (460.08,269.06) and (462.25,271.22) .. (462.25,273.89) .. controls (462.25,276.57) and (460.08,278.73) .. (457.41,278.73) .. controls (454.73,278.73) and (452.56,276.57) .. (452.56,273.89) -- cycle ;
\draw   (543.5,318.89) .. controls (543.5,316.22) and (545.67,314.06) .. (548.34,314.06) .. controls (551.02,314.06) and (553.19,316.22) .. (553.19,318.89) .. controls (553.19,321.57) and (551.02,323.73) .. (548.34,323.73) .. controls (545.67,323.73) and (543.5,321.57) .. (543.5,318.89) -- cycle ;
\draw   (498.53,318.89) .. controls (498.53,316.22) and (500.7,314.06) .. (503.38,314.06) .. controls (506.05,314.06) and (508.22,316.22) .. (508.22,318.89) .. controls (508.22,321.57) and (506.05,323.73) .. (503.38,323.73) .. controls (500.7,323.73) and (498.53,321.57) .. (498.53,318.89) -- cycle ;
\draw   (453.56,318.89) .. controls (453.56,316.22) and (455.73,314.06) .. (458.41,314.06) .. controls (461.08,314.06) and (463.25,316.22) .. (463.25,318.89) .. controls (463.25,321.57) and (461.08,323.73) .. (458.41,323.73) .. controls (455.73,323.73) and (453.56,321.57) .. (453.56,318.89) -- cycle ;
\draw   (677.5,318.89) .. controls (677.5,316.22) and (679.67,314.06) .. (682.34,314.06) .. controls (685.02,314.06) and (687.19,316.22) .. (687.19,318.89) .. controls (687.19,321.57) and (685.02,323.73) .. (682.34,323.73) .. controls (679.67,323.73) and (677.5,321.57) .. (677.5,318.89) -- cycle ;
\draw   (632.53,318.89) .. controls (632.53,316.22) and (634.7,314.06) .. (637.38,314.06) .. controls (640.05,314.06) and (642.22,316.22) .. (642.22,318.89) .. controls (642.22,321.57) and (640.05,323.73) .. (637.38,323.73) .. controls (634.7,323.73) and (632.53,321.57) .. (632.53,318.89) -- cycle ;
\draw   (587.56,318.89) .. controls (587.56,316.22) and (589.73,314.06) .. (592.41,314.06) .. controls (595.08,314.06) and (597.25,316.22) .. (597.25,318.89) .. controls (597.25,321.57) and (595.08,323.73) .. (592.41,323.73) .. controls (589.73,323.73) and (587.56,321.57) .. (587.56,318.89) -- cycle ;
\draw   (544.5,444.89) .. controls (544.5,442.22) and (546.67,440.06) .. (549.34,440.06) .. controls (552.02,440.06) and (554.19,442.22) .. (554.19,444.89) .. controls (554.19,447.57) and (552.02,449.73) .. (549.34,449.73) .. controls (546.67,449.73) and (544.5,447.57) .. (544.5,444.89) -- cycle ;
\draw   (499.53,444.89) .. controls (499.53,442.22) and (501.7,440.06) .. (504.38,440.06) .. controls (507.05,440.06) and (509.22,442.22) .. (509.22,444.89) .. controls (509.22,447.57) and (507.05,449.73) .. (504.38,449.73) .. controls (501.7,449.73) and (499.53,447.57) .. (499.53,444.89) -- cycle ;
\draw   (454.56,444.89) .. controls (454.56,442.22) and (456.73,440.06) .. (459.41,440.06) .. controls (462.08,440.06) and (464.25,442.22) .. (464.25,444.89) .. controls (464.25,447.57) and (462.08,449.73) .. (459.41,449.73) .. controls (456.73,449.73) and (454.56,447.57) .. (454.56,444.89) -- cycle ;
\draw   (591.53,444.89) .. controls (591.53,442.22) and (593.7,440.06) .. (596.38,440.06) .. controls (599.05,440.06) and (601.22,442.22) .. (601.22,444.89) .. controls (601.22,447.57) and (599.05,449.73) .. (596.38,449.73) .. controls (593.7,449.73) and (591.53,447.57) .. (591.53,444.89) -- cycle ;
\draw   (499.59,62.84) .. controls (499.59,60.17) and (501.76,58) .. (504.44,58) .. controls (507.12,58) and (509.29,60.17) .. (509.29,62.84) .. controls (509.29,65.51) and (507.12,67.67) .. (504.44,67.67) .. controls (501.76,67.67) and (499.59,65.51) .. (499.59,62.84) -- cycle ;
\draw   (499.59,105.05) .. controls (499.59,102.38) and (501.76,100.21) .. (504.44,100.21) .. controls (507.12,100.21) and (509.29,102.38) .. (509.29,105.05) .. controls (509.29,107.72) and (507.12,109.88) .. (504.44,109.88) .. controls (501.76,109.88) and (499.59,107.72) .. (499.59,105.05) -- cycle ;
\draw   (499.59,147.26) .. controls (499.59,144.59) and (501.76,142.42) .. (504.44,142.42) .. controls (507.12,142.42) and (509.29,144.59) .. (509.29,147.26) .. controls (509.29,149.93) and (507.12,152.1) .. (504.44,152.1) .. controls (501.76,152.1) and (499.59,149.93) .. (499.59,147.26) -- cycle ;
\draw   (500.12,189.47) .. controls (500.12,187.1) and (502.06,185.18) .. (504.44,185.18) .. controls (506.82,185.18) and (508.76,187.1) .. (508.76,189.47) .. controls (508.76,191.84) and (506.82,193.77) .. (504.44,193.77) .. controls (502.06,193.77) and (500.12,191.84) .. (500.12,189.47) -- cycle ;
\draw   (499.59,231.68) .. controls (499.59,229.01) and (501.76,226.85) .. (504.44,226.85) .. controls (507.12,226.85) and (509.29,229.01) .. (509.29,231.68) .. controls (509.29,234.35) and (507.12,236.52) .. (504.44,236.52) .. controls (501.76,236.52) and (499.59,234.35) .. (499.59,231.68) -- cycle ;
\draw   (452.59,61.84) .. controls (452.59,59.17) and (454.76,57) .. (457.44,57) .. controls (460.12,57) and (462.29,59.17) .. (462.29,61.84) .. controls (462.29,64.51) and (460.12,66.67) .. (457.44,66.67) .. controls (454.76,66.67) and (452.59,64.51) .. (452.59,61.84) -- cycle ;
\draw   (452.59,104.05) .. controls (452.59,101.38) and (454.76,99.21) .. (457.44,99.21) .. controls (460.12,99.21) and (462.29,101.38) .. (462.29,104.05) .. controls (462.29,106.72) and (460.12,108.88) .. (457.44,108.88) .. controls (454.76,108.88) and (452.59,106.72) .. (452.59,104.05) -- cycle ;
\draw   (452.59,146.26) .. controls (452.59,143.59) and (454.76,141.42) .. (457.44,141.42) .. controls (460.12,141.42) and (462.29,143.59) .. (462.29,146.26) .. controls (462.29,148.93) and (460.12,151.1) .. (457.44,151.1) .. controls (454.76,151.1) and (452.59,148.93) .. (452.59,146.26) -- cycle ;
\draw   (453.12,188.47) .. controls (453.12,186.1) and (455.06,184.18) .. (457.44,184.18) .. controls (459.82,184.18) and (461.76,186.1) .. (461.76,188.47) .. controls (461.76,190.84) and (459.82,192.77) .. (457.44,192.77) .. controls (455.06,192.77) and (453.12,190.84) .. (453.12,188.47) -- cycle ;
\draw   (452.59,230.68) .. controls (452.59,228.01) and (454.76,225.85) .. (457.44,225.85) .. controls (460.12,225.85) and (462.29,228.01) .. (462.29,230.68) .. controls (462.29,233.35) and (460.12,235.52) .. (457.44,235.52) .. controls (454.76,235.52) and (452.59,233.35) .. (452.59,230.68) -- cycle ;

\draw (104.52,46.61) node [anchor=north west][inner sep=0.75pt]   [align=left] {1};
\draw (104.52,131.03) node [anchor=north west][inner sep=0.75pt]   [align=left] {3};
\draw (104.52,173.24) node [anchor=north west][inner sep=0.75pt]   [align=left] {4};
\draw (104.52,215.46) node [anchor=north west][inner sep=0.75pt]   [align=left] {5};
\draw (149.49,215.46) node [anchor=north west][inner sep=0.75pt]   [align=left] {6};
\draw (194.45,215.46) node [anchor=north west][inner sep=0.75pt]   [align=left] {7};
\draw (239.42,215.46) node [anchor=north west][inner sep=0.75pt]   [align=left] {8};
\draw (261.93,180.97) node [anchor=north west][inner sep=0.75pt]   [align=left] {$(\frac{10}{\epsilon}-20)L_{\epsilon}$};
\draw (104.52,88.82) node [anchor=north west][inner sep=0.75pt]   [align=left] {2};
\draw (82.17,372.64) node [anchor=north west][inner sep=0.75pt]   [align=left] {$\vdots$};
\draw (-4.58,489.05) node [anchor=north west][inner sep=0.75pt]   [align=left] {$(\frac{30}{\epsilon}+50)L_{\epsilon}$};
\draw (-24.53,306.11) node [anchor=north west][inner sep=0.75pt]   [align=left] {$L_{\epsilon}$};
\draw (528.29,47.8) node [anchor=north west][inner sep=0.75pt]   [align=left] {1};
\draw (528.29,132.22) node [anchor=north west][inner sep=0.75pt]   [align=left] {3};
\draw (528.29,174.43) node [anchor=north west][inner sep=0.75pt]   [align=left] {4};
\draw (528.29,216.64) node [anchor=north west][inner sep=0.75pt]   [align=left] {5};
\draw (573.26,216.64) node [anchor=north west][inner sep=0.75pt]   [align=left] {6};
\draw (618.22,216.64) node [anchor=north west][inner sep=0.75pt]   [align=left] {7};
\draw (663.19,216.64) node [anchor=north west][inner sep=0.75pt]   [align=left] {8};
\draw (528.29,90.01) node [anchor=north west][inner sep=0.75pt]   [align=left] {2};
\draw (666.93,177.97) node [anchor=north west][inner sep=0.75pt]   [align=left] {$\frac{10}{\epsilon}-20$};
\draw (538.93,368.97) node [anchor=north west][inner sep=0.75pt]   [align=left] {$\vdots$};
\draw (437,478.05) node [anchor=north west][inner sep=0.75pt]   [align=left] {$\frac{30}{\epsilon}+50$};

\end{tikzpicture}

\caption{Homothetic transform from cubes to vertices.}

\end{figure}
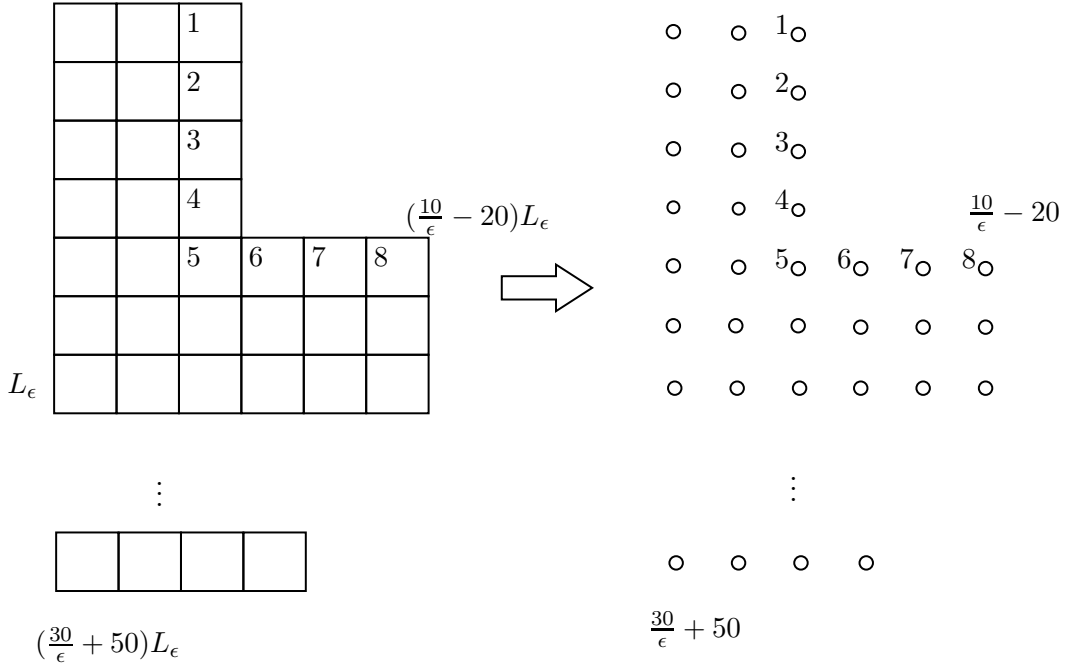
Now denote 
\[\mcF_1=\{k\in S_{(\frac{30}{\epsilon}+50)}\backslash S_{(\frac{10}{\epsilon}-20)}:\ C_k\cap U \neq\emptyset\}.\]
Then  
\[\# U\leq L_{\epsilon}^d \cdot (\#\mcF_1) \]
Moreover,  $j\in \mcA\Rightarrow|\log \epsilon|^{-1}|j|\geq \frac{4}{\epsilon|\log \epsilon|}L_{\epsilon}$. So,  for each $j\in C_k$, 
\[Q_{|\log \epsilon|^{-1}|j|}(j)\supset \bigcup_{k'\in Q_{\frac{2}{\epsilon|\log\epsilon|}}(k)}C_{k'}.\]
If  we denote 
\[\mcF_2=\bigcup_{k\in \mcF_1}\big( Q_{\frac{2}{\epsilon|\log\epsilon|}}(k) \cap (S_{\frac{30}{\epsilon}}\backslash S_{\frac{10}{\epsilon }})\big),\]
then  
\[\bigcup_{j\in U}Q_{|\log \epsilon|^{-1}|j|}(j) \cap (S_{L_2}\backslash S_{L_1})\supset \bigcup_{k\in \mcF_2}C_k.\]
Here the involvement of $S_{\frac{30}{\epsilon}}\backslash S_{\frac{10}{\epsilon }}$ is because that it's the homothetic transform of $S_{L_2}\backslash S_{L_1}$.  By Lemma \ref{equi-distri of LC}  and Remark B.1.(1),  for each $C_k$, the number of eigenfunctions  which have a localization center in it is bigger than $(1-\epsilon)^d L_{\epsilon}^d$. Hence,  if one chooses  $\La= \cup_{k\in \mcF_2}C_k$ and take $\tilde{\La}$ to be the extension  of $\La$ by  $\epsilon L_{\epsilon}$, $\hat{\La}$ to be the shrink of $\La$ by  $\epsilon L_{\epsilon}$,  respectively, then Remark B.1.(1) and the same argument as in the proof of Lemma \ref{equi-distri of LC}  will ensure  that 
\begin{equation}\label{La equi-distri}
	\#\{\psi_s:\ \LC(\psi_s)\cap\La\neq \emptyset\}\geq (1-\epsilon)^d\cdot\# \La. 
\end{equation}
In fact, the only difference between the proof of \eqref{La equi-distri} and  that of  Lemma \ref{equi-distri of LC} is the estimate of cardinality of $\tilde{\La}$ and $\hat{\La}$.  Since  $\La$ is the union of a  large number  of $L_{\epsilon}$-size blocks  and the extension  (or shrink) size is $\epsilon L_{\epsilon}$,  the cardinality of $\tilde{\La}$ (or $\hat{\La}$) is smaller (or larger) than  that of the set  given by extensions  (or shrinks) of  all cubes $C_k\in \La$ by $\epsilon L_{\epsilon}$. Thus,
\[\#\tilde{\La}\leq \#\mcF_2 \cdot (1+2\epsilon)^d L^d_{\epsilon}=(1+2\epsilon)^d \#\La, \]
\[\#\hat{\La}\geq \#\mcF_2 \cdot (1-2\epsilon)^d L^d_{\epsilon}=(1-2\epsilon)^d \#\La.  \]
Moreover, by our construction of $\mcF_2$ and $x_s$, we have $\La\subset S_{L_2}\backslash S_{L_1}$ and hence,
\[\psi_s:\ \LC(\psi_s)\cap\La\neq \emptyset\Rightarrow |x_s|\leq \frac{L_2-1}{2}.\]
However, we mention that  there is the  possibility that $|x_s|\leq \frac{L_1-1}{2}$.
Hence by \eqref{La equi-distri} (noting  that $\La= \cup_{k\in \mcF_2}C_k$), 
\begin{align*}
	\#\mcN(U)  &=\#\{\psi_s\in \mcB:\ \LC(\psi_s)\cap\big(\cup_{j\in U} Q_{|\log\epsilon|^{-1}|j|}(j) \big)\neq \emptyset\}\\
	  &\geq \#\{\psi_s\in \mcB:\  \LC(\psi_s)\cap \La\neq \emptyset \}\\
	  &=  \#\{\psi_s:\  \LC(\psi_s)\cap \La\neq \emptyset \}-\#\{\psi_s:\ |x_s|\leq \frac{L_1-1}{2} ,\LC(\psi_s)\cap \La\neq \emptyset \}\\
	  &\geq (1-\epsilon)^d\cdot\#(\bigcup_{k\in\mcF_2}C_k)- \Delta\\
	  & = (1-\epsilon)^d L^d_{\epsilon} \cdot(\#\mcF_2)-\Delta,
\end{align*} 
where 
\[\Delta= \#\{\psi_s:\ |x_s|\leq \frac{L_1-1}{2} ,\LC(\psi_s)\cap \La\neq \emptyset \}=\#\{\psi_s:\ |x_s|\leq \frac{5}{\epsilon}L_{\epsilon} -\frac12,\LC(\psi_s)\cap \La\neq \emptyset \}.\] 
It suffices to  estimate $\Delta$. Indeed,  by the argument in Remark B.1.(3),  we get 
\[ \{\psi_s:\ |x_s|\leq \frac{5}{\epsilon}L_{\epsilon}-\frac12 ,\LC(\psi_s)\cap \La\neq \emptyset \}\subset \{\psi_s:\ (\frac{5}{\epsilon}-2)L_{\epsilon}-\frac12< |x_s|\leq  \frac{5}{\epsilon}L_{\epsilon}-\frac12\}. \] 
Hence, by Lemma  \ref{equi-distri of LC}, we have 
\begin{align}\label{error Delta}
	\Delta &\leq \# \{\psi_s:\ (\frac{5}{\epsilon}-2)L_{\epsilon}-\frac12< |x_s|\leq  \frac{5}{\epsilon}L_{\epsilon}-\frac12\}\\
	\notag &\leq (\frac{5}{\epsilon}L_{\epsilon})^d(1+\epsilon)^d-(\frac{5}{\epsilon}L_{\epsilon}-3L_{\epsilon})^d(1-\epsilon)^d \\
	\notag &\lesssim_d L^d_{\epsilon}\epsilon^{-d+1}.
\end{align}
Now we discuss 
\begin{itemize}
	\item If $\mcF_2=S_{\frac{30}{\epsilon}}\backslash S_{\frac{10}{\epsilon}} $, i.e.,  the union $\cup_{k\in \mcF_1} Q_{\frac{2}{\epsilon|\log\epsilon|}}(k)$ covers the annulus, then   $\cup_{j\in U}Q_{|\log \epsilon|^{-1}|j|}(j)$ covers $S_{L_2}\backslash S_{L_1}$ and thus $\mcN(U)=\mcB$. So,
	     \[\#U\leq \#\mcA=\#\mcB=\#\mcN(U). \]
    \item Otherwise,  $\mcF_2$ does not cover the annulus $S_{\frac{30}{\epsilon}}\backslash S_{\frac{10}{\epsilon}}$. In this case, we just need to show
	\[(1-\epsilon)^d L^d_{\epsilon}\cdot \#\mcF_2-\Delta \geq L_{\epsilon}^d\cdot \#\mcF_1.\]
	Thus,  it suffices   to show 
	\begin{equation}\label{surfficient Hall}
			\#\mcF_2-\#\mcF_1\geq 2d\epsilon \cdot \#\mcF_1+\frac{2\Delta}{L_{\epsilon}^d} >(\frac{1}{(1-\epsilon)^d}-1)\cdot \#\mcF_1+\frac{\Delta}{(1-\epsilon)^dL_{\epsilon}^d}
	\end{equation}
	since  $0<\epsilon\ll1$.  Now choose an $x\in (S_{\frac{30}{\epsilon}}\backslash S_{\frac{10}{\epsilon}})\backslash \mcF_2$, and denote the closest point near $x$ in $\mcF_1$ is $k_0$. Note  
	\[D:=\dist(x,k_0)> \frac{2}{\epsilon|\log\epsilon|}.\]
	Hence, if we take 
	\[\Sigma=Q_D(x)\cap Q_{\frac{2}{\epsilon|\log\epsilon|}}(k_0)\cap (S_{\frac{30}{\epsilon}}\backslash S_{\frac{10}{\epsilon}}),\]
	then 
	\[\#\Sigma\geq (\frac{2}{\epsilon|\log\epsilon|}-50)^d\geq (\frac{1}{\epsilon|\log\epsilon|})^d.\]
	Notice that by the construction of $x,\Sigma$, we have 
	\[\mcF_1\subset S_{(\frac{30}{\epsilon}+50)}\backslash S_{(\frac{10}{\epsilon}-20)}, \ 
	\mcF_2\backslash ( \mcF_1\cap S_{\frac{30}{\epsilon}}\backslash S_{\frac{10}{\epsilon}})\supset \Sigma,\]
	which  gives 
	\[\#\mcF_1\lesssim_{d}\epsilon^{-d}.\] 
	This along with \eqref{error Delta}  implies  
	\[2d\epsilon \cdot \#\mcF_1+\frac{2\Delta}{L_{\epsilon}^d}\lesssim_d \epsilon^{-d+1}\]
	and (see FIGURE 3)
	\begin{align*}
			\#\mcF_2 -\#\mcF_1&\geq \#\Sigma-\#(S_{\frac{30}{\epsilon}+50}\backslash S_{\frac{30}{\epsilon}})-\#(S_{\frac{10}{\epsilon}}\backslash S_{\frac{10}{\epsilon}-20})\\
			 &\gtrsim_d(\epsilon|\log\epsilon|)^{-d}-C(d)\epsilon^{-(d-1)}\\
			 &\gtrsim_d \epsilon^{-d+1}\\
			 &\geq 2d\epsilon \cdot \#\mcF_1+\frac{2\Delta}{L_{\epsilon}^d} 
	\end{align*}
provided  $0<\epsilon<c(d)\ll 1$ and $\frac{1}{|\log\epsilon|}\gg \epsilon^{\frac{1}{d}}$. Hence,  the condition  \eqref{surfficient Hall}  is verified and  we prove the Hall's condition   
	\[\#\mcN(U)\geq \# U.\]
	\begin{figure}[htbp]
		\centering

		\tikzset{every picture/.style={line width=0.75pt}} 

		\begin{tikzpicture}[x=0.75pt,y=0.75pt,yscale=-0.7,xscale=0.7]
		
		\draw   (210.74,159.9) -- (498.81,159.9) -- (498.81,443.57) -- (210.74,443.57) -- cycle ;
		\draw   (306.81,254.5) -- (402.74,254.5) -- (402.74,348.97) -- (306.81,348.97) -- cycle ;
		\draw  [dash pattern={on 0.84pt off 2.51pt}] (197.83,147.19) -- (511.71,147.19) -- (511.71,456.29) -- (197.83,456.29) -- cycle ;
		\draw    (196.91,399.49) -- (210.54,399.49) ;
		\draw   (172.74,126) -- (262.17,126) -- (262.17,214.07) -- (172.74,214.07) -- cycle ;
		\draw  [line width=3] [line join = round][line cap = round] (217.71,169.28) .. controls (217.23,168.81) and (216.75,168.34) .. (216.27,167.87) ;
		\draw  [line width=3] [line join = round][line cap = round] (217.71,166.45) .. controls (215.87,166.45) and (216.27,167.74) .. (216.27,169.28) ;
		\draw  [line width=3] [line join = round][line cap = round] (217.71,167.16) .. controls (219.13,168.56) and (216.27,168.63) .. (216.27,166.45) ;
		\draw   (167,188.14) -- (233.48,188.14) -- (233.48,253.61) -- (167,253.61) -- cycle ;
		\draw  [line width=3] [line join = round][line cap = round] (199.78,219.42) .. controls (199.78,219.08) and (200.2,218.56) .. (200.5,218.71) .. controls (200.98,218.95) and (199.6,219.42) .. (199.06,219.42) ;
		\draw  [line width=3] [line join = round][line cap = round] (200.5,218.71) .. controls (200.5,219.07) and (201.35,221.41) .. (201.22,221.53) .. controls (200.53,222.21) and (201.12,219.42) .. (200.5,219.42) ;
		\draw  [line width=3] [line join = round][line cap = round] (200.5,221.53) .. controls (198.36,221.53) and (197.73,218.71) .. (199.78,218.71) ;
		\draw  [dash pattern={on 0.84pt off 2.51pt}] (318.79,266.3) -- (390.76,266.3) -- (390.76,337.17) -- (318.79,337.17) -- cycle ;
		\draw    (356.11,337.35) -- (356.11,347.94) ;
		\draw  [fill={rgb, 255:red, 225; green, 25; blue, 25 }  ,fill opacity=1 ] (211.46,189.05) -- (232.77,189.05) -- (232.77,213.77) -- (211.46,213.77) -- cycle ;
		\draw  [dash pattern={on 4.5pt off 4.5pt}]  (227.75,204.59) -- (329.57,208.82) ;
		\draw   (319.17,202.83) -- (330.28,208.83) -- (318.49,213.4) ;
		
		\draw (195.29,383.52) node [anchor=north west][inner sep=0.75pt]   [align=left] {$25$};
		\draw (469.97,332.68) node [anchor=north west][inner sep=0.75pt]   [align=left] {$\frac{30}{\epsilon}$};
		\draw (354.99,337.11) node [anchor=north west][inner sep=0.75pt]   [align=left] {$10$};
		\draw (217.47,169.8) node [anchor=north west][inner sep=0.75pt]   [align=left] {$x$};
		\draw (187.84,223.93) node [anchor=north west][inner sep=0.75pt]   [align=left] {$k_0$};
		\draw (328.28,202.03) node [anchor=north west][inner sep=0.75pt]   [align=left] {$\Sigma$};

		\end{tikzpicture}

		\caption{Control of  $\#\mcF_2-\#\mcF_1$.}

	\end{figure}
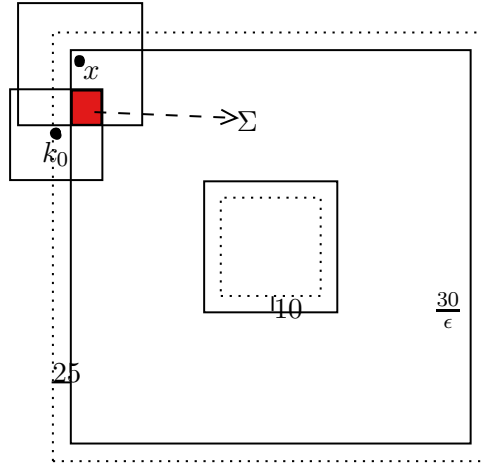

\end{itemize}
\end{proof}

Finally,  applying the  {\bf Hall's marriage theorem} to bipartite graph $(\mcA,\mcB;\mathbb{E})$ leads to the relabelling:  there is a bijection 
\[f:\ \{j:\ j\in\tilde{S}_{K_2}\backslash\tilde{S}_{K_1}\}\leftrightarrow\{\psi_s:\ \frac{L_1-1}{2}<|x_s|\leq \frac{L_2-1}{2}\}\]
such that,  for all $j\in \tilde{S}_{K_2}\backslash \tilde{S}_{K_1}$ and $\psi_s$ matching  with $j$, one has  
\[\exists \iota_j\in\LC(\psi_s)\ {\rm s.t.,}\ |\iota_j-j|\leq |\log\epsilon|^{-1}|j|.\]

(3)  We can prove the relabelling theorem inductively, i.e., for $k\geq 2,L_{k+1}=3L_{k}$, similar to  that  in (2).  Thus, repeating  the procedures in (2)   gives  the relabelling map to eigenfunctions  having some  localization center in $S_{L_{k+1}}\backslash S_{L_k}$. The only difference is that one should use $3^{k-1}L_{\epsilon}$-size cubes to pave the corresponding annulus.

Now we have constructed  a relabelling map such that for all $j\notin \tilde{S}_{K_1}$, we have 
\[\exists \iota_j\in\LC(\psi_s)\ {\rm s.t., }\ |\iota_j-j|\leq |\log\epsilon|^{-1}|j|.\]
Moreover,  by \eqref{estimate on K_1},  we know  that $|j|\geq 5L'_{\epsilon}\geq 2L_1\Rightarrow j \notin \tilde{S}_{K_1}$.
Finally,  replace   $\epsilon$ with  $e^{-\frac{1}{\epsilon}}$, so that $|\log\epsilon|^{-1}$ is replaced by $\epsilon$. Thus, by taking $0<\epsilon<c(d)\ll 1$  and letting $L''_{\epsilon}=5L'_{e^{-\frac{1}{\epsilon}}}$,  we  complete  the proof of Theorem \ref{relabelling thm}.
\end{proof}

\section{Resolvent identities}\label{JLSapp}
In this section, we introduce two types of resolvent identities, which are repeatedly used to derive exponential off-diagonal decay estimates for Green's functions. 

Let $A$ be a matrix on $\ell^2(\Z^d)$ satisfying
	\[|A(n,n')|\leq C_1 (1+|n-n'|)^{C_1}e^{-c_1|n-n'|},\ C_1,c_1>0.\]
We say that an elementary region $\La\in \mcE_{N'}$ is in the class $G$ (good) if 
\begin{equation}\label{range1}
	|(R_{\La}AR_{\La})^{-1}(n,n')|\leq e^{-c_2|n-n'|}\  {\rm for} \ |n-n'|\geq \sqrt{N'},
\end{equation}
where $\frac{1}{2}c_1< c_2\leq c_1$.
\begin{rmk}
	In linear spectral problems (cf. \cite{Bou07, JLS20}),  one  typically  deals  with operators such as
	\[|A(n,n')|\leq C_1 e^{-c_1|n-n'|},\ C_1,c_1\geq 0,\]
	without the polynomial factor $(|n-n'|+1)^{C_1}$. However, the extra polynomial factor does not significantly affect the proofs in \cite{JLS20} (cf. Lemma 3.2 and Theorem 3.3) and \cite{Liu22} (cf. Theorem 2.1), as one can replace, at each step, the estimate
	\[\sum_{n_1\in W \atop n_2\in \La\setminus W}e^{-c_1|n-n'|} |G_{\La}(n_2,n')|\leq (2N+1)^{2d}\sup_{n_2\in \La\setminus W}|G_{\La}(n_2,n')|\]
	with 
	\[\sum_{n_1\in W \atop n_2\in \La\setminus W}(|n-n'|+1)^{C_1}e^{-c_1|n-n'|} |G_{\La}(n_2,n')|\leq (2N+1)^{2d+C_1}\sup_{n_2\in \La\setminus W}|G_{\La}(n_2,n')|.\]
	\end{rmk}

\begin{thm}[cf. Theorem 2.1, \cite{Liu22}]\label{MSA1}
	 Let $ 0<\xi\leq \zeta<\frac{1}{2}$. Let $\tilde{\La}_0\in\mcE_N$ be an elementary region with the property that  for all $\La\subset\tilde{\La}_0,\La\in \mcR_{L}^{N^{\xi}}$ with $N^{\xi}\leq L\leq N$, the 
	Green's function $(R_{\La}A R_{\La})^{-1}$ satisfies 
	\[\nm (R_{\La}AR_{\La})^{-1}\nm \leq e^{L^{\frac{9}{10}}}.\]
	Assume that for any family $\mcF$ of pairwise disjoint elementary regions of size  $M=\lfloor N^{\xi}\rfloor$ contained in $\tilde{\La}_0,$  
	\[\#\{\La\in \mcF:\ \La \ {\rm is \ not \ in \ class} \ G\}\leq \frac{N^{\zeta}}{N^{\xi}}.\]
	Then for large $N$ (depending on $C_1,\zeta,\xi,c_1$),  we have 
	\begin{equation}\label{range2}
			|(R_{\tilde{\La}_0}AR_{\tilde{\La}_0})^{-1}(n,n')|\leq e^{-(c_2-N^{-\vartheta})|n-n'|}\ {\rm for} \ |n-n'|\geq \sqrt{N},
	\end{equation}
    where $\vartheta=\vartheta(\xi,\zeta)>0$.
\end{thm}
\begin{rmk}
	There is an important difference between the original form of  Theorem \ref{MSA1} in \cite{Liu22} and the present one. In \cite{Liu22},  \eqref{range1}--\eqref{range2} hold under the condition  of  $|n-n'|\geq\frac{N}{10}$,  $0<c_2\leq \frac{4}{5}c_1$ and $\zeta<1$. However, the relation $0<c_2\leq \frac{4}{5}c_1$ leads to the deterioration of decay rates in the Newton iteration when solving the P-equation. To address this issue, we impose a stronger restriction  $|n-n'|>\sqrt{N}$(this idea was first introduced in \cite{HSSY24}), which ensures that $0<c_2\leq c_1$. This refinement, however, requires the constant $0<\zeta<\frac{1}{2}$ instead of that in Theorem \ref{MSA1}. With these modifications, the proof of Theorem \ref{MSA1} in \cite{Liu22} remains valid, and we omit the details here.\end{rmk}

We also need

\begin{lem}[cf. Lemma 3.2, \cite{JLS20}]\label{MSA L2}
	Let $M_0\geq (\log N)^{\frac{10}{9}},\ \frac{c_1}{2}<c_2\leq c_1$ and $M_1\leq N$. Suppose that  $\La\subset\Z^d$ is connected and $\diam(\La)\leq 2N+1$. Suppose that for any $n\in\La,$ there exists some 
	$W=W(n)\in \mcE_M$ with $M_0\leq M\leq M_1$ such that $n\in W\subset \La,\dist(n,\La\setminus W)\geq \frac{M}{2}$ and 
	\[\nm (R_{W} A R_{W})^{-1} \nm\leq e^{M^{\frac{9}{10}}},\]
	\[	|(R_{W}AR_{W})^{-1}(n,n')|\leq e^{-c_2|n-n'|},\ {\rm for} \ |n-n'|\geq \sqrt{M}.\]
    Assume further that $N$ is large enough such that 
	\[\sup_{M_0\leq M\leq M_1} e^{M^{\frac{9}{10}}}(2M+1)^{d+C_1}e^{c_2\sqrt{N}}\sum_{j=0}^{\infty}(M+2j+1)^d e^{-c_2(j+M/2)}\leq \frac{1}{2}.\]
	Then we  have 
	\[\nm (R_{\La} A R_{\La})^{-1} \nm\leq 4(2M_1+1)^d e^{M_1^{\frac{9}{10}}}.\]
\end{lem}

\section*{Acknowledgement}
The authors would like to  thank Hongyi Cao for discussions  about  spectral properties of quasi-periodic Schr\"odinger operators with Lipschitz monotone potentials. Y. Shi is supported by the NSFC
(12271380).   Z. Zhang is  supported by  the National Key R\&D Program
of China under Grant 2023YFA1008801 and  the NSFC (12288101).

\section*{Data Availability}
		The manuscript has no associated data.
		\section*{Declarations}
		{\bf Conflicts of interest} \ The authors  state  that there is no conflict of interest.

\bibliographystyle{alpha}

\end{document}